\documentclass[11pt,reqno]{amsart}
\usepackage{amsmath, amssymb, latexsym, enumerate,  graphicx, MnSymbol}
\usepackage[all]{xy}
 \usepackage{pdfsync}

\usepackage{graphicx}
\usepackage[vcentermath]{youngtab}
\keywords{$3j$-symbols, quantum groups, Jones polynomial, complete intersections, Euler characteristic, Harish-Chandra bimodules, category $\cO$, Kazhdan-Lusztig polynomials, 3-manifold invariants, categorification}
 \newlength{\baseunit}               
 \newcount{\numlines}                
\newtheorem{theorem}{Theorem}
\newtheorem{thm}[theorem]{Theorem}
\newtheorem{lemma}[theorem]{Lemma}
\newtheorem{remark}[theorem]{Remark}
\newtheorem{prop}[theorem]{Proposition}
\newtheorem{corollary}[theorem]{Corollary}
\newtheorem{ex}[theorem]{Example}

\newtheorem{definition}[theorem]{Definition}
\newtheorem{define}[theorem]{Definition}

\newcommand{\cA}{{\mathcal A}}
\newcommand{\cB}{{\mathcal B}}
\newcommand{\cC}{{\mathcal C}}

\newcommand{\cE}{{\mathcal E}}
\newcommand{\cF}{{\mathcal F}}

\newcommand{\cH}{{\mathcal H}}

\newcommand{\cL}{{\mathcal L}}

\newcommand{\cO}{{\mathcal O}}

\newcommand{\cS}{{\mathcal S}}

\newcommand{\cU}{{\mathcal U}}

\newcommand{\Ddf}{D^\triangledown}
\def\down{\vee}
\def\up{\wedge}

\newcommand{\mg}{\mathfrak{g}}
\newcommand{\mh}{\mathfrak{h}}
\newcommand{\mb}{\mathfrak{b}}

\newcommand{\mN}{\mathbb{N}}

\newcommand{\mC}{\mathbb{C}}

\newcommand{\mZ}{\mathbb{Z}}

\newcommand{\sgn}{\operatorname{sgn}}

\newcommand{\la}{\lambda}

\newcommand{\HOM}{\operatorname{Hom}}

\newcommand{\MOD}{\operatorname{mod}}

\newcommand{\gMOD}{\operatorname{gmod}}

\newcommand{\op}{\operatorname}

\newcommand{\DIM}{\operatorname{dim}}
\newcommand{\GK}{\operatorname{GKdim}}

\newcommand{\surj}{\mbox{$\rightarrow\!\!\!\!\!\rightarrow$}}

\DeclareMathOperator{\Hom}{Hom}   
\DeclareMathOperator{\End}{End}   \DeclareMathOperator{\Id}{Id}
   \DeclareMathOperator{\Mod}{mod}
 \DeclareMathOperator{\Ext}{Ext}  
\DeclareMathOperator{\gmod}{gmod}

\begin{document}
\title[Jones-Wenzl projector and $3j$-symbols]{Categorifying fractional Euler characteristics, Jones-Wenzl projector and $3j$-symbols}
\author{Igor Frenkel}
\author{Catharina Stroppel}
\author{Joshua Sussan}
\date{\today}

\begin{abstract}
We study the representation theory of the smallest quantum group and its categorification. The first part of the paper contains an easy visualization of the $3j$-symbols in terms of weighted signed line arrangements in a fixed triangle and new binomial expressions for the $3j$-symbols. All these formulas are realized as graded Euler characteristics. The $3j$-symbols appear as new generalizations of  Kazhdan-Lusztig polynomials.

A crucial result of the paper is that complete intersection rings can be employed to obtain rational Euler characteristics, hence to categorify {\it rational} quantum numbers. This is the main tool for our categorification of the Jones-Wenzl projector, $\Theta$-networks and tetrahedron networks. Networks and their evaluations play an important role in the Turaev-Viro construction of $3$-manifold invariants. We categorify these evaluations by Ext-algebras of certain simple Harish-Chandra bimodules. The relevance of this construction to categorified colored Jones invariants and invariants of 3-manifolds will be studied in detail in subsequent papers.
\end{abstract}

\maketitle \tableofcontents

\section*{Introduction}
In the present paper we develop further the categorification program of the representation theory of the simplest quantum group $\mathcal{U}_q(\mathfrak{sl}_2)$ initiated in \cite{BFK} and continued in \cite{StDuke} and \cite{FKS}. In the first two papers the authors obtained a categorification of the tensor power $V_1^{\otimes n}$ of the natural two-dimensional representation $V_1$ of $\mathcal{U}_q(\mathfrak{sl}_2)$ using the category $\cO$ for the Lie algebra $\mathfrak{gl}_n$. In the third paper, this categorification has been extended to arbitrary finite dimensional representations of $\mathcal{U}_q(\mathfrak{sl}_2)$ of the form
\begin{eqnarray*}
V_{\bf d}=V_{d_1}\otimes V_{d_2}\otimes\cdots \otimes V_{d_r},
\end{eqnarray*}
where $V_{d_i}$ denotes the unique (type I) irreducible representation of dimension $d_i+1$. In this case the construction was based on (a graded version) of a category of Harish-Chandra bimodules for the Lie algebra $\mathfrak{gl}_n$, where $n=\sum_{i=1}^r d_i$, or equivalently by a certain subcategory of $\mathcal{O}$. The passage to a graded version of these categories is needed to be able to incorporate the quantum $q$ as a grading shift into the categorification. The existence of such a graded version is non-trivial and
requires geometric tools \cite{Sperv}, \cite{BGS}. Algebraically this grading can best be defined using a version of Soergel (bi)modules (\cite{Sperv}, \cite{Strgrad}). In \cite{Lauda1}, \cite{Lauda2} Lusztig's version of the quantum group itself was categorified. In the following we focus on the categorification of representations and intertwiners of  $\mathcal{U}_q(\mathfrak{sl}_2)$-modules.

There are two intertwining operators that relate the tensor power $V_1^{\otimes n}$ with the irreducible representation $V_n$, namely the projection operator
$$\pi_n:\quad V_1^{\otimes n}\longrightarrow V_n$$
and the inclusion operator
$$\iota_n:\quad V_n\longrightarrow V_1^{\otimes n}.$$
Their composition is known as the {\it Jones Wenzl-projector} which can be characterized by being an idempotent and a simple condition involving the cup and cap intertwiners (Theorem \ref{charJW}). The Jones-Wenzl projector plays an important role in the graphical calculus of the representation theory of $\mathcal{U}_q(\mathfrak{sl}_2)$, see \cite{FK}. Even more importantly, it is one of the basic ingredients in the categorification of the colored Jones polynomial and, in case of $q$ a root of unity, the Turaev-Viro and Reshetikhin-Turaev invariants of $3$-manifolds.

One of the basic results of the present paper is the categorification of the Jones-Wenzl projector including a characterization theorem.
This provides a crucial tool for the categorification of the complete colored Jones invariant for quantum $\mathfrak{sl}_2$, \cite{SS}.
The fundamental difficulty here is the problem of categorifying rational numbers that are intrinsically present in the definition of the Jones-Wenzl projector. We show that the rational numbers that appear in our setting admit a natural realization as the graded Euler characteristic of the Ext-algebra $\Ext^*_H(\mC,\mC)$ of the trivial module over a certain complete intersection ring $H$. The standard examples appearing here are cohomology rings of Grassmannian $\op{Gr}(k,n)$ of $k$-planes in $\mC^n$. The graded Euler characteristic of $\Ext^*_H(\mC,\mC)$ computed via an infinite minimal projective resolution of $\mC$ yields an infinite series in $q$ which converges (or equals in the ring of rational functions) to quantum rational numbers, e.g. the inverse binomial coefficient  $\frac{1}{\left[n\brack k\right]}$ for the Grassmannian.

In order to apply our categorification of the rational numbers
appearing in the Jones-Wenzl projector we further develop the previous
results of \cite{BFK}, \cite{StDuke} and \cite{FKS} on the categorification of $V_1^{\otimes n}$ and $V_{\bf d}$ and the relations between them.

Among other things, we give an interpretation of different bases in $V_{\bf d}$, namely Lusztig's (twisted) canonical, standard, dual standard and (shifted) dual canonical basis  in terms of indecomposable projectives, standard, proper standard and simple objects. This is a non-trivial refinement of \cite{FKS} since we work here not with highest weight modules but with the so-called {\it properly stratified} structures. This allows categories of modules over algebras which are still finite dimensional, but
of infinite global dimension which produces precisely the fractional Euler characteristics. In particular we prove that the endomorphism rings of our standard objects that categorify the standard basis in $V_{\bf d}$ are tensor products of cohomology rings of Grassmannians. As a consequence we determine first standard resolutions and then projective resolutions of proper standard objects.

Using the categorification of the Jones-Wenzl projector and inclusion operators we proceed then to a categorification of various networks from the graphical calculus of  $\cU_q(\mathfrak{sl}_2)$. The fundamental example of the triangle diagram leads to the {\it Clebsch-Gordan coefficients} or {\it $3j$-symbols} for $V_k$ in  $V_i\otimes V_j.$ Remarkably its evaluation in the standard basis of $V_i, V_j, V_k$ yields an integer (possibly negative) which we show can be obtained by counting signed isotopy classes of non-intersecting arc or line arrangements in a triangle with $i, j, k$ marked points on the edges. We introduce the notion of {\it weighted signed line arrangements}
to be able to keep track of the formal parameter $q$ in a handy combinatorial way.

We then present various categorifications of the $3j$-symbols which decategorify to give various identities for the $3j$-symbols. Our first categorification involves a double sum of Exts of certain modules. The second realizes the $3j$-symbol in terms of a complex of graded vector spaces with a distinguished basis. The basis elements are in canonical bijection with weighted signed line arrangements. The weight is the internal grading, whereas the homological degree taken modulo $2$ is the sign of the arrangement.
So the Euler characteristic of the complex is the $3j$-symbol and categorifies precisely the terms in our triangle counting formula. Thus the quite elementary combinatorics computes rather non-trivial Ext groups in a category of Harish-Chandra bimodules.

As an alternative we consider the $3j$-symbols evaluated in the (twisted) canonical basis and realize it as a graded dimension of a certain vector space of homomorphisms. In particular, in this basis the  $3j$-symbols are (up to a possible overall minus sign) in $\mathbb{N}[q]$. This finally leads to another binomial expression of the ordinary $3j$-symbols which seems to be new. Besides it interesting categorical meaning, this might also have an interpretation in terms of hypergeometric series. Our new positivity formulas resemble formulas obtained by Gelfand and Zelevinsky, \cite{GZ}, in the context of Gelfand-Tsetlin bases.

 Apart from the $3j$-symbol we also give a categorical interpretation of the colored unknot and the theta-networks as the Euler characteristic of $\Ext^*(L,L)$ for some irreducible Harish-Chandra bimodule $L$. These examples can be viewed as special cases of a complete categorification of the colored Jones polynomial. Additionally, we explain the construction of a categorification of a tetrahedron network, namely as an (infinite) dimensional graded vector space which in fact is a module over the four theta networks for the four faces. This approach will be explored in detail in two subsequent papers.

Lie theoretically, the categorification results are interesting, since it is very hard to compute the above mentioned Ext-algebras of simple modules in practice. For instance, the well-known evaluations of theta-networks from
\cite{KaLi} can be used to obtain the Euler characteristics of certain Ext
algebras, allowing one to gain some insight into their structure. In this context, $3j$-symbols can be viewed as generalizations of Kazhdan-Lusztig polynomials.

Finally we want to remark that the idea of categorifying the Jones-Wenzl projectors is not new. The first categorification was obtained in \cite{BFK}, but not studied in detail there. Independently, to the present work, such categorifications were defined and studied in \cite{CK}, \cite{Roszansky}. Based on two theorems characterizing the categorifications of the Jones-Wenzl projector, Theorem \ref{charJWcat} and \cite[Section 3]{CK}, we expect that these categorifications agree (up to Koszul duality). A detailed proof in the case $n=2$ is given in \cite{SScomparision}. On the other hand, our approach gives rise to a different (and actually much more subtle) colored Jones invariant than the original construction of Khovanov \cite{Khovcolor} and the one in \cite{BW}. The difference should become clear from our computations for the colored unknot. Although the value of the $V_n$-colored unknot is a quantum number the categorification is not given by a finite complex, but rather by an infinite complex with cohomology in all degrees.

We also expect that our invariants agree (up to some Koszul duality) with the more general invariants constructed in \cite{W}, but the details are not worked out yet. Our explicit description of the endomorphism rings of standard objects and the description of the value for the colored unknot will be necessary to connect the two theories.

\subsection*{Organization of the paper}
The paper is divided into three parts. The first part starts with recalling basics from the representation theory of $\cU_q(\mathfrak{sl}_2)$ and defines the $3j$-symbols. We develop the combinatorics of weighted signed line arrangements. Part II starts by the construction of rational Euler characteristics which we believe are interesting on its own. We explain then in detail where these categorifications occur naturally. This leads to a categorification of the Jones-Wenzl projector. This part also reviews a few previously constructed categorifications. The third part consists of applications and several new results.
It contains the categorification of the $3j$-symbols and explains categorically all the formulas obtained in the first part. We realize $3j$-symbols as a sort of generalized Kazhdan-Lusztig polynomials. Finally we explain a categorification of  $\theta$-networks and $6j$-symbols with some outlook to work in progress.

\subsection*{Acknowledgements}
We thank Christian Blanchet, Ragnar Buchweitz, Mikhail Khovanov and Nicolai Reshetikhin for interesting discussions and Geordie Williamson for very helpful comments on an earlier version of the paper.
The research of the first and third author was supported by the NSF grants DMS-0457444 and DMS-0354321. The second author deeply acknowledges the support in form of a Simons Visitor Professorship to attend the Knot homology Program at the MSRI in Berkeley which was crucial for writing up this paper.

\section*{Part I}
\label{reptheory}
In this part we recall basic structures on the representation theory of $\mathcal{U}_q(\mathfrak{sl}_2)$ from \cite{Kassel} and \cite{KaLi} which later will be categorified.

\section{Representation theory of $ \mathcal{U}_q(\mathfrak{sl}_2) $}
\label{U2basics}
Let $\mC(q)$ be the field of rational functions in an indeterminate $q$.
\begin{define}
Let $\mathcal{U}_q=\mathcal{U}_q(\mathfrak{sl}_2) $ be the associative algebra over $\mathbb{C}(q)$ generated by $ E, F, K, K^{-1} $ satisfying the
relations:
\begin{enumerate}[(i)]
\item $ KK^{-1} = K^{-1}K=1$ \item $ KE = q^2 EK $ \item $ KF = q^{-2} FK $ \item $ EF-FE = \frac{K-K^{-1}}{q-q^{-1}} $
\end{enumerate}
\end{define}

Let $ [k]=\sum_{j=0}^{k-1} q^{k-2j-1}$ and $ {n \brack {k}} = \frac{[n]!}{[k]![n-k]!}. $ Let $ \bar{V}_n $ be the unique (up to isomorphism) irreducible module for $\mathfrak{sl}_2$ of dimension $ n+1$. Denote by $V_n $ its quantum analogue (of type I), that is
the irreducible $ \mathcal{U}_q(\mathfrak{sl}_2)$-module with basis $ \lbrace v_0, v_1, \ldots, v_{n} \rbrace $ such that
\begin{equation}
\label{irreddef}
K^{\pm 1} v_i= q^{\pm (2i-n)} v_i\quad\quad Ev_i =[i+1] v_{i+1}\quad\quad F v_i = [n-i+1] v_{i-1}.
\end{equation}

There is a unique bilinear form $\langle\quad,\quad\rangle' \;\colon\; V_n \times V_n \rightarrow \mathbb{C}(q)$
which satisfies
\begin{equation}
\label{scalarprod} \langle v_k, v_l \rangle' = \delta_{k,l}q^{k(n-k)}{n \brack k}.
\end{equation}
The vectors $\lbrace v^0, \ldots, v^n \rbrace $ where  $v^i =\frac{1}{{n \brack i}} v_i$ form {\it the dual standard basis} characterized by $ \langle v_i, v^i \rangle' = q^{i(n-i)}$. Recall that $\mathcal{U}_q$ is a Hopf algebra with comultiplication
\begin{equation}
\label{comult}
\begin{split}
\triangle(E)=1\otimes E + E\otimes K^{-1},\quad&\quad\quad\triangle(F)=K\otimes F+F\otimes 1,\\
\triangle(K^{\mp 1})&=K^{\mp1}\otimes K^{\mp1},
\end{split}
\end{equation}
and antipode $S$ defined as $S(K)=K^{-1}$, $S(E)=-EK$ and $S(F)=-K^{-1}F$. Therefore, the tensor product $V_{\bf d}:=V_{d_1} \otimes \cdots\otimes V_{d_r} $ has the structure of a $\mathcal{U}_q$-module with {\it standard basis} $\lbrace v_{\bf a}=v_{a_1} \otimes \cdots \otimes v_{a_r} \rbrace $ where $ 0 \leq a_j \leq d_j $ for $ 1 \leq j \leq r$. Denote by $v^{\bf a}=v^{a_1} \otimes \cdots \otimes v^{a_r}$ the corresponding tensor products of dual standard basis elements.

There is also a unique {\it semi-linear} form (i.e. anti-linear in the first and linear in the second  variable) $\langle_-,_-\rangle$  on $V_{\bf d}$ such that

\begin{eqnarray}
\label{scalar}
\langle v_{\bf a}, v^{\bf b}\rangle&=&
\begin{cases}
\prod_{i=1}^rq^{a_i(d_i-a_i)}&\text{if ${\bf a}={\bf b}$}\\
0&\text{otherwise}.
\end{cases}
\end{eqnarray}
We want to call this form {\it evaluation form} since it will later be used to evaluate networks.

We finally have the {\it pairing} $(_-,_-)$, anti-linear in both variables, on $V_{\bf d}$ such that
\begin{eqnarray}
\label{pairing}
(v_{\bf a}, v^{\bf b})&=&
\begin{cases}
q^{a_i(d_i-a_i)}&\text{if ${\bf a}={\bf b}$}\\
0&\text{otherwise}.
\end{cases}
\end{eqnarray}

\subsection{Jones-Wenzl projector and intertwiners}
Next we will define morphisms between various tensor powers of $V_1$ which intertwine the action of the quantum group, namely $\cup\colon \mathbb{C}(q) \rightarrow V_1^{\otimes 2}$ and $ \cap \colon V_1^{\otimes 2} \rightarrow \mathbb{C}(q)$ which are given on the standard basis by
\begin{eqnarray}
\label{defcupcap}
\cup(1)=v_1\otimes v_0 -qv_0 \otimes v_1
&&
 \cap(v_i \otimes v_j)=
 \begin{cases}
 0 &\text{if $i=j$}\\
 1 &\text{if $i=0, j=1$}\\
 -q^{-1}&\text{if $i=1, j=0$}
 \end{cases}
\end{eqnarray}

We define $\cap_{i,n}=\Id^{\otimes (i-1)} \otimes \cap \otimes \Id^{\otimes
    (n-i-1)}$ and $\cup_{i,n}=\Id^{\otimes (i-1)} \otimes \cup \otimes \Id^{\otimes (n-i+1)}$  as $\cU_q$-morphisms from $V_1^{\otimes n}$ to $V_1^{\otimes (n-2)}$, respectively  $V_1^{\otimes (n+2)}$.
Let $C:=\cup\circ\cap$ be their composition and $C_i:=C_{i,n}:= \cup_{i,n-2}\circ\cap_{i,n}$.
We depict the cap and cup intertwiners graphically in Figure \ref{fig:capcup} (reading the diagram from bottom to top), so that $\cap\circ\cup$ is just a circle. In fact, finite compositions of these elementary morphisms generate the $\mathbb{C}(q)$-vector space of all intertwiners, see e.g. \cite[Section 2]{FK}).

If we encode a basis vector $v_{\bf d}$ of $V_1^{\otimes n}$ as a sequence of $\up$'s and $\down$'s according to
the entries of ${\bf d}$, where $0$ is turned into $\down$ and $1$ is turned into $\up$ then the formulas in \eqref{defcupcap} can be symbolized by\\
\vspace{2mm}\\
\begin{picture}(45,28)
\put(9.2,0){$\scriptstyle\up$}
\put(32.2,0){$\scriptstyle\up$}
\put(35,14){\line(0,-1){11.5}}
\put(12,14){\line(0,-1){11.5}}
\put(23.5,14){\oval(23,23)[t]}
\put(41,12){$=0=$}\put(101,12){$,$}
\put(69.2,2){$\scriptstyle\down$}
\put(92.2,2){$\scriptstyle\down$}
\put(95,14){\line(0,-1){11.5}}
\put(72,14){\line(0,-1){11.5}}
\put(83.5,14){\oval(23,23)[t]}
\put(109.2,2){$\scriptstyle\down$}
\put(132.2,0){$\scriptstyle\up$}
\put(135,14){\line(0,-1){11.5}}
\put(112,14){\line(0,-1){11.5}}
\put(123.5,14){\oval(23,23)[t]}
\put(141,12){$=1$,\quad}

\put(169.2,0){$\scriptstyle\up$}
\put(192.2,2){$\scriptstyle\down$}
\put(195,14){\line(0,-1){11.5}}
\put(172,14){\line(0,-1){11.5}}
\put(183.5,14){\oval(23,23)[t]}
\put(201,12){$=-q^{-1}$,\quad}

\put(243.5,14){$\bigcup(1)=$}
\put(279.2,22){$\scriptstyle\up$}
\put(292.2,25){$\scriptstyle\down$}
\put(295,6){\line(0,1){19.5}}
\put(282,6){\line(0,1){19.5}}
\put(301,12){$-\;q$}
\put(319.2,25){$\scriptstyle\down$}
\put(332.2,22){$\scriptstyle\up$}
\put(335,6){\line(0,1){19.5}}
\put(322,6){\line(0,1){19.5}}
\end{picture}
\vspace{4mm}

\begin{figure}
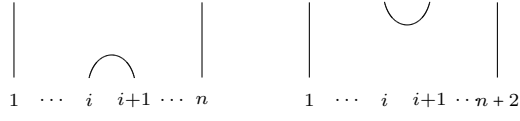

\begin{equation*}
 \xy (0,0);(0,10)*{}**\dir{-};(0,10); (10,0)*{};(16,0)*{}**\crv{(11,4)&(15,4)}; (25,0)*{};(25,10)*{}**\dir{-};(25,10)*{};
(0,-3)*{\mbox{\tiny 1}};(5,-3)*{\mbox{\tiny $\ldots$}};(10,-3)*{\mbox{\tiny $i$}};(16,-3)*{\mbox{\tiny
$i$+$1$}};(21,-3)*{\mbox{\tiny$\ldots$}};(25,-3)*{\mbox{\tiny$n$}};
\endxy
\hspace{0.5in} \xy (0,0)*{};(0,10)*{}**\dir{-};(0,10)*{}; (10,10)*{};(16,10)*{}**\crv{(11,6)&(15,6)};(10,10)*{};(16,10)*{};
(25,0)*{};(25,10)*{}**\dir{-};(25,10)*{}; (0,-3)*{\mbox{\tiny 1}};(5,-3)*{\mbox{\tiny $\ldots$}};(10,-3)*{\mbox{\tiny
$i$}};(16,-3)*{\mbox{\tiny $i$+$1$}};(21,-3)*{\mbox{\tiny$\ldots$}};(25,-3)*{\mbox{\tiny$n+2$}};
\endxy
\end{equation*}
\caption{The intertwiners $\cap_{i,n}$ and $\cup_{i,n}$}
\label{fig:capcup}
\end{figure}

The symmetric group $\mathbb{S}_n$ acts transitively on the set of $n$-tuples with $i$ ones and $n-i$ zeroes. The stabilizer of ${\bf d}_{\op{dom}}:=(\underbrace{1, \ldots, 1}_{i},
\underbrace{0, \ldots, 0}_{n-i})$ is $\mathbb{S}_{i}\times \mathbb{S}_{n-i}$. By sending the identity element $e$ to ${\bf d}_{\op{dom}}$ we fix for the rest of the paper a bijection between shortest coset representatives in $\mathbb{S}_n/\mathbb{S}_{i}\times \mathbb{S}_{n-i}$ and these tuples ${\bf d}$. We denote by $w_0$ the longest element of $\mathbb{S}_n$ and by $w_0^i$ the longest element in $\mathbb{S}_{i}\times \mathbb{S}_{n-i}$. By abuse of language we denote by $l(\bf{d})$ the (Coxeter) {\it length} of $\bf{d}$ meaning the Coxeter length of the corresponding element in $\mathbb{S}_n$. We denote by $|{\bf d}|$ the numbers of ones in ${\bf d}$.

\begin{define}{\rm
For $ {\bf a}=(a_1, \ldots, a_n)\in\{0,1\}^n$ let $ v_{\bf a}=v_{a_1} \otimes \cdots \otimes v_{a_n}\in V_1^{\otimes n}$ be the
corresponding basis vector.
\begin{itemize}
\item Let $ \pi_n \colon
V_1^{\otimes n} \rightarrow V_n $ be given by the formula
\begin{eqnarray}
\label{pn} \pi_n(v_{\bf a}) = q^{-l({\bf a})}\frac{1}{{n \brack |{\bf a}|}} v_{|{\bf a}|}=q^{-l({\bf a})}v^{|{\bf a}|}
\end{eqnarray}
where $l({\bf a})$ is equal to the number of pairs $(i,j)$ with $i<j$ and $a_i<a_j.$ This gives the {\it projection} $ \pi_{i_1} \otimes
\cdots \otimes \pi_{i_r} \colon V_1^{\otimes(i_1+\cdots+i_r)} \rightarrow V_{i_1} \otimes \cdots \otimes V_{i_r}. $
\item We denote by $ \iota_n \colon V_n \rightarrow V_1^{\otimes n} $ the intertwining map
\begin{eqnarray}
\label{in}
 v_k \mapsto \sum_{|{\bf a}|=k} q^{b({\bf a})} v_{\bf a}
\end{eqnarray}
where $b({\bf a})=|{\bf a}|(n-|{\bf a}|)-l({\bf a})$, i.e. the number of pairs $ (i,j) $ with $ i < j $ and $ a_i > a_j$. Define the {\it inclusion} $ \iota_1 \otimes \cdots \otimes
\iota_r \colon V_{i_1} \otimes \cdots \otimes V_{i_r} \rightarrow V_1^{\otimes (i_1+\cdots+i_r)}.
$
\end{itemize}
}
\end{define}
The composite $p_n=\iota_n\circ\pi_n$ is the {\it Jones-Wenzl projector}. We symbolize the projection and inclusion by
\begin{picture}(50,10)
\put(0,0){\line(1,1){10}}
\put(0,0){\line(1,0){50}}
\put(50,0){\line(-1,1){10}}
\put(10,10){\line(1,0){30}}
\end{picture}
and
\begin{picture}(50,10)
\put(10,0){\line(-1,1){10}}
\put(10,0){\line(1,0){30}}
\put(40,0){\line(1,1){10}}
\put(0,10){\line(1,0){50}}
\end{picture}
respectively, and the idempotent $p_n$ by
\scalebox{0.5}{
\begin{picture}(50,20)
\put(0,0){\line(1,1){10}}
\put(0,0){\line(1,0){50}}
\put(50,0){\line(-1,1){10}}
\put(10,10){\line(1,0){30}}
\put(10,10){\line(-1,1){10}}
\put(40,10){\line(1,1){10}}
\put(0,20){\line(1,0){50}}
\end{picture}}
or just
\begin{picture}(40,10)
\put(0,0){\line(0,1){10}}
\put(0,0){\line(1,0){40}}
\put(40,0){\line(0,1){10}}
\put(0,10){\line(1,0){40}}
\end{picture}.
\begin{ex}
\label{JWex}
{\rm For $n=2$ we have $\iota_2(v_0)=v_0\otimes v_0$,
$\iota_2(v_1)=qv_1\otimes v_0+v_0\otimes v_1$, $\iota_2(v_2)=v_1\otimes v_1$,
and $\pi_2(v_0\otimes v_0)=v_0$, $\pi_2(v_1\otimes v_0)=v^1=[2]^{-1}v_1$,
$\pi_2(v_0\otimes v_1)=q^{-1}v^1=q^{-1}[2]^{-1}v_1$,  $\pi_2(v_1\otimes
v_1)=v_2$.
}
\end{ex}

\begin{remark}
\label{identifications}
{\rm
Note that our projector $\pi_n$  differs from the one, call it $\pi_n'$, in \cite{FK}. This is due to the fact that Frenkel and Khovanov work in the dual space, but also use a different comultiplication. The precise connection is explained in Remark \ref{annoying}.}
\end{remark}

\subsection{Direct Summands and weight spaces}
\label{decs}
Any finite dimensional $ \mathfrak{sl}_2-$, respectively $\mathcal{U}_q$-module decomposes into weight spaces. For the irreducible modules we have seen this already in \eqref{irreddef}. Most important for us will be the decomposition of ${V}_1^{\otimes n}=\bigoplus_{i=0}^n ({{V}_1^{\otimes n}})_{i}$, where the index $i$ labels the $q^{(2i-n)}$-weight space spanned by all vectors $v_{\bf{d}}$ with $|{\bf d}|=i$. Arbitrary tensor products have an analogous weight space decomposition, for instance inherited when viewed as a submodule via the inclusion \eqref{in}. For any positive integer $n$ there is a decomposition of $ \mathfrak{sl}_2$-modules, respectively $\mathcal{U}_q(\mathfrak{sl}_2)$-modules

\begin{eqnarray}
\label{isotypic}
\bar{V}_1^{\otimes n}\cong\bigoplus_{r=0}^{\lfloor n/2 \rfloor} \bar{V}_{n-2r}^{\oplus b_{n-2r}}&\text{respectively}&{V}_1^{\otimes n}\cong\bigoplus_{r=0}^{\lfloor n/2 \rfloor} {V}_{n-2r}^{\oplus b_{n-2r}}
\end{eqnarray}
where $b_{n-2r} = \dim\Hom_{\mathfrak{sl}_2}(\bar{V}_1^{\otimes n}, \bar{V}_{n-2r})$.
This decomposition is not unique, but the decomposition into isotypic components (that means the direct sum of all isomorphic simple modules) is unique. In particular, the summand $V_n$ is unique. The Jones-Wenzl projector $p_n=\iota_n \circ \pi_n $ is precisely the projection onto this summand. It has the following alternate definition (see e.g. \cite[Theorem
3.4]{FK}).

\begin{prop}
\label{charJW}
The endomorphism $p_n$ of $V_1^{\otimes n}$ is the unique $\mathcal{U}_q(\mathfrak{sl}_2)$-morphism which satisfies {\rm (for $1 \leq i \leq n-1 $):}
\begin{enumerate}[(i)]
\item $ p_n \circ p_n = p_n $ \item $C_{i,n} \circ p_n = 0 $ \item $ p_n \circ C_{i,n} = 0$.
\end{enumerate}
\end{prop}

\subsection{Networks and their evaluations}
In Figure \ref{networks} we display networks and their evaluations. The first one represents an intertwiner $f: V_2\otimes V_3\rightarrow V_3$ (read from bottom to top). The second one the evaluation $f(v_0\otimes v_2)$. The third the evaluation $\langle  v_1,f(v_0\otimes v_2)\rangle$, that is the coefficient of $v^1$ when $f(v_0\otimes v_2)$ is expressed in the dual standard basis. The last one represents the {\it value of the unknot colored by $n$} (where the strand should be colored by $n$ or be cabled (which means be replaced by $n$ single strands), viewed either as an intertwiner $g: V_0\rightarrow V_0$ evaluated at $v_0$ or as evaluation $\langle v_0, g(1)\rangle$. It is easy to see that this value is $(-1)^n[n+1]$.

\begin{figure}
\label{networks}
\includegraphics{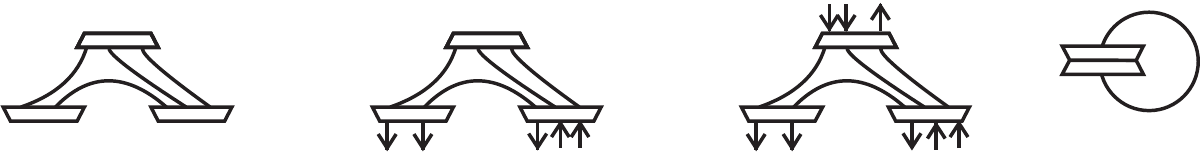}
\caption{Networks and their evaluations}
\end{figure}

\section{$3j$-symbols and weighted signed line arrangements in triangles}
\label{comb}
In this section we determine (quantum) $3j$-symbols by counting line arrangements of triangles. This provides a graphical interpretation of the q-analogue of the Van-der-Waerden formula for quantum $3j$-symbols, \cite{Kirillov}, \cite{KR}. Later, in Theorem \ref{trianglecat}, this combinatorics will be used to describe the terms of a certain resolution naturally appearing in a Lie theoretic categorification of the $3j$-symbols.

\subsection{Triangle arrangements}
 We start with the combinatorial data which provides a reformulation of the original definition of $3j$-symbols (following \cite{CFS}, \cite[Section 3.6]{FK}, see also \cite[VII.7]{Kassel}).
Let $i,j,k$ be non-negative integers. We say {\it  $i,j,k$ satisfy the triangle identities} or {\it are admissible} if $k\equiv i+j\mod 2$ and
\begin{equation}
\label{triangle} i+k\geq j,\quad j+k\geq i\quad i+j\geq k.
\end{equation}
To explain the above notion we consider a triangle with $ i, j, $ and $ k $ marked points on each side and call it an {\it
$i,j,k$-triangle}. (In the following the precise location of the marked points will play no role). We refer to the three sides as the $i$, $j$ and $k$-side respectively. A {\it line arrangement} for this triangle is an isotopy class of collections of non-intersecting arcs
inside the triangle, with boundary points the marked points such that every marked point is the boundary point of precisely one arc, and
the end points of an arc lie on different sides; see Figure \ref{fig:arcs} for an example. Given such a line arrangement $L$ we denote by
$z(L)$ (resp. $x(L)$) the number of arcs connecting points from the $i$-side with points from the $j$-side (resp. $k$-side), and by $y(L)$
the number of arcs connecting points from the $j$-side with points from the $k$-side.

\begin{figure}[htb]
\includegraphics{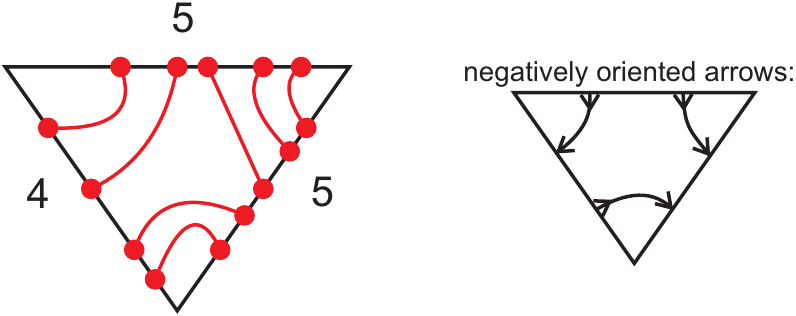}
\caption{A $(4,5,5)$-triangle with a line arrangement $L$ such that $x(L)=2=z(L)$, $y(L)=3$.}
\label{fig:arcs}
\end{figure}

\begin{lemma}
Suppose we are given an $i,j,k$-triangle $\triangle$. Then there is a line arrangement for $\triangle$ if and only if the triangle
equalities hold. Moreover, in this case the line arrangement is unique up to isotopy.\end{lemma}

\begin{proof}
Assume there is a line arrangement $L$. Then $i+j=x(L)+z(L)+z(L)+y(L)\geq x(L)+y(L)=k$, and $i,j,k$ satisfy the triangle identities.
Conversely, if the triangle identities hold, then there is obviously an arrangement $L$ with $z=z(L)$, $x=x(L)$, $y=y(L)$ as follows
\begin{equation}
\label{xyz}
z=\frac{i+j-k}{2},\quad x=\frac{i+k-j}{2},\quad y=\frac{j+k-i}{2}.
\end{equation}
Assume there is another line arrangement $L'$. So $x(L')=x(L)+a$ for some integer $a$, and
then $y(L')=y(L)-a$ and $z(L')=z(L)-a$. On the other hand $2x(L')+2y(L')+2z(L')=i+j+k=2x(L)+2y(L)+2z(L)$, hence $a=0$. The uniqueness
follows, since the values of $x,y,z$ uniquely determine a line arrangement up to isotopy.
\end{proof}

An {\it oriented line arrangement} is a line arrangement with all arcs oriented. We
call an arc {\it negatively oriented} if it points either from the $i$ to the $j$ side, or from the $k$ to either the $i$ or $j$ side, and {\it positively oriented} otherwise; see Figure \ref{fig:example3j} for examples and Figure \ref{fig:arcs} for an illustration. To each oriented
line arrangement we assign the sign $$\sgn{L}=(-1)^\text{ number of negatively oriented arcs}.$$ Given an oriented line arrangement $L$ of
an $i,j,k$-triangle, we denote by $r=r(L)$ resp. $s=s(L)$ the number of arcs with end point at the $i$-respectively $j$-side and the orientation points into the interior of the triangle. Denote by
$t=t(L)$ the number of arcs with endpoint at the $k$-side and oriented with the orientation pointing outside the triangle. We call $L$ an {\it $(i,j,k,r,s,t)$-arrangement}. The sum of the signs of
all the $(i,j,k,r,s,t)$-arrangements for a fixed $i,j,k$-triangle is denoted $\op{Arr}^{i,j,k}_{r,s,t}$, hence  $\op{Arr}^{i,j,k}_{r,s,t}=\sum_{L}\sgn{L}$.

\begin{figure}[htb]
\includegraphics{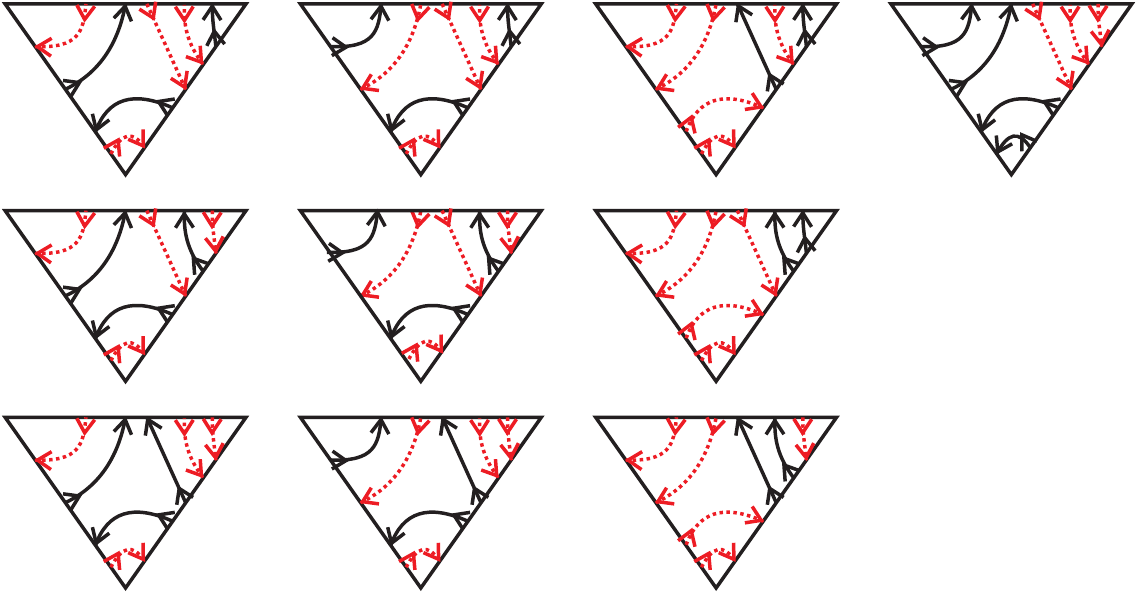}
\caption{Ten (out of $16$ possible) $(4,5,5,2,2,2)$-line arrangements; the negatively oriented arcs are in dotted light red, $\sgn(L)=1$
for the first two columns, $\sgn(L)=-1$ for the second two columns.} \label{fig:example3j}
\end{figure}

\nopagebreak
\subsection{$3j$-symbols}
Assume we are given $i,j,k$ satisfying the triangle identities. Denote by $ \Phi_{i,j}^k \colon\; V_1^{\otimes i+j} \rightarrow V_1^{\otimes k} $ the intertwiner given by the diagram
\begin{equation}
\label{intertwiner} \xy {(10,0)*{};(10,10)*{}**\dir{-}}; {(16,0)*{};(22,0)*{}**\crv{(17,4)&(21,4)}}; {(28,0)*{};(28,10)*{}**\dir{-}};
{(5,7)*{\frac{i+k-j}{2}}}; {(19,7)*{\frac{i+j-k}{2}}}; {(33,7)*{\frac{j+k-i}{2}}}; {(10,-4)*{}}
\endxy
\end{equation}
The number at a strand indicates that there are in fact that many copies of the strand, for instance $\frac{i+j-k}{2}\in\mZ_{\geq0}$
nested caps.
\begin{define}
\label{3j}
The $ \mathcal{U}_q$-intertwiner $ A_{i,j}^{k} \colon\; V_i \otimes V_j \rightarrow V_k $ is defined as
    $ \pi_k \circ \Phi_{i,j}^k \circ \iota_i \otimes \iota_j $.
The {\bf $3j$-symbol}
    $C_{i,j}^k(r,s,t) $ is defined to be $ \langle v_t,A_{i,j}^k (v_r \otimes v_s)\rangle,$  where $0\leq r\leq i$, $0\leq
    s\leq j$, $0\leq t\leq k$.
\end{define}

\begin{figure}[htb]
\includegraphics{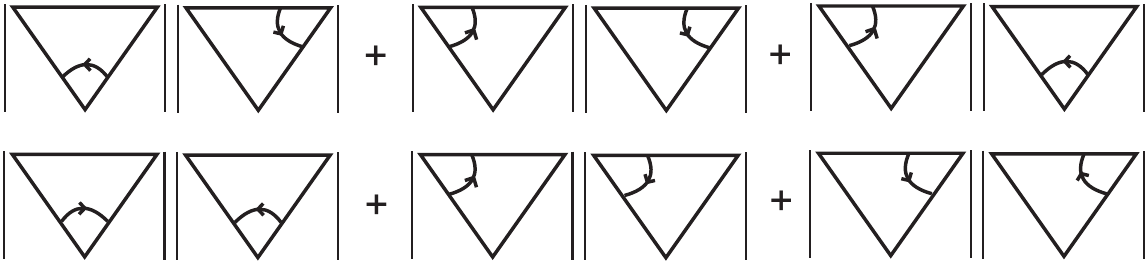}
\caption{The combinatorial formula for $\gamma_a$} for a fixed oriented triangle arrangement. For instance the first summand is the number of arcs oriented from the $j$- to the $i$-side times the number of arcs oriented from the $k$- to the $j$-side.
\label{fig:qcount}
\end{figure}

\begin{thm}
\label{ClebschGordon} Assume $i,j,k$ satisfy the triangle identity. With the notation from \eqref{xyz} we have
\begin{enumerate}[(i)]
\item The classical $3j$-symbol counts signed arrangements:
\begin{eqnarray*}
C_{i,j}^k(r,s,t)_{q=1}&=&\sum_{a=0}^z {(-1)}^{a}  {z \choose a} {x \choose r-a} {y \choose j-s-a}\\
&=&(-1)^{\frac{i+j+k}{2}+r+s}\op{Arr}^{i, j, k}_{r, s, t}.
\end{eqnarray*}
\item The quantum $3j$-symbol counts weighted signed arrangements: $$C_{i,j}^k(r,s,t)=q^{-t(k-t)}\sum_{a=0}^z (-1)^aq^{-a}q^{\gamma_a} {z \brack {a}} {x \brack {r-a}}{y \brack {j-s-a}},$$
where
$\gamma_a=\left((z-a)+(r-a)\right)(y-s+z-a)+(r-a)\left((x-r+a)+(z-a)\right)+a(z-a)+(j-s-a)(y-j+s+a)$. The latter number can be expressed in terms of $(i,j,k,r,s,t)$-triangle arrangements as displayed in Figure~\ref{fig:qcount}.
\end{enumerate}
\end{thm}

\begin{proof}
Consider first the case $q=1$. Let $v_r\in V_i$ and $v_s\in V_j$.
Then $\iota_r\otimes \iota_s(v_r\otimes v_s)$ is the sum over all basis vectors $v_{\bf d}\in V_1^{\otimes i+j}$, where ${\bf d}$ contains $r+s$ ones, $r$ of them amongst the first $i$ entries. To compute $v:=\Phi_{i,j}^k(v_{\bf d})$ consider the corresponding $\{\up, \down\}$-sequence for ${\bf d}$ and place it below the diagram \eqref{intertwiner}. If the caps are not getting oriented consistently then $v=0$, otherwise all caps and vertical strands inherit an orientation and $v=(-1)^{a}v_{{\bf d}'}$, where ${\bf d}'$ is the sequence obtained at the top of the strands, and $a$ denotes the number of clockwise oriented arcs arising from the $z$ caps. Applying the projection $\pi_n$ and evaluating at $v_t$ via the formula \eqref{scalar} gives zero if the number of $1$'s in ${\bf d}'$ is different from $t$ and evaluates to $(-1)^{a}$ otherwise. Hence we only have to count the number of possible orientations for \eqref{intertwiner} obtained by putting $\up$'s and $\down$'s at the end of the strands- with $r$ $\up$'s amongst the
first $i$ endpoints and $s$ $\up$'s amongst the last $j$ endpoints of strands at the bottom of the diagram, and $t$ $\up$'s amongst the
endpoints at the top. There are ${z \choose a}$ ways of arranging the $a$ clockwise arrows in the $z$-group.  Hence there must be $y-j+s+a$
lines pointing upwards in the $y$-group and $r-a$ pointing upwards in the $x$-group. There are $ {y \choose j-s-a} $, resp. $ {x \choose r-a} $ ways of arranging these arrows. Therefore,
$$C_{i,j}^k(r,s,t) =\sum_{a\geq 0} {(-1)}^{a} {z \choose a} {x \choose r-a} {y \choose j-s-a}.$$
Note that the same rules determine the number of signed line arrangements if we can make sure that the signs match. So, given a line
arrangement we compute its sign. There are $a$ arcs going from the $i$-side to the $j$-side, $x-r+a$ arcs from
the $k$-side to the $i$-side and $j-s-a$ from the $k$-side to the $j$-side. Using \eqref{xyz}, the sign is
$(-1)^{a+x+r+a+j+s+a}=(-1)^{\frac{i+j+k}{2}+r+s}(-1)^a$. Up to the constant factor  $(-1)^{\frac{i+j+k}{2}+r+s}$ this agrees with the sign appearing in the formula for the $3j$-symbols and the $q=1$ case follows.

Consider now generic $q$. Let $v_r$, $v_s$ be as above.  Note first that when splitting a tuple ${\bf d}$ apart into say ${\bf d}_1=(d_1,d_2,\ldots, d_l)$ and ${\bf d}_2=(d_{l+1},\ldots,d_r)$ then we have $b({\bf d})=b({\bf d}_1)+b({\bf d}_2)+n_0({\bf d}_2)n_1({\bf
d}_1)$, where $n_0$ resp. $n_1$ counts the number of $0$'s respectively $1$'s. Now let $v_{\bf d}\in V_1^{\otimes i+j}$ appear in $\iota_r\otimes \iota_s(v_r\otimes v_s)$ and $v:=\Phi_{i,j}^k(v_{\bf d})\not=0$. We split into four parts,  $ {\bf d}=v_{w({\bf d}_1)} \otimes v_{g({\bf d}_2)} \otimes v_{g'({\bf d}_3)} \otimes v_{w'({\bf d}_4)} $, according to the $x$-, $z$- and $y$- group. More precisely
\begin{eqnarray*}
{\bf d}_1=(\underbrace{0, \ldots, 0}_{i-r+a-z}, \underbrace{1, \ldots, 1}_{r-a}),&&
{\bf d}_2=(\underbrace{0, \ldots, 0}_{z-a},
\underbrace{1, \ldots, 1}_{a})\\
{\bf d}_3 =(\underbrace{0, \ldots, 0}_{a}, \underbrace{1, \ldots, 1}_{z-a})&& {\bf d}_4
=(\underbrace{0, \ldots, 0}_{j-s-a}, \underbrace{1, \ldots, 1}_{s-z+a})
\end{eqnarray*}
with $w\in \mathbb{S}_x$, $w'\in \mathbb{S}_y$, $g,g'\in \mathbb{S}_z$ the appropriate minimal coset representatives.
Then this term gets weighted with
\begin{eqnarray}
\label{c}
c&:=&q^{b(w)+b(g)+(z-a)(r-a)+b(g')+b(w')+(y-s+z-a)(z-a)}\nonumber\\
&=&q^{(z-a)(r-a+y-s+z-a)}q^{b(w)+b(w')+2b(g)}.
\end{eqnarray}
Applying $ \Phi_{i,j}^k $ then gives $(-q^{-1})^{a}c v_{w({\bf d}_1)} \otimes v_{w'({\bf d}_4)}.$ Applying the projector $ \pi_k $ from \eqref{pn} to this element gives
$$q^{-t(k-t)}q^{l(w)+l(w')+(r-a)(y-s+z-a)}(-1)^aq^{-a} c\frac{v_t}{{k \brack t}}.$$
Evaluating with $v_t$ using \eqref{scalarprod} gets rid of the binomial coefficient.
Now summing over all the ${\bf d}$ using the identity
$$ \sum_{w \in \mathbb{S}_{m+n} / \mathbb{S}_m \times \mathbb{S}_n} q^{2b(w)-mn} = {m+n \brack m}$$
(where $w$ runs over all longest coset representatives) gives the desired formula
$$C_{i,j}^k(r,s,t)=q^{-t(k-t)}\sum_{a=0}^z (-1)^aq^{-a} q^{\gamma_a} {z \brack {a}} {x \brack {r-a}}{y \brack {j-s-a}}.$$
It is obvious that $\gamma_a$ has the interpretation as displayed in Figure \ref{fig:qcount}.
\end{proof}

\begin{ex}
\label{3jex}
{\rm
Consider the case $i=j=k=2$ and $r=s=t=1$. Then the $3j$-symbol evaluates to $C_{2,2}^{2}(1,1,1)=-q^{-2}+q^2$. There are only two line arrangement which differ in sign. We have $\gamma_1=0$ and $\gamma_0=3$, hence our formula says $C_{2,2}^{2}(1,1,1)=q^{-1}((-1)q^{-1}+q^3)$. The other values for $i=j=k=2$ are zero except in the following cases:
$C_{2,2}^2(1,0,0)=-q^{-1}$, $C_{2,2}^2(2,0,1)=-q^{-1}$,
$C_{2,2}^2(0,1,0) = q$,
$C_{2,2}^2(2,1,2) = -q^{-1}$,
$C_{2,2}^2(0,2,1) = q^{-1}$,
$C_{2,2}^2(1,2,2) = q$.
}
\end{ex}

\begin{ex}{\rm
Counting all the $(4,5,5,2,2,2)$-triangle arrangements (see Figure \ref{fig:example3j}) with signs we obtain $C_{4,5}^{5}(2,2,2)_{q=1}=1-12+3=-8$. The contributions for
the $q$-version are:
\begin{eqnarray*}
&\gamma_0=2\cdot 3+2\cdot 3+2\cdot 2+0+0+0=16,&\\
&\gamma_1=1\cdot 2+1\cdot2+1\cdot1+1\cdot1+1\cdot1+1\cdot2=9,&\\
&\gamma_2=0+0+0+0+0+1\cdot2=2,&
\end{eqnarray*}
hence for $a=0$ we get a contribution of $q^{16}$, for $a=1$ a contribution of $-q^8 {2\brack 1}{2 \brack 1}{3 \brack 2}=-(q^{12}+3q^{10}+4q^8+3q^6+q^{4})$ and for $a=2$ a contribution of ${3\brack 1}=q^{2}+1+q^{-2}$. Altogether,  $$C_{4,5}^{5}(2,3,3)=q^{-3}(q^{16}-q^{12}-3q^{10}-4q^8-3q^6-q^{4}+q^{2}+1+q^{-2}).$$
}
\end{ex}

\section{(Twisted) canonical basis and an alternate $3j$-formula}
In this section we will deduce new integral and positive $3j$-formulas by working in a special basis. They will later be categorified using cohomology rings of Grassmannians.\\

We already introduced the standard basis and dual standard basis of $V_{\bf d}$ given by the set of vectors $v_{\bf a}$ and $v^{\bf a}$ respectively. There is also Lusztig's {\it canonical basis} $\lbrace v_{a_1} \diamondsuit \cdots \diamondsuit v_{a_r} \rbrace $ and Lusztig's {\it dual canonical basis} $\lbrace v^{a_1} \heartsuit \cdots \heartsuit v^{a_r}\rbrace $, with $a_j$ as above.
These two bases are dual with respect to the bilinear form $\langle_-,_-\rangle'$ satisfying $\langle v_i\otimes v_j,v^k\otimes v^l\rangle'=\delta_{i,l}\delta_{j,k}$. It pairs a tensor product of two irreducible representations with the tensor product where the tensor factors are swapped. We call this form therefore the {\it twisted form}. For definitions and explicit formulas relating these bases we refer to \cite[Theorem 1.6 and Proposition 1.7]{FK}. The translation from our setup to theirs is given as follows:

\begin{remark}
\label{annoying}
{\rm
Let $\triangle_{FK}$ be the comultiplication and $\overline\triangle_{FK}$ the dual comultiplication from \cite[(1.2) resp. (1.5)]{FK}. Then there is an isomorphism of $\mathcal{U}_q$-modules
\begin{eqnarray}
\label{Omega}
\Omega:\quad (V_1^{\otimes n},\triangle)&\longrightarrow&(V_1^{\otimes n},\triangle_{FK}),
\end{eqnarray}
where $\Omega=q^{l(w_0^i)}{\Pi}_{w_0}=q^{i(n-i)}{\Pi}_{w_0}$ where ${\Pi}_{w_0}$ is the full positive twist. For $n=2$ the map $\Omega$ is given as follows:
\begin{eqnarray}
v_i\otimes v_i&\mapsto& v_i\otimes v_i\nonumber\\
v_1\otimes v_0&\mapsto&v_0\otimes v_1\nonumber\\
v_0\otimes v_1&\mapsto&v_1\otimes v_0+(q^{-1}-q)v_0\otimes v_1\label{nequalstwo}
\end{eqnarray}
where $i=0,1$. For arbitrary $n$ we pick a reduced expression of the longest element $w_0$ of $S_n$ and replace each simple transposition $s_i$ by the corresponding map \eqref{nequalstwo} acting on the $i$th and $i+1$-st tensor factor.
There is also an isomorphism of  $\mathcal{U}_q$-modules
\begin{eqnarray}
\label{D}
D:\quad(V_1^{\otimes n},\psi\otimes\psi\circ\triangle\circ\psi)&\longrightarrow&(V_1^{\otimes n},\overline{\triangle}_{FK}),\\
v^{\bf a}&\longmapsto&q^{{|\bf a|}(|n-{\bf a}|)}v^{{\bf a}},\nonumber
\end{eqnarray}
where $\psi$ is the anti-linear anti-automorphism of $\mathcal{U}_q$ satisfying $\psi(E) = E$, $\tau(F) = F$, $\tau(K)=K^{-1}$.
The isomorphisms for arbitrary tensor products are completely analogous, namely $v_{\bf a}\mapsto v_{w_0({\bf a})}$ and $v^{\bf a}\mapsto\prod_{i=0}^rq^{a_i(d_i-a_i)}v^{{\bf a}}$.
}
\end{remark}

Under these identifications the bilinear form $\langle_-, _-\rangle'$ turns into the anti-bilinear form $(_-,_-)$ on $V_{\bf d}$, and we obtain two pairs of distinguished bases of $V_{\bf d}$. First, the image of the canonical basis under ${\Omega}^{-1}$ paired with the image under $D^{-1}$ of the dual canonical basis (the former will turn out to be the twisted canonical basis defined below) and secondly, the preimage under $\Omega^{-1}$ of the standard basis paired with $D^{-1}$ applied to the dual standard basis. We will call the image of the dual canonical basis under $D$ the {\it shifted dual canonical basis}, since its expression in terms of the standard basis just differs by replacing $q^{-1}$ by $q$ in the explicit formulas \cite[Proposition 1.7]{FK} and additionally multiplication of the $q$-power involved in the definition of $D$. The special role of the dual canonical basis from Section \ref{U2basics} becomes transparent in the following result \cite[Theorem 1.11]{FK}:

\begin{theorem}
\label{FKsimples}
Let $v$ be an element of the dual canonical basis in $V_1^{\otimes n}$. Let $n=d_1+d_2+\cdots+ d_r$ with $d_j\in\mZ_{\geq 0}$. Then $\pi_{d_1}\otimes \pi_{d_2}\otimes\cdots\otimes \pi_{d_r}(v)$  is either zero or an element of the shifted dual canonical basis of
$V_{\bf d}=V_{d_1}\otimes V_{d_2}\otimes\cdots\otimes V_{d_r}$ and every element of the shifted dual canonical basis of $V_{\bf d}$ has this form with $v$ defined uniquely.
\end{theorem}

In the following we will describe, for two tensor factors, a twisted basis which will turn out to be the image of the canonical basis under $\Omega^{-1}$.
\begin{define}
\label{twistedcanbasis}
{\rm The {\it twisted canonical basis} $\{v_r\;\spadesuit\; v_s\mid 1\leq
r\leq i, 1\leq s\leq j\}$ of $V_i\otimes V_j$ is defined as follows.
\begin{eqnarray*}
v_r\;\spadesuit\;v_s&:=&
\begin{cases}
E^{(s)} F^{(i-r)} v_i \otimes v_0,&\text{if $ r+s \leq j$}\\
F^{(i-r)} E^{(s)} v_i \otimes v_0,&\text{if $r+s \geq j$}
\end{cases}\\
&=&
\begin{cases}
\sum_{p'=0}^s q^{p'(p'-s+j)} {p'+r \brack p'} v_{r+p'} \otimes
v_{s-p'},&\text{if $ r+s \leq j$}\\
\sum_{p''=0}^{i-r} q^{p''(p''+r)} {j-s+p'' \brack p''} v_{r+p'} \otimes
v_{s-p''},&\text{if $r+s \geq j$}
\end{cases}
\end{eqnarray*}
}
\end{define}

The following formulas, analogous to the formulas \cite[Section
3.1.5]{L} can be proved by an easy induction. We explicitly mention them here, since we chose a slightly different
comultiplication.
\begin{prop} The comultiplication in the divided powers is given by
\label{divcomult}
\begin{eqnarray*}
\triangle(E^{(r)})&=&\mathop{\sum_{r', r''}}_{r'+r''=r} q^{-r'r''}
E^{(r')} \otimes E^{(r'')} K^{-r'},\\
\triangle(F^{(r)})&=&\mathop{\sum_{r', r''}}_{r'+r''=r} q^{-r'r''}
F^{(r')} K^{r''} \otimes F^{(r'')}.
\end{eqnarray*}
\end{prop}

The standard basis can be expressed in terms of the twisted canonical
basis as follows:
\begin{prop} For $1\leq r\leq i, 1\leq s\leq j$ the following holds
\label{inversetwistedcanbasis}
\begin{eqnarray*}
v_r\otimes v_s&=&
\begin{cases}
\sum_{\gamma} (-1)^{\gamma} q^{\gamma(j-s+1)} {r +\gamma \brack \gamma}
v_{r+\gamma}\;\spadesuit\; v_{s-\gamma},&\text{if $r+s \leq j$}\\
\sum_{\gamma} (-1)^{\gamma} q^{\gamma(r+1)} {j-s+ \gamma \brack \gamma}
v_{r+\gamma}\;\spadesuit\;v_{s-\gamma}, &\text{if $r+s \geq j$}.
\end{cases}
\end{eqnarray*}
\end{prop}

\begin{proof}
This can be proved by induction on $s$.
\end{proof}

\begin{ex}
\label{example11}
In case $V_1\otimes V_1$ we have
\begin{eqnarray*}
&v_0\spadesuit v_0=v_0\otimes v_0,  v_1\spadesuit v_0=v_1\otimes v_0, v_0\spadesuit v_1=qv_1\otimes v_0+v_0\otimes v_1, v_1\spadesuit v_1=v_1\otimes v_1,&\\
&v_0\heartsuit v_0=v_0\otimes v_0,  v_1\heartsuit v_0=v_1\otimes v_0, v_0\heartsuit v_1=qv_1\otimes v_0+v_0\otimes v_1, v_1\heartsuit v_1=v_1\otimes v_1.&
\end{eqnarray*}
\end{ex}

\begin{lemma}
\label{adjointnessform} Consider the semi-linear form $\langle\;, \;\rangle$ on $V_k$. Then
\begin{eqnarray*}
\langle v_{t+a}, E^{(a)} v_t \rangle &=&\langle q^{2at-ak+a^2} F^{(a)}
v_{t+a},v_t \rangle,\\
\langle v_{t-a}, F^{(a)}
    v_t \rangle&=&\langle q^{ak-2at+a^2} E^{(a)} v_{t-a}, v_t \rangle.
\end{eqnarray*}
\end{lemma}

\begin{proof}
This follows directly from the $ \mathcal{U}_q$-action,
see Section ~\ref{U2basics} and \eqref{scalarprod}.
\end{proof}

Note that the definition of the $3j$-symbols depends on a choice of (basis) vectors. By choosing the twisted canonical basis we define
$$D_{i,j}^k(r,s,t)=\langle v_t,A_{i,j}^k(v_r \spadesuit v_s)\rangle.$$ The following formulas resemble formulas in \cite{GZ}.

\begin{theorem}[Positivity]\hfill\\
\label{positivity}
In the twisted canonical basis, the $3j$-symbols (up to a factor of $(-1)^{\frac{i+j+k}{2}}$), belong to $ \mathbb{N}[q,q^{-1}]$, more precisely:
\begin{equation*}
D_{i,j}^k (r,s,t)=
\begin{cases}
(-1)^{\frac{i+j-k}{2}} q^{\eta_1} {k-t+s \brack s}{t-s+i-r \brack
t-s}&\text{if $r+s\leq j$}\\
\\
(-1)^{\frac{i+j-k}{2}} q^{\eta_2} {i-r+t \brack t}{k-i+r-t+s \brack
s}&\text{if $r+s\geq j$}
\end{cases}
\end{equation*}
where
\small
\begin{eqnarray*}
\eta_1&=&
s(k+s-2r)+(i-r)(h-i-r)+t(k-t)+\frac{i+j-k}{2}+\left(\frac{i+k-j}{2}\right)\left(\frac{j+k-i}{2}\right),
\\
\eta_2&=&(i-r)(2s+r-h)+s(k-s-2i)+t(k-t)
+\left(\frac{i+k-j}{2}\right)\left(\frac{j+k-i}{2}\right).
\end{eqnarray*}
\normalsize
\end{theorem}

\begin{proof}
Assume $ r+s \leq j$.  The other case is similar and therefore omitted.
\begin{eqnarray*}
&&D_{i,j}^k(r,s,t)=\langle
v_t,A_{i,j}^k(v_r \spadesuit v_s) \rangle\\
&=&\langle v_t, E^{(s)}
F^{(i-r)} A_{i,j}^k(v_i \otimes v_0)\rangle
\\
&=& q^{\delta} \langle  E^{(i-r)} F^{(s)} v_t, A_{i,j}^k(v_i \otimes v_0),
\rangle
\\
&=& q^{\delta} {k-t+s \brack s}{t-s+i-r \brack t-s} \langle
v_{t+i-r-s}, A_{i,j}^k(v_i \otimes v_0) \rangle
\\
&=& q^{\delta} (-q)^{\frac{i+j-k}{2}} {k-t+s \brack
s}{t-s+i-r \brack t-s} \langle \pi_k (v_1^{\otimes {\frac{i+k-j}{2}}}
\otimes v_0^{\otimes \frac{j+k-i}{2}}), v_{t+i-r-s} \rangle
\\
&=& q^{\delta} (-q)^{\frac{i+j-k}{2}}
q^{(\frac{i+k-j}{2})(\frac{j+k-i}{2})}{k-t+s \brack
s}{t-s+i-r \brack t-s}
\end{eqnarray*}
where
\small
$$\delta = sk-2sr-s^2+2(i-r)k-2i(i-r)+(i-r)^2.$$
\normalsize
The second, third, fourth, and fifth equalities above follow from the
definition of the twisted canonical basis, Lemma ~\ref{adjointnessform},
the definition of the action of the divided powers, and the definition of $
A_{i,j}^k$ respectively. The sixth equality follows from the definition of
the projection map and the semi-linear form. The theorem follows.
\end{proof}

Using the twisted canonical basis we get expressions for the
$3j$-symbols in terms of  binomial coefficients, which differ from the ones described in Theorem~\ref{ClebschGordon}:
\small
\begin{theorem}
\begin{equation*}
C_{i,j}^k (r,s,t)=
\begin{cases}
\displaystyle\sum_{\gamma} (-1)^{\gamma+\frac{r+s+t}{2}} q^{\zeta_1} {r+\gamma
\brack r}{k-t+s-\gamma \brack s-\gamma}{t-s+i-r \brack t-s+\gamma}&\text{if
$r+s\leq j$}\\
\\
\displaystyle\sum_{\gamma} q^{\zeta_2} {j-s+\gamma
\brack \gamma}{i-r-\gamma+t \brack t}{k-i+r-t+s \brack s-\gamma}&\text{if
$r+s\geq j$}
\end{cases}
\end{equation*}
\normalsize
where
\small
\begin{eqnarray*}
\zeta_1=-\gamma(j-s+1)+
(s-\gamma)(k-2t+s-\gamma)+(i-r-\gamma)(k-2t+2s+i+r-3\gamma)-r-s+t\\
\zeta_2=-\gamma
(r+1)+(i-r-\gamma)(2t-k+i-r-\gamma)+(s-\gamma)(k-2i+2r-2t+s+\gamma)-r-s+t.
\end{eqnarray*}
\end{theorem}
\normalsize

\begin{proof}
Let us assume $r+s\leq j$, the other case is similar. Using Proposition \ref{inversetwistedcanbasis} and Theorem \ref{positivity} we obtain
\begin{eqnarray*}
C_{i,j}^k(r,s,t)&=&\langle v_t, A_{i,j}^k(\sum_{\gamma} (-1)^{\gamma} q^{\gamma(j-s+1)}
{r+\gamma \brack r} v_{r+\gamma} \;\spadesuit\; v_{s-\gamma})
\rangle\nonumber\\
&=&\sum_{\gamma} (-1)^{\gamma} q^{-\gamma(j-s+1)} {r+\gamma \brack r}
\langle v_t,A_{i,j}^k(v_{r+\gamma}\; \spadesuit\; v_{s-\gamma}))
\rangle\\
&=&\sum_{\gamma} (-1)^{\gamma}(1)^\frac{i+j-k}{2} q^{\delta_1}
{r+\gamma \brack r}{k-t+s-\gamma \brack s-\gamma}{t-s+i-r \brack
t-s+\gamma}
\end{eqnarray*}

where $\delta_1=-2(s-r)(r+\gamma)+(s-\gamma)k-(s-\gamma)^2-(i-r-\gamma)k+2i(i-r-\gamma)-(i-r-\gamma)^2)+t(k-t)+
\frac{i+j-k}{2}+\frac{i+k-j}{2}\frac{j+k-i}{2}$.
The asserted formula follows.
\end{proof}

\begin{remark}{\rm
The above formulas seem to be different from the standard formulas in the literature expressing $3j$-symbols as an alternating sum. For instance the formula using the dual canonical basis \cite[Proposition 3.18]{FK}.}
\end{remark}

\section*{Part II}
\section{Fractional graded Euler characteristics}
\label{Euler}
A general belief in the categorification community is that only integral structures can be categorified. Note however that \eqref{pn} includes a division by a binomial coefficient. In this section we illustrate how complete intersection rings can be used to categorify rational (quantum) numbers in terms of an Euler characteristic of an Ext-algebra. This will later be used in our categorification of the Jones-Wenzl projector. The approach provides furthermore interesting and subtle categorifications of integers and quantum numbers in terms of infinite complexes with cohomology in infinitely many degrees. The most important example will be the colored unknot discussed later in this paper. For the complete categorification of the colored Jones polynomial we refer to \cite{SS}. We start with a few results on Poincare polynomials of complete intersection rings.

Let $\lsem n\rsem=1+q^2+q^{4}+\cdots+q^{2(n-1)}$ be the renormalized quantum number, set  $\lsem n\rsem!=\lsem1\rsem\lsem2\rsem\cdots \lsem n\rsem$, and denote the corresponding binomial coefficients by $\left[n \brack k\right]$. By convention this binomial expression is zero if one of the numbers is negative or if $k>n$.

For an abelian (or triangulated) category $\cA$ let $[\cA]$ be the {\it Grothendieck group} of $\cA$ which is by definition the free abelian group generated by the isomorphism classes $[M]$ of objects $M$ in $\cA$  modulo the relation $[C]=[A]+[B]$ whenever there is a
short exact sequence (or distinguished triangle) of the form $A\rightarrow C\rightarrow B$. When $ \cA $ is a triangulated category,
denote $n$ compositions of the shift functor by $\ullcorner n\ulrcorner$. In the following $\cA$ will always be a (derived) category of $\mZ$-graded modules over some finite dimensional algebra $A$. Then $[\cA]$ has a natural $\mZ[q,q^{-1}]$-module structure where $q$ acts by shifting the grading up by $1$. We denote by $\langle i\rangle$ the functor which shifts the grading up by $i$. In the following we will only consider the case where $[\cA]$ is free of finite rank $r$ and often work with the $q$-adic completion of $[\cA]$, which is by definition, the free module over the formal Laurent series ring $\mZ[[q]][q^{-1}]$ of rank $r$. We call this the {\it completed Grothendieck group} (see \cite{AcS} for details).

\subsection{Categorifying $\frac{1}{\lsem2\rsem}$}
The complex cohomology ring of $\mC \mathbb{P}^1$ is the graded ring $R=\mC[x]/(x^2)$ where $x$ is homogeneous of degree $2$. In particular, $\lsem2\rsem$ agrees with its Poincare polynomial. We would like to have a categorical interpretation of its inverse. As a graded $R$-modules, $R$ fits into a short exact sequence of the form $\mC\langle 2\rangle\rightarrow R\rightarrow\mC$, where $\langle i\rangle$ means the grading is shifted up by $i$. Hence, we have the equality $[R]=[\mC\oplus\mC\langle 2\rangle]$ in the Grothendieck group of graded $R$-modules. The latter is a free $\mZ$-module with basis given by the isomorphism classes $[\mC\langle i\rangle]$, of the modules $\mC\langle i\rangle$ where $i\in \mZ$. Alternatively we can view it as a free $\mZ[q,q^{-1}]$-module on basis $[\mC]$ where $q^i[\mC]=[\mC\langle i\rangle]$. Then the above equality becomes $[R]=(1+q^{2})[\mC]$. We might formally write $[\mC]=\frac{1}{1+q^2}[R]=(1-q^2+q^4-q^6\ldots)[R]$ which then makes perfect sense in the completed Grothendieck group. Categorically it can be  interpreted as the existence of a (minimal) graded projective resolution of $\mC$ of the form
\begin{eqnarray}
\label{SL2}
\cdots\stackrel{f} {\longrightarrow}R\langle4\rangle\stackrel{f} {\longrightarrow}R\langle 2\rangle\stackrel{f} {\longrightarrow}R\stackrel{p}{\surj}\mC,
\end{eqnarray}
where $f$ is always multiplication by $x$ and $p$ is the standard projection. The graded Euler characteristic of the above complex resolving $\mC$ is of course just $1=(1+q^2)^{-1}(1+q^2)$. However, the Ext-algebra $\Ext^*_R(\mC,\mC)$ has the graded Euler characteristic $\frac{1}{1+q^2}$.

\subsection{Categorifying $\frac{1}{\lsem n\rsem}$}
More generally we could consider the ring $R=\mC[x]/(x^{n})$ viewed as the cohomology ring of $\mC \mathbb{P}^{n-1}$. Then there is a graded projective resolution of $\mC$ of the form
\begin{eqnarray}
\label{hyper}
\cdots\quad R\langle 4n\rangle \stackrel{g}{\longrightarrow}R\langle 2n+2\rangle \stackrel{f} {\longrightarrow}R\langle 2n\rangle\stackrel{g} {\longrightarrow}R\langle 2\rangle\stackrel{f} {\longrightarrow}R\stackrel{p}{\surj}\mC,
\end{eqnarray}
where $f$ is multiplication with $x$ and $g$ is multiplication with $x^{n-1}$. The identity of formal power series
$\frac{1}{\lsem n\rsem}=(1+q^2+q^4+\cdots +q^{2(n-1)})^{-1}=1-q^2+q^{2n}-q^{2n+2}+q^{4n}-\cdots$ can be verified easily. The algebra  $\Ext^*_R(\mC,\mC)$ has graded Euler characteristic $\frac{1}{\lsem n\rsem}$.

\subsection{Categorifying $\frac{1}{\lsem n\rsem!}$}
Consider the graded ring $$H=H_n=\mC[x_1,x_2,\ldots ,x_n]/I,$$ where $I$ is the ideal generated by symmetric polynomials
without constant term, and where $x_i$ has degree $2$. Via the Borel presentation this ring $H_n$ of coinvariants can be identified with the cohomology ring of the full flag variety $GL(n,\mC)/B$, where $B$ is the Borel subgroup of all upper triangular matrices or with the cohomology ring of $SL(n,\mC)/(B\cap SL(n,\mC))$.

\begin{theorem}[Categorification of fractions]\label{Thmres}\hfill
\begin{enumerate}[(i)]
\item We have the equalities $[H_n]=\lsem n\rsem![\mC]$ in the Grothendieck group of graded $H_n$-modules.
\item Any projective resolution of the module $\mC$ is infinite. The graded Poincare polynomial of the minimal resolution is of the form
\begin{eqnarray}
\label{Poincare}
\frac{(1+q^2t)^{n-1}}{\prod_{j=2}^n(1-q^{2j}t^2)}.
\end{eqnarray}
Here $t$ encodes the homological grading, whereas $q$ stands for the internal algebra grading.
\item The Ext-algebra $\Ext^*_{H_n}(\mC,\mC)$ has graded Euler characteristic $\frac{1}{\lsem n\rsem!}$.
\end{enumerate}
\end{theorem}

\begin{proof}
Since $H=H_n$ is a graded local commutative ring, the first statement is equivalent to the statement that the Poincare polynomial of $R_n$ equals $\lsem n\rsem!$, which is a standard fact, see for example \cite[Theorem 1.1]{Haglund}, \cite[10.2]{Fu}. To see the second statement, assume there is a finite minimal projective resolution
\begin{equation}
0\longrightarrow H^{n_r}\longrightarrow\cdots\longrightarrow H^{n_0}\surj \mC.
\end{equation}
Up to shifts, $H$ is self-dual (\cite[26-7]{Kane}) and so in particular injective as module over itself. Hence, the sequence splits and we get a contradiction to the minimality of the resolution. Moreover, $H$ has Krull dimension zero, since any prime ideal is maximal. (To see this let $\mathfrak{m}$ be the unique maximal ideal and $\mathfrak{p}$ any prime ideal. If $x\in \mathfrak{m}$ then $x^k=0\in\mathfrak{p}$ for big enough $k$ by degree reasons. Hence $x\in\mathfrak{p}$ and therefore $\mathfrak{m}=\mathfrak{p}$.) Now $H$ is minimally generated by the $x_i$'s for $1\leq i\leq n-1$, and $I$ is minimally generated by the $n-1$ different elementary symmetric functions of degree $r>2$ in the $x_i$'s for $1\leq i\leq n-1$. Hence $R$ is a complete intersection ring (\cite[Theorem 2.3.3]{CM}) and the Poincare series can be computed using \cite{Avramov}. We sketch the main steps. For each formal power series $P(t)=1+\sum_{i=1}^\infty a_jt^j$, there exist uniquely defined
$c_k\in\mZ$ such that the equality
\begin{eqnarray}
\label{product}
P(t)&=&\frac{\prod_{i=1}^{\infty}(1+t^{2i-1})^{c_{2i-1}}}
{\prod_{i=1}^{\infty}(1-t^{2i})^{c_{2i}}}
\end{eqnarray}
holds as formal power series or in the $(t)$-adic topology. (If we define $q_i(t):=(1-(-t)^i)^{(-1)^{i+1}}$, the right hand side of formula \eqref{product} is exactly $\prod_{i=1}^{\infty}q_i(t)^{c_i}$. We define $Q_0:=1$ and
inductively $Q_{m+1}:=Q_m(t)(q_{m+1}(t))^{c_m}$ with $c_m$ defined as
$P(t)-Q_m(t)\equiv c_mt^{m+1}\pmod{t^{m+2}}$. By definition we have
$Q_m(t)=\prod_{i=1}^{m}q_i(t)^{c_i}$. Using the binomial formula we see that
$P(t)\equiv Q_m(t)\pmod{t^{m+1}}$ holds for all $m\in\mN$.) In the case of the (ungraded) Poincare polynomial we are interested in the {\it deviations} $c_m$ are usually
denoted by $\epsilon_m(R)$ and complete intersections are characterized (\cite[Theorem 7.3.3]{Avramov}) by the property  that $c_m(R)=0$ for $m\geq 3$. Moreover, $c_1(H_n)=n=c_2(H_n)$ and $c_1(C_n)=n-1=c_2(C_n)$,  see \cite[Corollary 7.1.5]{Avramov}. Formula \eqref{product} implies the statement \eqref{Poincare} for the ungraded case $q=1$. The graded version follows easily by invoking the degrees of the generators $x_i$ and the degrees of the homogeneous generators of $I$.  The statement \eqref{3} is clear.
\end{proof}
\begin{ex}
\label{exflag}
{\rm
Consider the flag variety $SL(3,\mC)/B$. Then its cohomology algebra is isomorphic to the algebra of coinvariants
$C\cong\mC[X,Y]/(X^2+XY+Y^2,X^3+\frac{3}{2}X^2Y-\frac{3}{2}XY^2-Y^3)$, where $X$ and $Y$ correspond to the
simple coroots. If we choose as generators of the maximal ideal the elements
$x:=\overline{X}$, and $y:=\overline{X}+2\overline{Y}$, the defining relations turn into $y^2=-3x^2$, $xy^2=x^3$, $x^2y=0$ and the elements $1,x,y,x^2,xy,x^3$ form a basis of $C$. By a direct calculation one can check that in this basis the minimal projective resolution is of the form
\begin{eqnarray}
\cdots\longrightarrow C^4\stackrel{f_3} {\longrightarrow}C^3\stackrel{f_2} {\longrightarrow}C^2\stackrel{f_1} {\longrightarrow}C\stackrel{p}{\surj}\mC
\end{eqnarray}
with the linear maps given by matrices of the form
\small
\begin{eqnarray*}
[f_{2i+1}]=
\begin{pmatrix}
\boxed{\quad M\quad}\quad\quad\\
\quad\quad\quad\fbox{\quad M\quad }\\
\quad\quad\quad\quad\quad\ddots\\
&\boxed{\quad M\quad}\quad\quad\\
&\quad\quad\boxed{\:[x]\:[y]\:}
\end{pmatrix}
&&[f_{2i}]=
\begin{pmatrix}
\boxed{\quad N\quad}\quad\quad\\
\quad\quad\quad\fbox{\quad N\quad }\\
\quad\quad\quad\quad\quad\ddots\\
&\boxed{\quad N\quad}\\
\end{pmatrix},
\end{eqnarray*}
\normalsize
where $[x]$ denotes the matrix describing the multiplication with $x$ and where $M$ and $N$ are $2\cdot6\times 3\cdot6$ matrices as follows:
\begin{eqnarray*}
M=
\begin{pmatrix}
[x]&[y]&0\\
[y]&[-3x]&[x^2]\\
\end{pmatrix}
&\mbox{and}&
N=
\begin{pmatrix}
[3x]&[y]&0\\
[y]&[-x]&[x^2]\\
\end{pmatrix}.
\end{eqnarray*}
Hence the graded resolution is of the form
\begin{eqnarray*}
{\rightarrow}(2q^{12}+2q^{14}+2q^{16}+q^{20})C{\rightarrow}(2q^{10}+2q^{12}+2q^{14})C{\rightarrow}(2q^8+2q^{10}+q^{12})C
\quad\quad\\
\rightarrow (2q^6+2q^8)C{\rightarrow}(2q^4+q^6)C{\rightarrow}2q^2C{\rightarrow}C
\stackrel{p}{\surj}\mC.
\end{eqnarray*}
Pictorially this can be illustrated as follows, where we drew the standard generators of the free modules as dots labeled by their homogeneous degrees.
\begin{equation}
\cdots\rightarrow\;\stackrel{10}\bullet\;\stackrel{10}\bullet\;\stackrel{12}\bullet\;\stackrel{12}\bullet\;\stackrel{14}\bullet\;\stackrel{14}\bullet\;\;
\rightarrow\;\stackrel{8}\bullet\;\stackrel{8}\bullet\;\stackrel{10}\bullet\;\stackrel{10}\bullet\;\stackrel{12}\bullet\;\;
\rightarrow\;\stackrel{6}\bullet\;\stackrel{6}\bullet\;\stackrel{8}\bullet\;\stackrel{8}\bullet\;
\;\rightarrow\;
\stackrel{4}\bullet\;\stackrel{4}
\bullet\;\stackrel{6}\bullet\;
\;\rightarrow\;\stackrel{2}\bullet\;\stackrel{2}\bullet\;\;\rightarrow\; \stackrel{0}\bullet
\end{equation}

The graded Euler characteristic of the Ext-algebra $\Ext^*_C(\mC,\mC)$ equals
\begin{eqnarray*}
\frac{(1-q)^2}{(1-q^2)(1-q^3)}=\frac{1}{(1+q)(1+q+q^2)}=\frac{1}{\lsem3\rsem!}
\end{eqnarray*}
}
\end{ex}
\subsubsection*{Categorifying $\frac{1}{\left[n\brack k\right]}$ and inverse quantum multinomial coefficients}
Generalizing the above example one can consider the Grassmannian $\op{Gr}(k,n)$ of $k$-planes in $\mC^{n}$. Its complex cohomology ring $H_{k,n-k}=H^\bullet(\op{Gr}(k,n),\mC)$ is explicitly known (see for instance \cite{Fu}) and by the same arguments as above, a complete intersection (see e.g. \cite{RWY}). We have the equality $[H_{k,n-k}]=\left[n\brack k\right][\mC]$ in the Grothendieck group of graded $H_{k,n-k}$-modules. The graded Euler characteristic of $\op{Ext}^*_{H_{k,n-k}}(\mC,\mC)$ is equal to $\frac{1}{\left[n\brack k\right]}$. Following again \cite{Avramov} one could give an explicit formula for the Poincare polynomial of the minimal projective resolution of $\mC$. Note that all this generalizes directly to partial flag varieties such that the Euler characteristic of $\Ext^*_{H_{\bf d}}(\mC,\mC)$ is the inverse of a quantum multinomial coefficient
$$ \binom{n}{d_1, \ldots, d_r} = \frac{n!}{d_1 ! d_2 ! \cdots d_r ! (n-d_1
- \cdots - d_r)!},$$
if $H_{\bf d}$ denotes the cohomology ring of the partial flag variety of type ${\bf d}=(d_1,d_2,\cdots, d_r)$.

\section{Serre subcategories and quotient functors}
\label{sec:Serre}
Let $\cA$ be an abelian category. A {\it Serre subcategory} is a full subcategory $\mathcal{S}$ such that for any short exact sequence $M_1\rightarrow M\rightarrow M_2$  the object $M$ is contained in $\mathcal{S}$ if and only if $M_1$ and $M_2$ are contained in $\cS$.
Let $A$ be a finite dimensional algebra and let $S$ be a subset of the isomorphism classes of simple objects. Let $X$ be a system of representatives for the complement of $S$. Then the modules with all composition factors isomorphic to elements from $S$ form a Serre subcategory $\mathcal{S}$ of $\cA:=A-\Mod$. In fact, any Serre subcategory is obtained in this way and obviously an abelian subcategory. The quotient category $\cA/\cS$ can be characterized by a universal property, \cite{Ga}, similar to the characterization of quotients of rings or modules. The objects in $\cA/\cS$ are the same objects as in $\cA$, but the morphisms are given by
$$\Hom_{\cA/\cS}(M,N) =\varinjlim \Hom_{\cA}(M_0,N/N_0),$$
where the limit is taken over all pairs of submodules $M_0\subset M$ and $N_0\subset N$ such that $M/M_0$, and $N_0$ are contained in $\cS$. For an irreducible $A$-module $N$ let $P(N)$ denote the projective cover of $N$. Then the following is well-known (see for instance \cite[Prop. 33]{AM} for a detailed proof).

\begin{prop}
\label{Serre}
With the notation above set $P=P_{\cS}=\bigoplus_{N\in X} P(N)$. Then there is an equivalence of categories $$\cA/\cS\cong \mod-\End(P).$$ In particular, $\cA/\cS$ is abelian.
\end{prop}

The {\it quotient functor} is then $\Hom_\cA(P,?):\cA\rightarrow \cA/\cS$. We call its left adjoint $P\otimes_{\End(P)} ?$ the {\it inclusion functor}.

\section{Categorification of the Jones-Wenzl projector - basic example}
\label{categjones}
In the following we will categorify the Jones-Wenzl projector as an exact quotient functor followed by the corresponding derived inclusion functor, see Theorem \ref{catJW}.

The categorification of both, the $3j$-symbol and the colored Jones polynomial is based on a categorification of the representation ${V}_1^{\otimes n}$ and the Jones-Wenzl projector. By this we roughly mean that we want to upgrade each weight space into a $\mZ$-graded abelian category with the action of $E$, $F$ and $K$, $K^{-1}$ via exact functors (see below for more precise statements). Such categorifications were first constructed in \cite{FKS} (building on previous work of \cite{BFK}) via graded versions of the category $ \mathcal{O}$ for $ \mathfrak{gl}_n $ and various functors acting on this category.\\

\subsection{Categorification of irreducible modules}
\label{motivation}

The categorification from \cite[6.2]{FKS} of the irreducible modules $V_n$ has a very explicit description in terms of cohomology rings of Grassmannians and correspondences. It was axiomatized (using the language of $2$-categories) by Chuang and Rouquier in \cite{CR}. Although they only work in the not quantized setup, their results could easily be generalized to the quantized version. In the smallest non-trivial case, the example of $V_2$, we consider the direct sum of categories
$$\mathcal{C}_2\;:=\;\mC-\gmod\;\oplus\; \mC[x]/(x^2)-\gmod\;\oplus\;\mC-\gmod$$
of graded modules. Then there is an isomorphism of $\mC(q)$-vector spaces from the Grothendieck space $\mC(q)\otimes_{\mZ[q,q^{-1}]}[\cC]$ of $\cC$ to $V_2$ by mapping the isomorphism classes of simple modules concentrated in degree zero to the dual canonical basis elements $v^i$. The action of $E$ and $F$ are given by induction functors $\mC-\gmod\rightarrow \mC[x]/(x^2)-\gmod, M\mapsto  \mC[x]/(x^2)\otimes_\mC M\langle -1\rangle$ and restriction functors  $\mC[x]/(x^2)-\gmod\rightarrow \mC-\gmod$, illustrated in Figure \ref{fig:cat1}. For general $n$, the category $\mathcal{C}_2$ should be replaced by $\mathcal{C}_n=\bigoplus_{i=0}^n H^*(\op{Gr}(i,n))-\gmod$, see \cite[6.2]{FKS}, \cite[Example 5.17]{CR}.

\subsection{Categorification of $V_1\otimes V_1$}
The smallest example for a non-trivial Jones-Wenzl projector is displayed in Figure \ref{exbasic}.
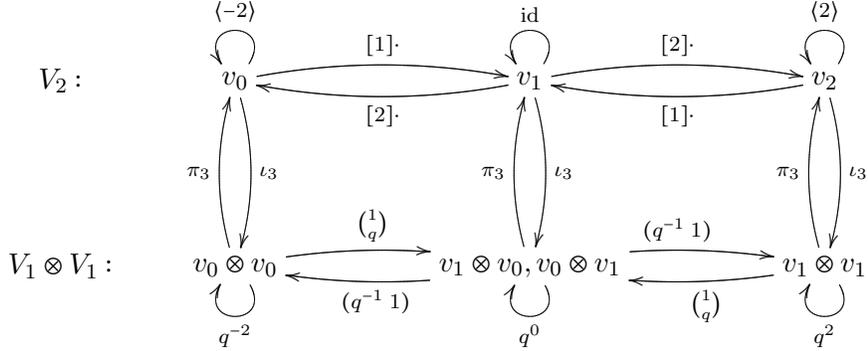
\begin{figure}
\label{exbasic}
\begin{array}[t]{lcl}
\text{ $\bf U_q(\mathfrak{sl}_2)$ \bf (not categorified)}\\
\xymatrix{
V_2:&v_0\ar@(ur,ul)[]_{\langle-2\rangle}\ar@/^/[rr]^{[1]\cdot}\ar@/^/[dd]^{\iota_3}&&
v_1\ar@(ur,ul)[]_{\op{id}}\ar@/^/[rr]^{[2]\cdot}\ar@/^/[ll]^{[2]\cdot}\ar@/^/[dd]^{\iota_3}&&
v_2\ar@(ur,ul)[]_{\langle2\rangle}\ar@/^/[ll]^{[1]\cdot}\ar@/^/[dd]^{\iota_3}
\\
\\
V_1\otimes V_1:&v_0\otimes v_0\ar@(dr,dl)[]^{q^{-2}}
\ar@/^/[rr]^{\!\!\!\!\!\!\binom{1}{q}}\ar@/^/[uu]^{\pi_3}
&&
v_1\otimes v_0, v_0\otimes v_1\ar@/^/[rr]^{(q^{-1}\;1)}
\ar@/^/[ll]^{\!\!\!\!(q^{-1}\;1)}\ar@(dr,dl)[]^{q^{0}}\ar@/^/[uu]^{\pi_3}
&&
v_1\otimes v_1\ar@/^/[ll]^{\quad\quad\binom{1}{q}}\ar@(dr,dl)[]^{q^{2}}\ar@/^/[uu]^{\pi_3}
}
\end{array}
\caption{Example of a Jones-Wenzl projector and inclusion.}
\end{figure}
where the horizontal lines denote the modules $V_2$ and $V_1\otimes V_1$ respectively with the standard basis denoted as ordered tuples. The horizontal arrows indicate the action of $E$ and $F$ in this basis, whereas the loops show the action of $K$. The vertical arrows indicate the projection and inclusion morphisms.

Let $A=\End_{\mC[x]/(x^2)}(\mC\oplus \mC[x]/(x^2))$. This algebra clearly contains $R$ as a subalgebra with basis $1,X$. One can identify $A$ with the path algebra of  the quiver $\stackrel{1}\bullet\leftrightarrows\stackrel{2}\bullet$ (with the two primitive idempotents $e_1$ and $e_2$) subject to the relation $1\rightarrow 2\rightarrow 1$ being zero. It is graded by the path length and $R=\Hom_A(Ae_2,Ae_2)=e_2Ae_2\cong\mC[x]/(x^2)$. In particular, $R-\Mod$ is a quotient category of $A-\Mod$. Using the graded version we have $R-\gmod\cong A-\gmod/\cS$, where $\cS$ denotes the Serre subcategory of all modules containing simple composition factors isomorphic to graded shifts of $\mC e_1$. Figure~\ref{fig:cat1} presents now a categorification of $V_1\otimes V_1$.

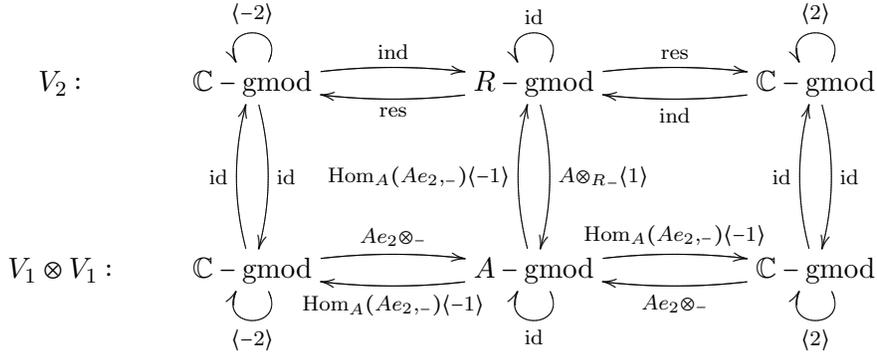
\begin{figure}
\label{fig:cat1}
\begin{eqnarray*}
\xymatrix{
V_2:&\mC-\gmod\ar@(ur,ul)[]_{\langle{-2}\rangle}\ar@/^/[rr]^{\op{ind}}\ar@/^/[dd]^{\op{id}}&&
R-\gmod\ar@(ur,ul)[]_{\op{id}}\ar@/^/[rr]^{\op{res}}\ar@/^/[ll]^{\op{res}}\ar@/^/[dd]^{{A\otimes_R}_-\langle 1\rangle}&&
\mC-\gmod\ar@(ur,ul)[]_{\langle 2\rangle}\ar@/^/[ll]^{\op{ind}}\ar@/^/[dd]^{\op{id}}\\
\\
V_1\otimes V_1:&\mC-\gmod\ar@(dr,dl)[]^{\langle-2\rangle}
\ar@/^/[rr]^{Ae_2\otimes_-}\ar@/^/[uu]^{\op{id}}
&&
A-\gmod\ar@/^/[rr]^{\Hom_A(Ae_2,_-)\langle -1\rangle}
\ar@/^/[ll]^{\Hom_A(Ae_2,_-)\langle -1\rangle}\ar@(dr,dl)[]^{\op{id}}\ar@/^/[uu]^{\Hom_A(Ae_2,_-)\langle -1\rangle}
&&
\mC-\gmod\ar@/^/[ll]^{Ae_2\otimes_-}\ar@(dr,dl)[]^{\langle 2\rangle}\ar@/^/[uu]^{\op{id}}
}
\end{eqnarray*}
\caption{The categorification of the Jones Wenzl-projector and the inclusion (using projective modules).}
\end{figure}

Note that the bases from Example \ref{example11} have a nice interpretation here. If we identify the dual canonical basis elements with the isomorphism classes of simple modules, then the standard basis corresponds to the following isomorphism classes of representations of the above quiver
\small
\begin{eqnarray}
\label{Vermas}
&\left\{[\mC], [\xymatrix{\mC\ar@/^/[r]^{0}&\mC\ar@/^/[l]^{\op{id}}}], [\xymatrix{0\ar@/^/[r]^{0}&\mC\ar@/^/[l]^{0}}], [\mC]\right\}, \left\{[\mC], [\xymatrix{0\ar@/^/[r]^{0}&\mC\ar@/^/[l]^{0}}], [\xymatrix{\mC\ar@/^/[r]^{0}&0\ar@/^/[l]^{0}}], [\mC]\right\}, &\nonumber\\
&\left\{[\mC], [\xymatrix{\mC\ar@/^/[r]^{\!\!\!\!\!\!\!\!\!\binom{1}{0}}&\mC\langle1\rangle\oplus\mC
\langle-1\rangle\ar@/^/[l]^{\!\!\!\!\!\!\!\!\!(0\;1)}}], [\xymatrix{0\ar@/^/[r]^{0}&\mC\ar@/^/[l]^{0}}], [\mC] \right\}&
\end{eqnarray}
\normalsize
and the twisted canonical basis corresponds to the indecomposable projective modules. One can easily verify that, when applied to the elements of the twisted canonical basis, the functors induce the $U_q(\mathfrak{sl}_2)$-action and morphisms of the above diagram. Since however the displayed functors are not exact, one has to derive them and pass to the (unbounded) derived category to get a well-defined action on the Grothendieck group.  We will do this in Section \ref{sec:JW} and at the same time extend the above to a Lie theoretic categorification which works in greater generality.

\begin{remark}
{\rm Note that the algebra $A$ has finite global dimension, hence a phenomenon as in Theorem \ref{Thmres} does not occur. Note also that for $V_1^{\otimes 2}$ we have three distinguished bases such that the transformation matrix is upper triangular with $1$'s on the diagonal. This is not the case for the irreducible representations. Categorically this difference can be expressed by saying $A$ is a quasi-hereditary algebra, whereas $R$ is only properly stratified, \cite[2.6]{MS2}.
}
\end{remark}

\begin{remark}{\rm
 The explicit connection of the above construction to Soergel modules can be found in \cite{Strquiv}. As mentioned already in the introduction we would like to work with the abelianization of the category of Soergel modules which we call later the {\it Verma category}.}
\end{remark}

\section{The Verma category $\cO$ and its graded version}
We start by recalling the Lie theoretic categorification of $\overline{V}_1^{\otimes n}$.
Let $n$ be a non-negative integer. Let $\mathfrak{g}=\mathfrak{gl}_n$ be the Lie algebra of complex $n\times n$-matrices.
Let $\mh$ be the standard Cartan subalgebra of all diagonal matrices with the standard basis $\lbrace E_{1,1}, \ldots, E_{n,n} \rbrace $ for $ i=1,\ldots, n$.  The dual space $ \mathfrak{h}^* $ comes with the dual basis $ \lbrace e_i \mid i=1,
\ldots, n \rbrace $ with $e_i(E_{j,j})=\delta_{i,j}$. The nilpotent subalgebra of strictly upper diagonal matrices spanned by $ \lbrace E_{i,j} \mid i<j \rbrace$ is denoted $\mathfrak{n}^{+}$. Similarly, let $ \mathfrak{n}^{-} $ be the subalgebra consisting of lower triangular matrices.  We fix the standard Borel subalgebra $\mathfrak{b}=\mathfrak{h} \bigoplus \mathfrak{n}^{+}$.  For any Lie algebra $L$ we denote by $\mathcal{U}(L)$ its universal enveloping algebra, so $L$-modules are the same as (ordinary) modules over the ring $\mathcal{U}(L)$.

Let $W=\mathbb{S}_n$ denote the Weyl group of $ \mathfrak{gl}_n $ generated by simple reflections (=simple transpositions) $ \lbrace s_i, 1 \leq i \leq n-1 \rbrace$. For $w\in W $ and $ \lambda \in \mathfrak{h}^*$, let $w\cdot\lambda = w(\lambda+\rho_n)-\rho_n$, where $ \rho_n= \frac{n-1}{2}e_1 + \cdots + \frac{1-n}{2}e_n$ is half the sum of the positive roots. In the following we will always consider this action. For $\la\in\mh^*$ we denote by  $W_{\lambda} $ the stabilizer of $\lambda \in \mathfrak{h}^*$.

Let $\la\in\mh^*$ and $\mC_\la$ the corresponding one-dimensional $\mh$-module. By letting $\mathfrak{n}^{+}$ act trivially we extend the action to $\mb$. Then the {\it Verma module} of highest weight $\la$ is
$$M(\la)=\mathcal{U}\otimes_{\mathcal{U}(\mb)}\mC_\la.$$

\begin{define}
\label{defO}
{\rm
We denote by $\mathcal{O} = \mathcal{O}(\mathfrak{gl}_n)$  the smallest abelian category of $\mathfrak{gl}_n$-modules containing all Verma modules and which is closed under tensoring with finite dimensional modules, finite direct sums, submodules and quotients. We call this category the {\it Verma category} $\cO$.
}
\end{define}

This category was introduced in \cite{BGG} (although it was defined there in a slightly different way) under the name category $\mathcal{O}$. For details and standard facts on this category we refer to \cite{Hu}.

Every Verma module $M(\la)$ has a unique simple quotient which we denote by $L(\la)$. The latter form precisely the isomorphism classes of simple objects in $\cO$. Moreover, the Verma category $\cO$ has enough projectives. We denote by  $P(\la)$ the projective cover in $\cO$. The category decomposes into indecomposable summands $\cO_\la$, called {\it blocks}, under the action of the center of $\cU(\mg)$. These blocks are indexed by the $W$-orbits (or its maximal representatives $\la$, called {\it dominant weights}, in $\mh^*$ for the Bruhat ordering). Note that the module $L(\la)$ is finite dimensional if and only if $\la$ is dominant and integral. Then the $L(w\cdot\la)$, $w\in W/W_\la$ are precisely the simple objects in $\cO_\la$.

Weight spaces of $V_1^{\otimes n}$ will be categorified using the blocks $\mathcal{O}_k(\mathfrak{gl}_n)$ corresponding to the integral dominant weights $e_1+\cdots+e_k-\rho_n$ for $1\leq k\leq n$. To make calculations easier denote also by $ M(a_1, \ldots, a_n) $ the Verma module with highest weight $ a_1 e_1 + \cdots + a_n e_n - \rho_n$ with simple quotient $L(a_1, \ldots, a_n) $ and projective cover $ P(a_1, \ldots, a_n) $ in $ \mathcal{O}(\mathfrak{gl}_n).$ They are all in the same block and belong to $ \mathcal{O}_k(\mathfrak{gl}_n) $ if and only if $k$ of the $a_j$'s are $1$ and $n-k$ of them are $0$. In this case we can identify the isomorphism class $[M(a_1, \ldots, a_n)]$ of $M(a_1, \ldots, a_n)$ with a standard basis  vector in $(V_1^{\otimes n})_k$ via
\begin{eqnarray}
\label{isoK0}
[M(a_1, \ldots, a_n)]&\mapsto&v_{a_1}\otimes v_{a_2}\otimes \cdots\otimes v_{a_n}
\end{eqnarray}
which then can be reformulated (using e.g. \cite[Theorems 3.10, 3.11]{Hu}) as

\begin{lemma}{\rm (\cite{BFK})}
\label{Grothungraded} The map \eqref{isoK0} defines an isomorphism of vector spaces:
$$ \mathbb{C} \otimes_{\mathbb{Z}} [\bigoplus_{k=0}^{n} \mathcal{O}_k(\mathfrak{gl}_n)] \cong \bar{V}_1^{\otimes n}. $$
\end{lemma}

\begin{ex}
\label{easiest}{\rm
The natural basis $v_0\otimes v_0$, $v_1\otimes v_0$, $v_0\otimes v_1$, $v_1\otimes v_1$ will be identified with the isomorphism classes of the Verma modules $$M(00), M(10), M(01), M(11)$$ respectively. The first  and last Verma module are simple modules as well and  the categories $\mathcal{O}_0(\mathfrak{gl}_n)$ and  $\mathcal{O}_2(\mathfrak{gl}_n)$ are both equivalent to the category of finite dimensional complex vector spaces.  The category $\mathcal{O}_1(\mathfrak{gl}_n)$ is equivalent to the category of finite dimensional modules over the above mentioned path algebra.}
\end{ex}

Each block of $\cO$ is equivalent to a category of right modules over a finite dimensional algebra, namely the endomorphism ring of a minimal projective generator. These algebras are not easy to describe, see \cite{Strquiv} where small examples were computed using (singular) Soergel modules. Therefore, our arguments will mostly be Lie theoretic in general, but we will need some properties of the algebras. Denote by  $A_{k,n}$ the endomorphism algebra of a minimal projective generator $P_k$ of $ \mathcal{O}_k(\mathfrak{gl}_n)$, hence
$$\epsilon:\quad\mathcal{O}_k(\mathfrak{gl}_n)\cong \MOD-A_{n,k}, \quad M\mapsto\HOM_\mg(P_k,_-).$$
The following statement is crucial and based on a deep fact from \cite{BGS}:

\begin{prop}
\label{Koszul}
 There is a unique non-negative $\mZ$-grading on $A_{k,n}$ which is Koszul.
\end{prop}

\begin{remark}
{\rm Note that we only work with very special blocks, hence use very special cases of \cite{BGS}. In these cases the Koszul dual algebra $A_{k,n}^!$ can be defined diagrammatically, \cite{BS1}, \cite{BS3}, by slightly generalizing Khovanov's arc algebra from \cite{KhovJones}. Koszulity can then be proved by elementary tools \cite{BS2}.
}
\end{remark}

Proposition \ref{Koszul} allows us to work with the category $A_{k,n}-\gmod$, hence a graded version of our Verma category $\cO$ with the grading forgetting functor ${\bf f}: A_{k,n}-\gmod\rightarrow A_{k,n}-\gmod$. An object in $\cO_k$ is called {\it gradable}, if there exists a graded module $\hat{M}\in A_{k,n}-\gmod$ such that ${\bf f} (\hat{M})\cong\epsilon (M)$. The following is well-known (note however that not all modules are gradable by \cite[Theorem 4.1]{Strgrad}):

\begin{lemma}
Projective modules, Verma modules and simple modules are gradable. Their graded lifts are unique up to isomorphism and grading shift.
\end{lemma}

\begin{proof}
The first part can be proved precisely as for instance in \cite[Lemma 8.1]{MS2} using the description of Verma modules as in \cite[p.2945]{MS2}. The uniqueness is \cite[Lemma 2.5.3]{BGS} or \cite[Lemma 2.5]{Strgrad}.
\end{proof}

We pick and fix such lifts $\hat{P}(a_1, \ldots, a_n), \hat{M}(a_1, \ldots, a_n) $ and $ \hat{L}(a_1, \ldots, a_n) $ by requiring that their heads are concentrated in degree zero.

Category $\cO$ has a contravariant duality which via $\epsilon$ just amounts to the fact that $A_{k,n}$ is isomorphic to its opposite algebra and after choosing such an isomorphism the duality is just $\Hom_{A_{k,n}}(_-,\mC)$.
In particular we can from now on work with the category $A_{k,n}-\gmod$ of graded left modules and choose the graded lift  $\op{d}=\Hom_{A_{k,n}}(_-,\mC)$ of the duality, see \cite[(6.1)]{Strgrad}. Then $\op{d}$ preserves the $\hat{L}(a_1, \ldots, a_n)$'s.
The module $\hat{\nabla}(a_1, \ldots, a_n)=\op{d}(\hat{M}(a_1, \ldots, a_n)$ is a graded lift of the ordinary dual Verma module ${\nabla}(a_1, \ldots, a_n)$. Similarly $\hat{I}(a_1, \ldots, a_n)=\op{d}(\hat{P}(a_1, \ldots, a_n))$ is the injective hull of $\hat{L}(a_1, \ldots, a_n)$ and a graded lift of the injective module ${I}(a_1, \ldots, a_n)$.

\section{Categorification of $ V_1^{\otimes n}$ using the Verma category $\cO$}
Let $ \langle r \rangle \:\colon\: A_{k,n}-\gmod \rightarrow A_{k,n}-\gmod$ be the functor of shifting the grading up by $r$ such that if $M=M_i$ is concentrated in degree $i$ then $M\langle k\rangle=M\langle k\rangle_{i+k}$ is concentrated in degree $i+k$. The
additional grading turns the Grothendieck group into a $\mZ[q,q^{-1}]$-module, the shift functor $\langle r\rangle$ induces the multiplication with $q^r$. Then we obtain the following:

\begin{prop}{\rm \cite[Theorem 4.1, Theorem 5.3]{FKS}}
\label{Grothgraded} There is an isomorphism of $\mathbb{C}(q)$-vector spaces:
\begin{eqnarray}
\label{iso26}
\Phi_n:\; \mathbb{C}(q) \otimes_{\mathbb{Z}[q,q^{-1}]} [\bigoplus_{k=0}^{n} A_{k,n}-\gmod] &\cong&V_1^{\otimes n}.\\
\left[\hat{M}(a_1, \ldots, a_n)\right]&\mapsto&v_{a_1}\otimes v_{a_2}\otimes \cdots\otimes v_{a_n}\nonumber
\end{eqnarray}
Under this isomorphism, the isomorphism class of $[L(a_1, \ldots, a_n)]$ is mapped to the dual canonical basis element $v^{a_1} \heartsuit \cdots \heartsuit v^{a_n}$.
\end{prop}

The following theorem categorifies the $\cU_q$-action.

\begin{theorem}[{\rm \cite[Theorem 4.1]{FKS}}]
\label{projfunc} There are exact functors of graded categories
\begin{eqnarray*}
&\hat{\mathcal{E}}_k \colon\quad A_{k,n}-\gmod \rightarrow A_{k+1,n}-\gmod,\quad
\hat{\mathcal{F}}_k \colon\quad A_{k,n}-\gmod \rightarrow A_{k-1,n}-\gmod,&\\
&\hat{\mathcal{K}}_k^{\pm 1} \colon\quad A_{k,n}-\gmod
\rightarrow A_{k,n}-\gmod&
\end{eqnarray*}
such that
\begin{eqnarray*}
&{\hat{\mathcal{K}}_{i+1}} \hat{\mathcal{E}}_i \cong {\hat{\mathcal{E}}_i} {\hat{\mathcal{K}}}_i \langle 2
    \rangle&\\
    &{\hat{\mathcal{K}}}_{i-1}\hat{\mathcal{F}}_i \cong {\hat{\mathcal{F}}_i} {\hat{\mathcal{K}}}_i \langle -2 \rangle& \\
    &{\hat{\mathcal{K}}}_i {\hat{\mathcal{K}}}_i^{-1} \cong {\Id}
    \cong {\hat{\mathcal{K}}}_i^{-1} {\hat{\mathcal{K}}}_i^1&\\
    &\displaystyle \hat{\mathcal{E}}_{i-1}
    \hat{\mathcal{F}}_i \oplus
    \bigoplus_{j=0}^{n-i-1} {\Id} \langle n-2i-1-2j \rangle \cong \hat{\mathcal{F}}_{i+1} \hat{\mathcal{E}}_i
    \oplus
    \bigoplus_{j=0}^{i-1}
    {\Id} \langle 2i-n-1-2j \rangle&
\end{eqnarray*}
and the isomorphism \eqref{iso26} becomes an isomorphism of $\mathcal{U}_q(\mathfrak{sl}_2)$-modules.
\end{theorem}
The functor $\hat{\mathcal{E}}=\oplus_{i\in\mZ}\hat{\mathcal{E}}_i$ is defined as a graded lift of tensoring with
the natural $n$-dimensional representation of $\mathfrak{gl}_n$ and $\hat{\mathcal{E}}_i$ is then obtained by projecting onto the required block. The functors
$\hat{\mathcal{F}}=\oplus \hat{\mathcal{F}}_i$ and
$\hat{\mathcal{F}}_i$ are then (up to grading shifts) their adjoints, whereas ${\hat{\mathcal{K}}}$ acts by an appropriate grading shift ${\hat{\mathcal{K}}}_i$ on each block.

\begin{remark}
\label{divpowers}
{\rm More generally, exact functors $\hat{\mathcal{E}}_k^{(r)}, \hat{\mathcal{F}}_k^{(r)}$ were defined which categorify divided
powers $E^{(r)}, F^{(r)}$ (\cite[Proposition 3.2]{FKS}). The ungraded version, denote them by ${\mathcal{E}}_k^{(r)}, {\mathcal{F}}_k^{(r)}$, are just certain direct summands of tensoring with the $k$-th tensor power of the natural representation or its dual respectively.}
\end{remark}

\begin{remark}
\label{rmCR}
{\rm
Chuang and Rouquier showed that the (ungraded) functorial action of $ \mathfrak{sl}_2$ on the Verma category $\cO$ is an example of an $\mathfrak{sl}_2$-categorification \cite[Sections 5.2 and 7.4]{CR}. Hence there is an additional action of a certain Hecke algebra on the space of $2$-morphisms or natural transformations between compositions of the functors $\mathcal{E}$ and $\mathcal{F}$.
}
\end{remark}

\subsection{Categorification of the Jones-Wenzl projector as a quotient functor}
\label{sec:JW}
Theorem \ref{FKsimples} characterizes the projector as a quotient map which either sends dual canonical basis elements to shifted dual canonical basis elements or annihilates them. On the other hand, Theorem \ref{Grothgraded} identifies the dual canonical basis elements with simple modules. Hence it is natural to categorify the projector as the quotient functor with respect to the Serre subcategory $\cS$ generated by all simple modules whose corresponding dual canonical basis element is annihilated by the projector.

Let $n=d_1+d_2+\cdots d_r$ with $d_j\in\mZ_{\geq 0}$. Define ${}_k\pi'_{\bf d}$ to be the projector $\pi_{\bf d}=\pi_{d_1}\otimes \pi_{d_2}\otimes\cdots\otimes \pi_{d_r}(v)$ from \eqref{iso26} restricted to the $q^{(2i-n)}$-weight space of $V_1^{\otimes n}$.
Consider the set of dual canonical basis vectors in $V_1^{\otimes n}$ annihilated resp. not annihilated by the projector  ${}_k\pi'_{\bf d}$. Let $S$ and $S'$, respectively, the corresponding set of isomorphism classes of simple modules under the bijection \eqref{iso26} taken with all possible shifts in the grading. According to Proposition~\ref{Serre} the set $S$ defines a Serre subcategory $\cS$ in $\cA:=A_{k,n}-\gmod$.  If we set $P_{k,{\bf d}}:=P_\cS=\bigoplus_{N} P(N)$, where the
where we sum over the projective covers $P(N)$ of all $N$ from a complete system of representatives from $S'$, then the quotient category $\cA/\cS$ is canonically equivalent to the category of modules over the endomorphism ring $A_{k, {\bf d}}$ of $P_{k,{\bf d}}$.

\begin{definition}
Let $e$ be an idempotent of $A_{k,n}$ such that $P_{k,{\bf d}}=A_{k,n}e$. Define the quotient and inclusion functors
\begin{eqnarray}
{}_k\hat\pi_{\bf d}:=\Hom_{A_{k, n}}(A_{k,n}e,_-)\langle -k(n-k)\rangle:\; A_{k, n}-\gmod\rightarrow A_{k, {\bf d}}-\gmod\label{algebraic1}\\
{}_k\hat\iota_{\bf d}:={A_{k, {\bf d}}\otimes_{eA_{k, {\bf d}}e}}_-\langle k(n-k)\rangle:\;A_{k, {\bf d}}-\gmod\rightarrow A_{k, n}-\gmod.\label{algebraic2}
\end{eqnarray}
Set ${}_k\hat\pi_{\bf d}=\bigoplus_{k=0}^n{}_k\hat\pi_{\bf d}$ and $\hat\iota_{\bf d}=\bigoplus_{k=0}^n{}_k\hat\iota_{\bf d}$.
\end{definition}

\begin{lemma}
\label{idempotent}
The composition $\tilde{p}_{k,\bf d}={}_k\iota'_{\bf d}{}_k\pi'_{\bf d}$ is an idempotent, i.e. $\tilde{p}_{k,\bf d}^2=\tilde{p}_{k,\bf d}\tilde{p}_{k,\bf d}$.
\end{lemma}

\begin{proof}
This follows directly from the standard fact that ${}_k\pi'_{\bf d}{}_k\iota'_{\bf d}$ is the identity functor, \cite{Ga}.
\end{proof}

The shift in the grading should be compared with \eqref{D}. The above functors are graded lifts of the functors
\begin{eqnarray}
{}_k\pi'_{\bf d}=\Hom_{A_{k, n}}(A_{k,n}e,_-): &&A_{k, n}-\mod\rightarrow A_{k, {\bf d}}-\MOD\label{alg1}\\
{}_k\iota'_{\bf d}={A_{k, {\bf d}}\otimes_{eA_{k, {\bf d}}e}}_-:&&A_{k, {\bf d}}-\mod\rightarrow A_{k, n}-\MOD.\label{alg2}
\end{eqnarray}

\section{The Lie theoretic description of the quotient categories}
In this section we recall a well-known Lie theoretic construction of the category  $A_{k, {\bf d}}-\Mod$ which then will be used to show that the quotient and inclusion functor, hence also $\tilde{p}_{k,\bf d}$, naturally commute with the categorified quantum group action.

 We first describe an equivalence of $A_{k, {\bf d}}-\Mod$ to the category $ {}_{k} \mathcal{H}_{\mu}^{1}(\mathfrak{gl}_n)$ of certain Harish-Chandra bimodules. Regular versions of such categories in connection with categorification were studied in detail in \cite{MS2}. For the origins of (generalized) Harish-Chandra bimodules and for its description in terms of Soergel bimodules see \cite{Soergel-HC}, \cite{StGolod}. Ideally one would like to have a graphical description of these quotient categories $A_{k, {\bf d}}-\Mod$, similar to \cite{EK}.
 The Lie theoretic details can be found in \cite{BG}, \cite[Kapitel 6]{Ja}.  Note that in contrast to the additive category  of Soergel bimodules, the categories $A_{k, {\bf d}}-\Mod$ have in general many indecomposable objects, more precisely are of wild representation type, see \cite{DM}.

\begin{define}{\rm
Let $\la, \mu$ be integral dominant weights. Let $ \mathcal{O}_{\lambda, \mu}(\mathfrak{gl}_n) $ be the full subcategory of $ \mathcal{O}_{\lambda}(\mathfrak{gl}_n) $ consisting of
modules $ M $ with projective presentations $ P_2 \rightarrow P_1 \rightarrow M \rightarrow 0 $ where $ P_1 $ and $ P_2 $ are direct sums
of projective objects of the form $ P(x.\lambda) $ where $x$ is a longest element representative in the double coset space $ \mathbb{S}_{\mu}
\backslash \mathbb{S}_n / \mathbb{S}_{\lambda} $ and $ \mathbb{S}_{\mu}, \mathbb{S}_{\lambda} $ are the stabilizers of the weights $ \mu $ and $ \lambda $
respectively.}
\end{define}

The following result is standard (see e.g. \cite[Prop. 33]{AM} for a proof).
\begin{lemma}
\label{ppres}
Let $\mathbb{S}_{\lambda}\cong \mathbb{S}_k\times \mathbb{S}_{n-k}$ and $\mathbb{S}_{\mu}\cong \mathbb{S}_{\bf d}$. Then there is an equivalence of categories $\alpha:\mathcal{O}_{\lambda, \bf d}(\mathfrak{gl}_n)\rightarrow A_{k,{\bf d}}-\MOD$ such that $\alpha\circ\iota\cong {}_k\iota'_{\bf d}\circ\alpha$, where $\iota$ is the inclusion functor from  $\mathcal{O}_{\lambda, \mu}(\mathfrak{gl}_n)$ to $\mathcal{O}_{\lambda}(\mathfrak{gl}_n)$.
\end{lemma}

The subcategory has a nice intrinsic definition in terms of the following category:

\begin{define}{\rm
Let $\mg=\mathfrak{gl}_n$ and define for $\mu$, $\la$ dominant integral weights $ {}_{\lambda} \mathcal{H}_{\mu}^1(\mg) $ to be the full subcategory of $\mathcal{U}(\mg)$-bimodules of finite length with objects $M$ satisfying the following conditions
\begin{enumerate}[(i)]
\item $M$ is finitely generated as $\mathcal{U}(\mg)$-bimodule,
\item every element $m\in  M$ is contained in a finite dimensional vector space stable under the adjoint action $x.m=xm-mx$ of $\mg$ (where $x\in\mg$, $m\in M$),
\item for any $m\in M$ we have $m\chi_\mu=0$ and there is some $n\in\mZ_{>0}$ such that $(\chi_\lambda)^nm=0$, where $\chi_\mu$, resp $\chi_\la$ is the maximal ideal of the center of $\cU(\mg)$ corresponding to $\mu$ and $\la$ under the Harish-Chandra isomorphism. (One usually says {\it $M$ has generalized central character $\chi_\la$ from the left and (ordinary) central character $\chi_\mu$ from the right}).
\end{enumerate}
}
\end{define}
We call the objects in these categories short {\it Harish-Chandra bimodules}.

\subsection{Simple Harish-Chandra bimodules and proper standard modules}
Given two $\mg$-modules $M$ and $N$ we can form the space $\HOM_\mC(M,N)$ which is naturally a $\mg$-bimodule, but very large. We denote by $\cL(M,N)$ the ad-finite part, that is the subspace of all vectors lying in a finite dimensional vector space invariant under the adjoint action $X.f:=Xf-fX$ for $X\in \mg$ and $f\in\HOM_\mC(M,N)$. This is a Harish-Chandra bimodule, see \cite[6]{Ja}. This construction defines the so-called Bernstein-Gelfand-Joseph-Zelevinsky functors (see \cite{BG}, \cite{Ja}):
\begin{eqnarray}
{}_{\lambda} \bar{\pi}_{\mu} \colon \mathcal{O}_{\lambda}(\mathfrak{gl}_n) \rightarrow {}_{\lambda}
    \mathcal{H}_{\mu}^1(\mathfrak{gl}_n),&& {}_{\lambda} \bar{\pi}_{\mu}(X) =\cL(M(\mu),X),\label{Lietheoretic1}\\
{}_{\lambda} \bar{\iota}_{\mu} \colon {}_{\lambda}\mathcal{H}_{\mu}^1(\mathfrak{gl}_n) \rightarrow \mathcal{O}_{\lambda}(\mathfrak{gl}_n), &&{}_{\lambda} \bar{\iota}_{\mu}(M) = M \otimes_{\mathcal{U}(\mathfrak{gl}_n)} M(\mu).\label{Lietheoretic2}
\end{eqnarray}

\begin{theorem}{\rm (\cite{BG})}
\label{BG} The functors $ {}_{\lambda} \bar{\iota}_{\mu} $ and $ {}_{\lambda} \bar{\pi}_{\mu} $ provide inverse equivalences of categories
between $ \mathcal{O}_{\lambda, \mu}(\mathfrak{gl}_n) $ and $ {}_{\mu} \mathcal{H}_{\lambda}^1(\mathfrak{gl}_n).$
\end{theorem}

From the quotient category construction the following is obvious. It might be viewed, with Proposition \ref{Grothgraded}, as a categorical version of \cite[Theorem 1.11]{FK}.
\begin{corollary}
\label{simples}
\begin{enumerate}[(i)]
\item$ {}_{\lambda} \bar{\pi}_{\mu} $ maps simple objects to simple objects or zero. All simple Harish-Chandra bimodules are obtained in this way.
\item The simple objects in ${}_{\mu} \mathcal{H}_{\lambda}^1(\mathfrak{gl}_n)$  are precisely the $\cL(M(\mu), L(x.\lambda))$ where $ x $ is a longest element representative in the double coset space $ \mathbb{S}_{\mu}
\backslash \mathbb{S}_n / \mathbb{S}_{\lambda}$. In particular if $\mu$ has stabilizer $\mathbb{S}_{\bf d}$, then the
$$L\left(k_1,d_1|k_2,d_2|\cdots |k_r,d_r\right):=\cL\left(M(\mu), L(\underbrace{0, \ldots, 0}_{d_1-k_1}
\underbrace{1, \ldots, 1}_{k_1}\cdots
\underbrace{0, \ldots, 0}_{d_r-k_r}, \underbrace{1, \ldots, 1}_{k_r})\right)$$
for $0\leq k_i\leq d_i$, $\sum k_i=k$ are precisely the simple modules in ${}_{k} \mathcal{H}_{\bf {d}}^1(\mathfrak{gl}_n)$.
\end{enumerate}
\end{corollary}

The following special modules are the images of Verma modules under the quotient functor and will play an important role in our categorification:

\begin{definition}
\label{propstandard}
{\rm The {\it proper standard module} labeled by $(k_1,d_1|k_2,d_2|\cdots |k_r,d_r)$ is defined to be
\begin{eqnarray*}
&\blacktriangle\left(k_1,d_1|k_2,d_2|\cdots |k_r,d_r\right)&\\
&:=\cL\left(M(\mu), M(\underbrace{0, \ldots, 0}_{d_1-k_1},
\underbrace{1, \ldots, 1,}_{k_1}\cdots,
\underbrace{0, \ldots, 0,}_{d_r-k_r} \underbrace{1, \ldots, 1}_{k_r})\right)\in {}_{k} \mathcal{H}_{\bf {d}}^1(\mathfrak{gl}_n).&
\end{eqnarray*}
The name stems from the fact that this family of modules form the proper standard modules in a properly stratified structure (see \cite[2.6, Theorem 2.16]{MS2} for details.)}
\end{definition}
To summarize: we have an equivalence of categories
$$A_{k,{\bf d}}{\text -}\Mod\cong {}_{k} \mathcal{H}_{\lambda}^1(\mathfrak{gl}_n)$$ under which (after forgetting the grading) the algebraically defined functors \eqref{alg1}, \eqref{alg2} turn into the Lie theoretically defined functors \eqref{Lietheoretic1}, \eqref{Lietheoretic2}. Hence $A_{k,{\bf d}}-\gmod$ is a graded version of Harish-Chandra bimodules. Let
\begin{eqnarray*}
\hat{L}\left(k_1,d_1|k_2,d_2|\cdots |k_r,d_r\right)&&\hat{\blacktriangle}\left(k_1,d_1|k_2,d_2|\cdots |k_r,d_r\right)
\end{eqnarray*}
be the standard graded lifts with head concentrated in degree zero of the corresponding Harish-Chandra bimodules.

\begin{lemma}
\label{action}
The  $U_q(\mathfrak{sl}_2)$-action from Theorem \ref{projfunc} factors through the Serre quotients and induces via
\begin{eqnarray*}
{}_k\hat\pi_{\bf d}:=\Hom_{A_{k, n}}(A_{k,n}e,_-)\langle -k(n-k)\rangle: &&A_{k, n}-\gmod\rightarrow A_{k, {\bf d}}-\gmod\\
{}_k\hat\iota_{\bf d}:={A_{k, {\bf d}}}{\otimes_{eA_{k, {\bf d}}e}}{}_-\langle k(n-k)\rangle:&&A_{k, {\bf d}}-\gmod\rightarrow A_{k, n}-\gmod.
\end{eqnarray*}
an $U_q(\mathfrak{sl}_2)$-action on $\bigoplus_{k=0}^n A_{k, {\bf d}}-\gmod$ such that
$\hat\pi_{\bf d}\hat{\mathcal{E}}\cong \hat\pi_{\bf d}\hat{\mathcal{E}}$ and  $\hat\iota_{\bf d}\hat{\mathcal{E}}\cong\hat{\mathcal{E}}\hat\iota_{\bf d}$, similarly for $\hat{\mathcal{F}}$,  $\hat{\mathcal{K}}^\mp$. The composition ${}_k\hat\iota_{\bf d}{}_k\hat\pi_{\bf d}$ equals $\tilde{p}_{k,\bf d}$, hence is an idempotent.
\end{lemma}

\begin{proof}
Denote by $\cS_k$ the Serre subcategory of ${A_{k, n}}-\gmod$ with respect to which we take the quotient. Consider the direct sum $\cC:=\oplus_{k=0}^n\cS_k$, a subcategory of $\bigoplus_{k=0}^n A_{k, n}-\gmod$, the kernel of the functor $\hat\pi_{\bf d}$.
Recall that the  $U_q(\mathfrak{sl}_2)$-action is given by exact functors, hence by \cite{Ga}, it is enough to show that the functors preserve the subcategory $\cC$. This is obvious for the functors $\hat{\mathcal{K}}^{\mp 1}$, since $\cC$ is closed under grading shifts. It is then enough to verify the claim for $\hat{\mathcal{E}}$, since the statement for $\hat{\mathcal{F}}$ follows by adjointness, see \cite{Ga}. Since the functor $\hat{\mathcal{E}}$ is defined as a graded lift of tensoring with the natural $n$-dimensional representation of $\mathfrak{gl}_n$ and then projecting onto the required block, it preserves the category of all modules which are not of maximal possible Gelfand-Kirillov dimension by \cite[Lemma 8.8]{Ja}. We will show in Proposition \ref{GK} that this category agrees with $\cC$. Hence $\hat\pi_{\bf d}\hat{\mathcal{E}}\cong \hat\pi_{\bf d}\hat{\mathcal{E}}$ and $\hat\iota_{\bf d}\hat{\mathcal{E}}\cong\hat{\mathcal{E}}\hat\iota_{\bf d}$ at least when we forget the grading. It is left to show that the grading shift is chosen correctly on each summand. This is just a combinatorial calculation and clear from the embedding \cite[(39)]{FKS}. The last statement follows from Lemma \ref{idempotent}, since the two grading shifts cancel each other.
\end{proof}

The following theorem describes the combinatorics of our categories and strengthens \cite[Theorem 4.1]{FKS}:

\begin{theorem}[Arbitrary tensor products and its integral structure]\hfill\\
\label{cattensor}
With the structure from Lemma \ref{action}, there is an isomorphism of  $U_q(\mathfrak{sl}_2)$-modules
\begin{eqnarray}
\Phi_{\bf d}:\; \displaystyle \mathbb{C}(q) \otimes_{\mathbb{Z}[q,q^{-1}]} \left[\bigoplus_{k=0}^{n} A_{k, {\bf
d}}-\gmod\right] &\cong& V_{d_1} \otimes \cdots \otimes V_{d_r} \nonumber\\
\hat{L}\left(k_1,d_1|k_2,d_2|\cdots |k_r,d_r\right)&\longmapsto&v^{k_1}\heartsuit v^{k_2}\heartsuit \cdots \heartsuit v^{k_r},\label{basissimples}
\end{eqnarray}
This isomorphism sends proper standard modules to the dual standard basis:
\begin{eqnarray}
\label{basisproperstandards}
\hat{\blacktriangle}\left(k_1,d_1|k_2,d_2|\cdots |k_r,d_r\right)&\longmapsto&v^{k_1}\otimes v^{k_2}\otimes \cdots \otimes v^{k_r}.
\end{eqnarray}
\end{theorem}

\begin{proof}
Theorem \ref{Grothgraded} identifies the isomorphism classes of simple objects in the categorification of $V_1^{\otimes n}$ with dual canonical basis elements and the isomorphism classes of Verma modules with the standard basis. The exact quotient functor $\hat\pi_{\bf d}$ sends simple objects to simple objects or zero and naturally commutes with the $U_q(\mathfrak{sl}_2)$-action.
Since the isomorphism classes of the $\hat{L}\left(k_1,d_1|k_2,d_2|\cdots |k_r,d_r\right)$ form a $\mC(q)$-basis of the Grothendieck space the map $\Phi_{\bf d}$ is a well-defined $\mC(q)$-linear isomorphism.
Let $p$ be the linear map induced by $\hat\pi_{\bf d}$ on the Grothendieck space. Then Proposition \ref{Grothgraded}, Theorem \ref{FKsimples} imply that
\begin{eqnarray}
\label{commutes}
\Phi_{\bf d}\circ p=\pi_{\bf d}\circ\Phi_n,
\end{eqnarray}
in the basis of the isomorphism classes of simple modules (note that the $q$-power $q^{-\prod_{i=1}^rk_i(d_i-k_i)}$ appearing in \eqref{pn} is precisely corresponding to the shift in the grading appearing in the definition of \eqref{algebraic1}) and \eqref{basissimples} follows.
Now $p$ is surjective, and  $p$, $\pi_{\bf d}$ and $\Phi_{\bf d}$ are $U_q(\mathfrak{sl}_2)$-linear, hence so is $\Phi_{\bf d}$.
By definition, $\hat{\blacktriangle}\left(k_1,d_1|k_2,d_2|\cdots |k_r,d_r\right)$ is (up to some grading shift), the image under ${}_k\hat{\pi}_{\bf d}$ of $\hat{M}({\bf a})$ for ${\bf a}$ as in Definition \ref{propstandard}. On the other hand the Jones-Wenzl projector maps $v^{\bf a}$ for such ${\bf a}$ to $v^{k_1}\otimes v^{k_2}\otimes \cdots \otimes v^{k_r}$ up to some $q$-power. This $q$-power equals $q^{-\prod_{i=1}^rk_i(d_i-k_i)}$ by \eqref{pn}, hence by comparing with the grading shift in \eqref{algebraic1}, the statement follows from \eqref{commutes}.
\end{proof}

\subsection{The Categorification theorem}
While the functor $ {}_{\lambda}{\pi}_{\mu} $ is exact, the functor $ {}_{\lambda}{\hat\iota}_{\mu} $ is only right exact, so we consider its left derived functor $\mathbb{L}{}_{\la}{\hat\iota}_{\mu}$. Note that ${\hat\pi}_{\bf d}$ extends uniquely to a functor on the bounded derived category, but might produce complexes which have only infinite projective resolutions. Therefore to be able to apply  $\mathbb{L}{}_{\la}{\hat\iota}_{\mu}$ afterwards, we have to work with certain unbounded derived categories. In particular, it would be natural to consider the graded functors
\begin{eqnarray*}
\mathbb{L}{}_k\hat\iota_{\bf d}:\;D^-(A_{k, {\bf d}}-\gmod)\longrightarrow D^-(A_{k, n}-\gmod).
\end{eqnarray*}
where we use the symbol $D^{-}(?)$ to denote the full subcategory of the derived category $D(?)$ consisting of complexes bounded to the right.  The exact functors from the $U_q(\mathfrak{sl}_2)$-action extend uniquely to the corresponding  $D^{}(?)$ and preserve $D^{-}(?)$.

However passing to $D^{-}(?)$ causes problems from the categorification point of view. The Grothendieck group of $D^{-}(?)$ might collapse, \cite{Miyachi}. To get the correct or desired Grothendieck group we therefore restrict to a certain subcategory $\Ddf\subset D^{-}(A_{k, {\bf d}}-\gmod)$ studied in detail in \cite{AcharStroppel}. This category is still large enough so that our derived functors make sense, but small enough to avoid the collapsing of the Grothendieck group.
It is defined as the full subcategory of $D^{-}(A_{k, {\bf d}}-\gmod)$ of all complexes $X$ such that for each $m \in \mZ$, only finitely many of the cohomologies $H^i(X)$ contain a composition factor concentrated in degree $<m$. It is shown in \cite{AcharStroppel} that
the Grothendieck group $[\Ddf(A_{k, {\bf d}}-\gmod)]$ is a complete topological $\mZ[q,q^{-1}]$-module and that the natural map
$[A_{k, {\bf d}}-\gmod] \rightarrow [\Ddf(A_{k, {\bf d}}-\gmod)]$ is injective and induces an isomorphism
\label{it:compl-isom}
\[
\hat{K}\left(A_{k, {\bf d}}-\gmod\right) \cong [\Ddf(A_{k, {\bf d}}-\gmod)]
\]
where $\hat{K}\left(A_{k, {\bf d}}-\gmod\right)$ denotes the $q$-adic completion of $[A_{k, {\bf d}}-\gmod]$ which is as $\mZ[[q]][q^{-1}]$-module free of rank equal to the rank of $[A_{k, {\bf d}}-\gmod]$. Let $\hat{V}_{\bf d}$ be the $\mZ[[q]][q^{-1}]$-module obtained by $q$-adic completion of $V_{\bf d}$.

\begin{theorem}[Categorification of the Jones-Wenzl projector]\hfill
\label{catJW}
\begin{enumerate}
\item The composition $\hat{p}_{k,\bf d}:=(\mathbb{L}{}_k\hat{\iota}_{\bf d}){}_k\hat{\pi}_{\bf d}$ is an idempotent.
\item $\mathbb{L}({}_{k} {\hat\iota}_{\bf d})$ induces a functor
$$\Ddf(\bigoplus_{k=0}^n A_{k,{\bf d}}-\gmod) \rightarrow \Ddf(\bigoplus_{k=0}^n A_{k,n}-\gmod).$$
\item There is an isomorphism of  $\mZ[[q]][q^{-1}]$-modules
$$\Phi:\quad \bigoplus_{k=0}^n\hat{K}\left(A_{k, {\bf d}}-\gmod\right)\cong \hat{V}_{\bf d}$$
sending the isomorphism classes of the standard modules defined in \eqref{defstandard} to the corresponding standard basis elements.
\item Under the isomorphism $\Phi$ the induced map
$$ \displaystyle \bigoplus_{k=0}^n [\mathbb{L}({}_{k} \hat{\iota}_{\bf d})] \colon\quad
    \bigoplus_{k=0}^n\hat{K}\left(A_{k, {\bf d}}-\gmod\right)\rightarrow \bigoplus_{k=0}^n \hat{K}(A_{k,n}-\gmod) $$
    is equal to the tensor  product ${\iota}_{d_1} \otimes \cdots \otimes {\iota}_{d_r}$
    of the inclusion maps.
\item Moreover,
$$ \displaystyle \bigoplus_{k=0}^n [\mathbb{L}({}_{k} \hat{\pi}_{\bf d})] \colon\quad
\bigoplus_{k=0}^n \hat{K}(A_{k,n}-\gmod)\rightarrow \bigoplus_{k=0}^n\hat{K}\left(A_{k, {\bf d}}-\gmod\right)$$
is equal to the tensor product  $\pi_{d_1} \otimes \cdots \otimes \pi_{d_r}$ of the projection maps.
\end{enumerate}
\end{theorem}
\begin{proof}
Note that ${}_k\hat{\pi}_{\bf d}$ is exact and that ${}_k\hat{\pi}_{\bf d}\mathbb{L}\hat{\iota}_{\bf d}$ is the identity functor (this follows e.g. from \cite[Lemma 2.1]{MS2}). Hence  $\hat{p}_{k,\bf d}$ is an idempotent as well. The second statement is proved in \cite{AcharStroppel}. As well as the analogous statement for $\mathbb{L}({}_{k} \hat{\pi}_{\bf d})$.
The last statement is then \eqref{commutes}. The statement about the induced action on the Grothendieck group for the inclusion functors will be proved as Corollary \ref{standards} in the next section.
\end{proof}

\begin{ex}[Gigantic complexes]
\label{gigantic}
{\rm The complexity of the above functors is already transparent in the situation of Example \ref{JWex}:
The projective module $P:=\hat{P}(01)\in A_{1,2}-\gmod$ fits into a short exact sequence of the form
\begin{eqnarray}
\label{ses}
\hat{M}(10)\langle 1\rangle\rightarrow\hat{P}(01)\rightarrow \hat{M}(01),
\end{eqnarray}
categorifying the twisted canonical basis from Example \ref{example11}. Now both, $\hat{M}(10)$ and $\hat{M}(01)$ are mapped under the Jones-Wenzl projector to one-dimensional $\mC[x]/(x^2)$-modules only different in the grading, namely to $\mC$ and $\mC\langle -1\rangle$ respectively. To compute the derived inclusion  $\mathbb{L}({}_1\hat\iota_{1,1})$ we first have to choose a projective resolution, see \eqref{SL2},
and then apply the functor ${}_1\hat{\iota}_{1,1}$. So, for instance the formula $\iota_2\circ\pi_2(v_0\otimes v_1)=\iota_2(q^{-1}[2]^{-1}v_1)=[2]^{-1}(v_1\otimes v_0+q^{-1}v_0\otimes v_1)=\lsem2\rsem^{-1}(qv_1\otimes v_0+v_0\otimes v_1)$ gets categorified by embedding the infinite resolution from \eqref{SL2}, shifted by $\langle -1\rangle\langle1\rangle=\langle 0\rangle$, via ${}_1\hat\iota_{1,1}$ to obtain
\begin{eqnarray}
\label{SL2P}
\cdots\stackrel{f} {\longrightarrow}\hat{P}(01)\langle4\rangle\stackrel{f} {\longrightarrow}\hat{P}(01)\langle 2\rangle\stackrel{f} {\longrightarrow}\hat{P}(01)\stackrel{p}{\surj}\mC.
\end{eqnarray}

This complex should then be interpreted as an extension of $\lsem2\rsem^{-1}\hat{M}(10)\langle1\rangle$ and $\lsem2\rsem^{-1}\hat{M}(01)$ recalling that we have a short exact sequence \eqref{ses}.
\begin{center}
{\it Note that this is a complex which has homology in all degrees!}
\end{center}
}
\end{ex}

A categorical version of the characterization of the Jones-Wenzl projector (Proposition \ref{charJW}) will be given in Theorem \ref{charJWcat} and a renormalized (quasi-idempotent) of the Jones-Wenzl projector in terms of Khovanov's theory will be indicated in Section \ref{normalizedJW} and studied in more detail in a forthcoming paper.

\section{Fattened Vermas and the cohomology rings of Grassmannians}
\label{fatVermas}
In this section we describe the natural appearance of the categorifications from Section \ref{Euler} in the representation theory of semisimple complex Lie algebras. More precisely we will show that the Ext-algebra of proper standard modules has a fractional Euler characteristic. Hence their homological properties differ seriously from the homological properties of Verma modules. Verma modules $M(\mu)$ are not projective modules in $\cO$, but they are projective in the subcategory of all modules whose weights are smaller than $\mu$. When passing to the Harish-Chandra bimodule category (using the categorified Jones-Wenzl projector) these Vermas become proper standard and the corresponding homological property gets lost. To have objects with similar properties as Verma modules we will introduce {\it fattened Vermas} or {\it standard objects}. They can be realized as extensions of all those Verma modules that are, up to shift, mapped to the same proper standard object under the Jones-Wenzl projector.

\begin{definition}{\rm
The simple modules $L(\la)$ in $\cO$ are partially ordered by their highest weight. We say $L(\la)<L(\mu)$ if $\la<\mu$ in the usual ordering of weights. This induces also an ordering on the simple modules in $A_{k,{\bf d}}$ which is explicitly given by
$$\hat{L}\left(k_1,d_1|k_2,d_2|\cdots |k_r,d_r\right)\leq\hat{L}\left(l_1,d_1|l_2,d_2|\cdots |l_r,d_r\right)$$
if and only if $\sum_{i=1}^nk_i=\sum_{i=1}^nl_i$ and $\sum_{i=1}^jk_i\leq\sum_{i=1}^jl_i$ for all $j$. For a simple module $\hat{L}:=\hat{L}\left(k_1,d_1|k_2,d_2|\cdots |k_r,d_r\right)\in A_{k,{\bf d}}-\gmod$, let $\hat{P}\left(k_1,d_1|k_2,d_2|\cdots |k_r,d_r\right)$ be its projective cover. Then we define the {\it standard module}
\begin{equation}
\label{defstandard}
\hat{{\Delta}}\left(k_1,d_1|k_2,d_2|\cdots |k_r,d_r\right)
\end{equation}
to be the maximal quotient of $\hat{P}\left(k_1,d_1|k_2,d_2|\cdots |k_r,d_r\right)\langle-\prod_{i=1}^rk_i(d_i-k_i)\rangle$ contained in the full subcategory $A_{k,{\bf d}}^{\leq L}-\gmod$ which contains only simple composition factors smaller or equal than $\hat{L}$.
}
\end{definition}

\begin{remark}
\label{standardsinduced}
{\rm After forgetting the grading, the standard objects can also be defined Lie theoretically as parabolically induced `big projectives' in category $\cO$ for the Lie algebra $\mathfrak{gl}_{\bf d}:=\mathfrak{gl}_{d_1}\oplus\mathfrak{gl}_{d_2}\oplus\cdots\oplus\mathfrak{gl}_{d_r}$, in formulas
\small
\begin{eqnarray*}
{\Delta}\left(k_1,d_1|k_2,d_2|\cdots |k_r,d_r\right)&=&\cU(\mathfrak{gl_n})\otimes_{\cU(\mathfrak{p})}\left(P(k_1|d_1)\boxtimes P(k_2|d_2)\boxtimes \cdots\boxtimes P(k_r|d_r)\right),
\end{eqnarray*}
\normalsize
where the $\mathfrak{gl}_{\bf d}$-action is extended by zero to $\mathfrak{p}=\mathfrak{gl}_{\bf d}+\mathfrak{n}$,
see \cite[Proposition 2.9]{MS2}. Similarly, \cite[p.2948]{MS2}, the proper standard modules can be defined as
\small
\begin{eqnarray*}
{\blacktriangle}\left(k_1,d_1|k_2,d_2|\cdots |k_r,d_r\right)&=&\cU(\mathfrak{gl_n})\otimes_{\cU(\mathfrak{p})}\left(L(k_1|d_1)\boxtimes L(k_2|d_2)\boxtimes \cdots\boxtimes L(k_r|d_r)\right).
\end{eqnarray*}
\normalsize
}
\end{remark}

\begin{prop}
\label{acyclic}
The standard objects are acyclic with respect to the inclusion functors, i.e.
$$\mathbb{L}^i({}_k\hat\iota_{\bf d}){\hat\Delta}=0$$ for any $i>0$ and $\hat{\Delta}={\hat\Delta}\left(k_1,d_1|k_2,d_2|\cdots |k_r,d_r\right)$.
\end{prop}

\begin{proof}
We do induction on the partial ordering. If $\hat{L}\left(k_1,d_1|k_2,d_2|\cdots |k_r,d_r\right)$ is maximal, then the corresponding standard module is projective and the statement is clear. From the arguments in \cite[Prop. 2.9, Prop. 2.13, Lemma 8.4]{MS2} it follows that any indecomposable projective module $$P\left(k_1,d_1|k_2,d_2|\cdots |k_r,d_r\right)\in A_{k,{\bf d}}-\gmod$$ has a graded {\it standard filtration} that means a filtration with subquotients isomorphic to (possibly shifted in the grading) standard modules.
Amongst the subquotients, there is a unique occurrence of ${\hat\Delta}\left(k_1,d_1|k_2,d_2|\cdots |k_r,d_r\right)$, namely as a quotient, and all other ${\hat\Delta}$'s are strictly larger in the partial ordering. Hence by induction hypothesis, using the long exact cohomology sequence we obtain $\mathbb{L}^i({}_k\hat\iota_{\bf d}){\hat\Delta}$ for $i\geq 2$ and then finally, by comparing the characters, also the vanishing for $i=1$.
\end{proof}

The following should be compared with \eqref{in}:
\begin{prop}
\label{prop50}
Let $\hat{\Delta}=\left({\hat\Delta}\left(k_1,d_1|k_2,d_2|\cdots |k_r,d_r\right)\right)\in A_{k,{\bf d}}-\gmod$.
${}_k\hat\iota_{\bf d}(\hat{\Delta})\in A_{k,n}-\gmod$ has a graded Verma flag. The occurring subquotients are, all occurring with multiplicity $1$, precisely the $\hat{M}({\bf a})\langle{b(\bf a)}\rangle$'s, where ${\bf a}$ runs through all possible sequences with $k$ ones such that precisely $\sum_{i=1}^jk_i$ appear amongst the first $\sum_{i=1}^jd_i$ indices.
\end{prop}

\begin{proof}
 As above, the existence of a Verma flag when we forget the grading follows directly from \cite[Prop. 2.18]{MS2} using Remark \ref{standardsinduced}. In the graded setting this is then the general argument \cite[Lemma 8.4]{MS2}.
 To describe the occurring subquotients we first work in the non-graded setting. The highest weight of $\Delta$ is of the form $w\cdot\mu$, where $\la=e_1+\cdots+e_k-\rho_n$ and $w\in\mathbb{S}_{\mu}\backslash \mathbb{S}_n / \mathbb{S}_{\lambda}$ a longest double coset representative, where $\mathbb{S}_{\lambda}=\mathbb{S}_k\times\mathbb{S}_{n-k}$ and $\mathbb{S}_{\mu}=\mathbb{S}_{\bf d}$.

 Our claim is then equivalent to the assertion that the occurring modules are precisely the isomorphism classes of the form $M(yw\cdot\la)$, where $y$ runs through a set of representatives from
 \begin{eqnarray}
 \label{indices}
 \mathbb{S}_{\mu}/(\mathbb{S}_{\mu}\cap (w\mathbb{S}_{\la}w^{-1})).
\end{eqnarray}
 By Remark \ref{standardsinduced} the translation functor $\theta^0_\la$ to the principal block maps $\Delta$ to $\Delta(w\cdot0)$ (since translation commutes with parabolic induction and translation out of the wall sends antidominant projective modules to such). The module $\Delta(w\cdot0)$ has by \cite[Prop. 8.3]{MS2} a Verma flag with subquotients precisely the $M(yw\cdot0)$'s with $y\in \mathbb{S}_{\mu}$. Since ${}_{0} \bar{\iota}_{\mu}\theta^0_\la\cong \theta^0_\la{}_{\lambda} \bar{\iota}_{\mu}$ the claim follows directly from the formula $[\theta^0_\la M(w\cdot0)]=[\bigoplus_{x\in \mathbb{S}_{\la}}M(wx\cdot0)]$. Indeed the latter formula says that we have to find a complete set of representatives for the $\mathbb{S}_\la$-orbits acting from the right on $\{yw\mid y\in \mathbb{S}_{\mu}\}$. Now $y'w=ywa$ for  $y, y'\in \mathbb{S}_{\mu}$ and  $a\in \mathbb{S}_{\la}$ if and only if $y^{-1}y'w=wa$ or equivalently $y'=ywaw^{-1}$, hence \eqref{indices} follows. Using the graded versions of translation functors as in \cite{Strgrad}, the proposition follows at least up to an overall shift.
 Since $\hat{\Delta}({\bf a})$ is a quotient of $\hat{P}({\bf a})\langle -\prod_{i=1}^rk_i(d_i-k_i)\rangle$, the additional shift appearing in \eqref{in} implies that $\hat{M}(\bf {a})$ appears as a quotient of ${}_k\hat\iota_{\bf d}(\hat{\Delta}(\bf {a}))$ which agrees with the fact that $b({\bf a})=0$.
\end{proof}

\begin{proof}[Proof of Theorem \ref{catJW}]
By Proposition \ref{acyclic}, the inclusion applied to a standard module is a module which has a graded Verma flag as in Proposition \ref{prop50}.
\end{proof}

\begin{prop}
\label{transfmatrix}
The standard module $\hat{{\Delta}}\left(k_1,d_1|k_2,d_2|\cdots |k_r,d_r\right)$ has a filtration with subquotients isomorphic to $\hat{\blacktriangle}\left(k_1,d_1|k_2,d_2|\cdots |k_r,d_r\right)$ such that in the Grothendieck space  $[\hat{\Delta}]={{d_1}\brack{k_1}}{{d_2}\brack{k_2}}\cdots{{d_r}\brack{k_r}}[\hat{\blacktriangle}]$.
\end{prop}

\begin{proof}
The existence of the filtration follows by the same arguments as in \cite[Theorem 2.16]{MS2}.
Note that the projective module $P(k_i|d_i)$ has $\binom{k_i}{d_i}$ occurrences of the simple module $L(k_i|d_i)$ in a composition series, \cite[4.13]{Ja}. Hence the module $\left(L(k_1|d_1)\boxtimes L(k_2|d_2)\boxtimes \cdots\boxtimes L(k_r|d_r)\right)$ occurs $\binom{k_1}{d_1}\binom{k_2}{d_2}\cdots\binom{k_r}{d_r}$ times as a composition factor in $P(k_1|d_1)\boxtimes P(k_2|d_2)\boxtimes \cdots\boxtimes P(k_r|d_r)$. Since parabolic induction is exact and the simple module in the head of a proper standard module appears with multiplicity one, this gives by Remark \ref{standardsinduced} precisely the number of proper standard modules appearing as subquotients in a filtration. The formula therefore holds when we forget the grading. The graded version will follow from the proof of Theorem \ref{Endstandards}.
\end{proof}

\begin{corollary}
\label{standards}
Under the isomorphism $\Phi_{\bf d}$ the standard basis corresponds to the standard modules:
$\hat\Delta\left(k_1,d_1|k_2,d_2|\cdots |k_r,d_r\right)\longmapsto v_{k_1}\otimes v_{k_2}\otimes \cdots \otimes v_{k_r}.$
\end{corollary}
\begin{proof}
By Theorem \ref{cattensor}, proper standard modules correspond to the dual standard basis.
Theorem \ref{transfmatrix} and the relationship between the dual standard basis and the standard basis imply the claim.
\end{proof}

The following result will be used later to compute the value of the categorified colored unknot. It exemplifies \cite[Theorem 6.3]{MS2} and connects with Section \ref{Euler}:
\begin{theorem}[Endomorphism ring of standard objects]\hfill\\
\label{Endstandards}
Let $\hat\Delta\left(k_1,d_1|k_2,d_2|\cdots |k_r,d_r\right)$ be a standard object in $A_{k,{\bf d}}-\gmod$.  Then there is a canonical isomorphism of rings
\begin{equation*}
\End_{A_{k,{\bf d}}}(\hat\Delta\left(k_1,d_1|k_2,d_2|\cdots |k_r,d_r\right))
\cong H^\bullet(Gr(k_1,d_1))\otimes \cdots\otimes H^\bullet(Gr(k_r,d_r)).
\end{equation*}
\end{theorem}

\begin{proof}
Abbreviate $\hat\Delta=\hat\Delta\left(k_1,d_1|k_2,d_2|\cdots |k_r,d_r\right)$ and let $L$ be its simple quotient.
By definition $\hat\Delta$ is a projective object in $A_{k,{\bf d}}^{\leq L}$, hence the dimension of its endomorphism ring equals $[\hat\Delta\::\:L]$, that is the number of occurrences of $L$ as a composition factor (i.e. as a subquotient in a Jordan-H\"older series). Note that $L$ occurs precisely once in the corresponding proper standard object $\hat\blacktriangle$, \cite[Theorem 2.16]{MS2}. Since $\hat\Delta$ has a filtration with subquotients isomorphic to $\hat\blacktriangle$, we only have to count how many such subquotients we need. This is however expressed by the transformation matrix from Proposition \ref{transfmatrix}. Hence the two rings in question have the same (graded) dimension and therefore isomorphic as graded vector spaces. To understand the ring structure we invoke the alternative definition of the fattened Vermas from Remark \ref{standardsinduced} which says that $\Delta$ is isomorphic to $U(\mathfrak{gl_n})\otimes_{\cU(\mathfrak{p})} \left(P(d_1)\boxtimes P(d_1)\boxtimes\cdots\boxtimes P(d_r)\right)$, where $P(d_i)$ is the antidominant projective module in $\cO_{k_i}(\mathfrak{gl}_{d_i})$ and $\mathfrak{p}$ is the parabolic in $\mathfrak{gl}_n$ containing $\mathfrak{gl}_{d_1}\oplus\cdots\oplus \mathfrak{gl}_{d_r}$. Now by Soergel's endomorphism theorem, \cite{Sperv}, we know that the endomorphism ring of $P(d_i)$ is isomorphic to $H^\bullet(Gr(k_i,d_i))$. Hence there is an inclusion of rings $H^\bullet(Gr(k_1,d_1))\otimes \cdots\otimes H^\bullet(Gr(k_r,d_r))\hookrightarrow \End_{A_{k,{\bf d}}}(\hat\Delta\left(k_1,d_1|k_2,d_2|\cdots |k_r,d_r\right))$. The statement follows.
\end{proof}

\begin{corollary}[Standard resolution of proper standard modules]\hfill\\
\label{standardres}
Let $\hat\Delta:=\hat\Delta\left(k_1,d_1|k_2,d_2|\cdots |k_r,d_r\right)$ be a standard object in $A_{k,{\bf d}}-\gmod$ and $\hat{\blacktriangle}$ the corresponding proper standard module. Then $\hat{\blacktriangle}$ has an infinite resolution by direct sums of copies of $\hat\Delta$. The graded Euler characteristic of $\Ext^*_{A_{k,{\bf d}}}(\hat{\blacktriangle}, \hat{\blacktriangle})$ equals $$\frac{1}{\left[{{d_1}\brack{k_1}}\right]\left[{{d_2}\brack{k_2}}\right]\cdots\left[{{d_r}\brack{k_r}}\right]}.$$
\end{corollary}

\begin{proof}
The main part of the proof is to show that the situation reduces to the one studied in Section \ref{Euler}. Let $\cF(\hat{\blacktriangle})$ be the full subcategory of ${A_{k,{\bf d}}}-\gmod$ containing modules which have a filtration with subquotients isomorphic to $\hat{\blacktriangle}\langle j\rangle$ with various $j\in \mZ$. Clearly, $\cF(\hat{\blacktriangle})$ contains $\hat{\blacktriangle}$, but also $\hat\Delta$. Now assume $M,N\in\cF(\hat{\blacktriangle})$ and $f:M\rightarrow N$ is a surjective homomorphism, then the kernel, $\op{ker}(f)$ is contained in $\cF(\hat{\blacktriangle})$: First of all it contains a filtration with subquotients isomorphic to various shifted proper standard objects. This follows easily from the characterization \cite[Proposition 2.13 (iv)]{MS2}, since given {\it any} dual standard module $\nabla$ we can consider the part of the part of the long exact Ext-sequence
$$\cdots\rightarrow\Ext^i_{A_{k,{\bf d}}}(M,\nabla)\rightarrow \Ext^i_{A_{k,{\bf d}}}(\op{ker}{f},\nabla)\rightarrow\Ext^{i+1}_{A_{k,{\bf d}}}(N,\nabla)\rightarrow\cdots$$
for any $i>0$. Then the outer terms vanish, and so does the middle. Since $[\op{ker}(f)]=[M]-[N]$ in the Grothendieck group, we deduce that all proper standard modules which can occur are isomorphic to $\hat{\blacktriangle}$ up to shift in the grading.

Now the projective cover in  ${A_{k,{\bf d}}}-\gmod$ of any  object $X$ in $\cF(\hat{\blacktriangle})$ is just a direct sum of copies of the projective cover $P$ of $\cF(\hat{\blacktriangle})$ which agrees with the projective cover of $\hat\Delta$. Now the definition of standard modules \eqref{defstandard} implies that any morphism $P\rightarrow X$ factors through  $\hat\Delta$. Altogether, we can inductively build a minimal resolution of $\hat{\blacktriangle}$ in terms of direct sums of copies (with appropriate shifts) of $\Delta$. This resolution lies actually in $A_{k,{\bf d}}^{\leq L}-\gmod$ and is a projective resolution there. Moreover, $\Ext^i_{A_{k,{\bf d}}}(\hat\Delta,X)=0$ for any $X\in\cF(\hat{\blacktriangle})$ and $i>0$, we can use this resolution to compute the graded Euler characteristic of the algebra of extensions we are looking for.

Now consider the functor
\begin{eqnarray*}
G:=\HOM_{A_{k,{\bf d}}}(\hat{\Delta},_-):\quad A_{k,{\bf d}}^{\leq L}-\gmod&\rightarrow&\Mod-\End(\hat\Delta)
\end{eqnarray*}
This functor is obviously exact and induced natural isomorphisms
$$\HOM_{A_{k,{\bf d}}}(Z_1,Z_2)\cong\HOM_{\End(\hat\Delta)}(GZ_1,GZ_2),$$
if $Z_1$ is just copies of (possibly shifted in the grading) $\hat\Delta$'s and $Z_2$ arbitrary.
Hence, by Theorem \ref{Endstandards}, the claim is equivalent to finding a projective (=free) resolution of the trivial one-dimensional $ H^\bullet(Gr(k_1,d_1))\otimes \cdots\otimes H^\bullet(Gr(k_r,d_r))$-module. Recall from Section \ref{Euler} that each factor is a complete intersection ring, hence obviously also the tensor product. Therefore we can directly apply the methods of that section and the statement follows.
\end{proof}

\begin{remark}{\rm
Note that the above corollary provides a infinite standard resolutions of proper standard objects (which are in principle computable). On the other hand, the minimal projective resolutions of standard objects are quite easy to determine and in particular finite. This step will be done in Theorem \ref{resolution}. This resembles results from \cite{W}. Combining the above we are able to determine a projective resolution of proper standard objects. This is necessary for explicit computations of our colored Jones invariants from \cite{SS} as well as the more general invariants constructed in \cite{W}.
}
\end{remark}

As a special case we get the following version of Soergel's endomorphism theorem:

\begin{corollary}
\label{HGr}
There is an equivalence $A_{k,{(n)}}-\gmod\cong H^\bullet(\op{Gr}(k,n))-\gmod$.
\end{corollary}
\begin{proof}
In this case we have precisely one simple object and the corresponding standard module is projective.
\end{proof}

\section{Decomposition into isotypic components, Lusztig's a-function and $\mathfrak{sl}_2$-categorification}
\label{filtrations}
The Jones-Wenzl projector projects onto a specific summand using the fact that the category of finite dimensional $\mathcal{U}_q$-modules is semisimple. However, their categorification is not. In this section we explain what remains of the original structure and connect it to the theory of Chuang and Rouquier on minimal categorifications of irreducible $\mathfrak{sl}_2$-modules.

The tensor product of irreducible representations can be decomposed as a sum of its weight spaces or as a direct sum of irreducible
representations which occur in it. In terms of categorification, we have already seen the weight space decomposition as a decomposition
of categories into blocks.  Here, we give a categorical analogue of the decomposition \eqref{isotypic} based on Gelfand-Kirillov dimension (which is directly connected with Lusztig's ${\bf a}$-function, see \cite[Remark 42]{MS4}). For the definition and basic properties of this dimension we refer to \cite{Ja}. The idea of constructing such filtrations is not new and was worked out earlier in much more generality for modules over symmetric groups and Hecke algebras (\cite[6.1.3]{GGOR}, \cite[7.2]{MS4}) and over Lie algebras \cite[Proposition 4.10]{CR}, \cite[Theorem 5.8]{Rou2}).

In this paper we want to describe this filtration in terms of the graphical calculus from \cite{FK} with an easy explicit formula for the Gelfand-Kirillov dimension of simple modules: basically it just amounts to counting cups in a certain cup diagram.

\subsection{Graphical calculus, GK-dimension and Lusztig's a-function}
\label{graphical}
Recall from \eqref{basissimples} the bijection between dual canonical bases elements of $V_1^{\otimes n}$ and isomorphism classes of simple objects in $\bigoplus_{k=0}^nA_{k,n}-\gmod$ up to grading shift. In \cite[p.444]{FK}, to each dual canonical basis was
assigned in a unique way an oriented cup diagram. This finally provides a graphical description of the dual canonical basis in $V_{\bf d}$. The cup diagram $\op{Cup}({\bf a})$ associated with $v_{\bf a}^\heartsuit$ is constructed as follows: First we turn the sequence ${\bf a}$ into a sequence $\la=\la({\bf a})$ of $\up$'s and $\down$'s by replacing $1$ by $\up$ and $0$ by $\down$. Then we successively match any $\up$, $\down$ pair (in this ordering), not separated by an unmatched $\up$ or $\down$, by an arc. We repeat this as long as possible. (Note that it is irrelevant for the result in which order we choose the pairs. It only matters that the pairs do not enclose an unmatched $\up$ or $\down$). Finally we put vertical rays for the unmatched $\up$'s and $\down$'s. The result is a cup diagram consisting of clockwise oriented arcs and oriented rays which do not intersect. Figure \ref{fig:CMs} displays the $10$ cup diagrams associated to sequences with $3$ $\up$'s and $2$ $\down$'s.
\begin{figure}
\includegraphics{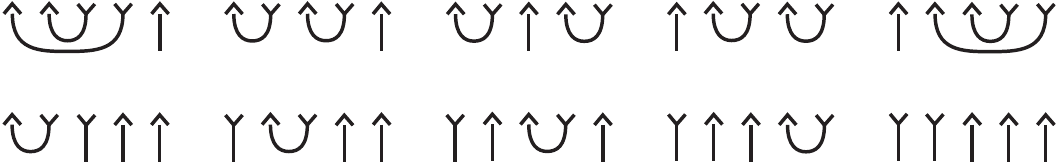}
\caption{Diagrams associated with simple modules. The Gelfand-Kirillov dimensions are $8,8,8,8,8,9,9,9,9,10$.}
\label{fig:CMs}
\end{figure}

\begin{prop}
\label{GK}
The Gelfand-Kirillov dimension of $L({\bf a})\in \cO(\mathfrak{gl}_n)$ is equal to
$$\GK(L({\bf a}))= n(n-1)/2 -c,$$
where $c$ denotes the number of cups in $\op{Cup}({\bf a})$.
\end{prop}

\begin{remark}
{\rm The simple modules with minimal Gelfand-Kirillov dimension correspond under Koszul duality to precisely the projective modules which are also tilting, \cite[Theorem 6.1]{BS2}.
}
\end{remark}

Before we start the proof we want to fix  a bijection between the set of longest coset representatives from $\mathbb{S}_n/\mathbb{S}_k\times \mathbb{S}_{n-k}$ and either of the following sets
\label{longestcosets}
\begin{itemize}
\item sequences $x_1>x_2>\ldots >x_k$ and $x_{k+1}>x_{k+2}>\cdots >x_n$ of distinct numbers $1\leq x_j\leq n$, ($1\leq j\leq n$) by mapping $w\in \mathbb{S}_n$ to the sequence with $x_j=w(j)$.
\item $\{\up,\down\}$-sequences of length $n$ with $k$ $\up$'s and $n-k$ $\down$'s by mapping $w\in \mathbb{S}_n$ to the sequence $\varphi(w)$ which has $\down$'s at the places $w(1), w(2),\ldots, w(k)$ and $\up$'s at all the other places.
\end{itemize}
In the situation above the longest increasing subsequence of $(w(1), \ldots, w(n))$ has length at most $2$ and the total number of such length $2$ subsequences equals the number of cups in $\op{Cup}(\varphi(w))$. On the other hand, the Robinson-Schenstedt algorithm (see e.g. \cite[4]{Fu}, \cite[5]{Ja}) associates to the sequence $(w(1), w(2),\cdots, w(n))$ a pair of standard tableaux of the same shape $Y(w)$.

\begin{ex}
Let $n=4$, $k=2$. Then the bijections from above identify the elements
\begin{equation}
\down\down\up\up\quad \down\up\down\up \quad \down\up\up\down\quad \up\down\down\up\quad \up\down\up\down\quad \up\up\down\down
\end{equation}
with the values $(x_1,x_2,x_3,x_4)=(w(1),w(2),w(3),w(4))$ in the following list:
\begin{equation}
(4,3,2,1)\quad (4,2,3,1)\quad (3,2,4,1)\quad (4,1,3,2)\quad (3,1,4,2)\quad (2,1,4,3)
\end{equation}
and the Young diagrams $Y(w)$ of the form
\begin{equation}
\Yvcentermath1\yng(1,1,1,1)\quad\quad\Yvcentermath1\yng(2,1,1)\quad\quad\Yvcentermath1\yng(2,1,1)\quad\quad  \Yvcentermath1\yng(2,1,1)\quad\quad\Yvcentermath1\yng(2,2)\quad\quad\Yvcentermath1\yng(2,2)
\end{equation}
\end{ex}

\begin{proof}[Proof of Proposition \ref{GK}]
Let $w\in\mathbb{S}_n$. Joseph's formula for the Gelfand-Kirillov dimension (see e.g. [Ja2, 10.11 (2)])
states that $$2\GK(L(w\cdot 0)) = n(n -1)-\mu_i(\mu_i-1),$$ where $\mu$ is the partition encoding the shape of the tableaux associated to  $(w(1), w(2),\cdots, w(n))$. If $w$ is a longest coset representative as above, then $Y(w)$ has at most two columns, the second of length $c$, the number of cups in $\op{Cup}(\varphi(w))$, and the first column of length $n-2c$, hence $\GK(L(w.0))=n(n-1)/2 -c$. Now it is enough to prove the following: let $\lambda$ be an integral dominant weight with stabilizer $\mathbb{S}_k \times \mathbb{S}_{n-k}$, then $$\GK(L(w.\lambda)) = \GK(L(w.0)).$$ To see this let $ T_0^{\lambda} \colon \mathcal{O}_0 \rightarrow \mathcal{O}_{\lambda} $ and $ T_{\lambda}^0 \colon \mathcal{O}_{\lambda} \rightarrow
\mathcal{O}_0 $ be translation functors to and out of the wall in the sense of \cite{Ja}, in particular $T_0^{\lambda}L(w.0)=L(w\cdot\la)$. By definition, translation functors do not increase the Gelfand-Kirillov dimension \cite{Ja}. Thus
$$ \GK(L(w.\lambda)) = \GK(T_0^{\lambda}L(w.0)) \leq \GK(L(w.0)). $$
Also, $\GK(L(w.\lambda)) \geq \GK(T_{\lambda}^0 L(w.\lambda)). $ Now $ L(w.0) $ is a submodule of $ T_{\lambda}^0 L(w.\lambda)$, as
$\Hom_{\mathfrak{g}}(L(w.0), T_{\lambda}^0 L(w.\lambda)) = \Hom_{\mathfrak{g}}(T_0^{\lambda} L(w.0), L(w.\lambda)) =
\Hom_{\mathfrak{g}}(L(w.\lambda), L(w.\lambda)) = \mathbb{C}$ using the adjunction properties of translation functors, and so there is a nontrivial map $ \Phi \colon L(w.0) \rightarrow
T_{\lambda}^0 L(w.\lambda) $ which is clearly injective, since $L(w.0)$ is a simple object. Thus $ \GK(L(w.0)) \leq \GK(T_{\lambda}^0 L(w.\lambda))$ (\cite[8.8]{Ja}). This gives $
\GK(L(w.\lambda)) \geq \GK(L(w.0))$ and the statement follows.
\end{proof}

The following is the a direct consequence of Proposition \ref{GK} and \cite[10.12]{Ja}:
\begin{corollary}
\label{GKbimod}
Let $\mu\in\mh^*$ be integral and dominant. If $\cL\left(M(\mu), L({\bf a})\right)\not=0$ then
$$\GK\left(\cL\left(M(\mu), L({\bf a})\right)\right) = 2 \GK(L(w.\lambda))= n(n-1) -2c.$$
\end{corollary}

\subsection{Filtrations}
Let $(\mathcal{O}_i(\mathfrak{gl}_n))^k$ denote the full subcategory of objects having Gelfand-Kirillov dimension at most $ n(n-1)/2 -
k$. Then there is a filtration of categories:
\begin{equation}
\label{isofilt}
\mathcal{O}_k(\mathfrak{gl}_n) = (\mathcal{O}_k(\mathfrak{gl}_n))^0 \supseteq (\mathcal{O}_k(\mathfrak{gl}_n))^1 \supseteq \cdots
\supseteq (\mathcal{O}_k(\mathfrak{gl}_n))^{\lfloor n/2 \rfloor}.
\end{equation}
which induces by \cite[8.6]{Ja} a corresponding filtration of Serre subcategories
\begin{equation}
\label{isofilt2}
A_{k,n}-\gmod= (A_{k,n}-\gmod)^0\supseteq (A_{k,n}-\gmod)^1\supseteq\cdots\supseteq (A_{k,n}-\gmod)^{\lfloor n/2 \rfloor}.
\end{equation}

\begin{theorem}[Categorical decomposition into irreducibles]\hfill
\label{filto}
\begin{enumerate}[(i)]
\item The category $\bigoplus_{k=0}^n (A_{k,n}-\gmod)^j$ is (for any $j$) stable under the functors $\hat{\mathcal{E}}$ and $\hat{ \mathcal{F}}$
\item There is an isomorphism of $\mathcal{U}_q$-modules
$$ \displaystyle \mathbb{C}(q) \otimes_{\mathbb{Z}[q,q^{-1}]} [\bigoplus_{k=0}^n(A_{k,n}-\gmod)^j] \cong
\bigoplus_{r=j}^{\lfloor n/2 \rfloor} {V}_{n-2r}^{\oplus b_{n-2r}} $$ where $ b_{n-2r} = \dim
\Hom_{\mathfrak{sl}_2}(\bar{V}_1^{\otimes n}, \bar{V}_{n-2r}). $
\item Set $r=\sum_{t=0}^{\lfloor n/2 \rfloor}b_{n-2t}$. The filtration \eqref{isofilt2} can be refined to a filtration
$$\{0\}=\cS_r\subset \cS_{r-1}\subset \cS_{r-2}\subset\cdots\subset\cS_{1}\subset\cS_{0}=\bigoplus_{k=0}^n A_{k,n}-\gmod$$ such that $\cS_{m+1}$ is a Serre subcategory inside $\cS_{m}$ for all $m$. With the induced additional structure from Remark \ref{rmCR}, the quotient  $\cS_{m}/\cS_{m+1}$ is, after forgetting the grading, an $\mathfrak{sl}_2$-categorification in the sense of \cite{CR}. It is isomorphic to a minimal one in the sense of \cite[5.3]{CR}.
\end{enumerate}
\end{theorem}

\begin{proof}
We formulate the proof in the ungraded case, the graded case follows then directly from the definitions. The first part is a consequence of the fact that tensoring with finite dimensional modules and projection onto blocks do not increase the Gelfand-Kirillov dimension (\cite[8.8]{Ja}).

Consider all simple modules in $\overline{S}_r:=\bigoplus_{k=0}^n\mathcal{O}_k(\mathfrak{gl}_n)$ of minimal GK-di\-men\-sion. By Proposition \ref{GK} these are precisely the ones corresponding to cup diagrams with the maximal possible number, say $c$, of cups.
Amongst these take the ones where $k$ is minimal, and amongst them choose $L({\bf a})$ minimal (in terms of the $\up$-$\down$-sequences a sequence gets smaller if we swap some $\up$ with some down $\down$, moving the $\up$ to the right.) Let $\langle L({\bf a})\rangle$ be the subcategory generated by $L({\bf a})$ (that means the smallest abelian subcategory containing $L({\bf a})$ and closed under extensions). Since simple modules have no first self-extensions,  $\langle L({\bf a})\rangle$ is a semisimple subcategory with one simple object. Proposition \ref{Grothgraded} and Lemma \ref{Grothungraded} in particular imply that the simple composition factors of $\cE_k L(\la)$ can be determined completely combinatorially using the formula \cite[p.445]{FK}. This formula says that given $\op{Cup}({\bf a})$, then the composition factors of  $\cE_k L({\bf a})$ are labeled by the same oriented cup diagram but with the rightmost down-arrow turned into an up-arrow or by cup diagrams with more cups. If neither is possible, then $\cE_k L(\la)$ is zero, see Figure \ref{fig:filtration}.

Consider the smallest abelian subcategory $\overline{\cS}_r$ of $\overline{\cS}_0$ containing $L({\bf a})$ and closed under the functorial $\mathfrak{sl}_2$-action. Now by \cite[7.4.3]{CR} and Remark \ref{rmCR}, the categorification of $\overline{V}_1^{\otimes n}$ can be refined to the structure of a strong  $\mathfrak{sl}_2$-categorification which induces such a structure of a strong $\mathfrak{sl}_2$-categorification on $\overline{\cS}_0$. From the combinatorics it then follows directly that this is a $\mathfrak{sl}_2$-categorification of a simple $\mathfrak{sl}_2$-module where the isomorphism class of $L({\bf a})$ corresponds to the canonical lowest weight vector. Since $L({\bf a})$ is projective in $\langle L({\bf a})\rangle$, hence also in $\overline{\cS}_r$, it is a minimal categorification by \cite[Proposition 5.26]{CR}. Note that it categorifies the $n-2c+1$-dimensional irreducible representation (with $c$ as above).
Next consider the quotient of $\overline{\cS}_0/\overline{\cS}_r$ and choose again some $L({\bf a})$ there with the above properties and consider the Serre subcategory $\overline{\cS_{r-1}}$ in $\overline{\cS}_0$ generated by $\overline{\cS}_r$ and this $L({\bf a})$. Then the same arguments as above show that $\overline{\cS}_{r-1}/\overline{\cS}_r$ is a minimal categorification. In this way we can proceed and finally get the result.
\end{proof}

In Figure \ref{fig:filtration} we display the simple modules from  $\bigoplus_{k=0}^n(\mathcal{O}_k(\mathfrak{gl}_4))$ in terms of their cup diagrams. The bottom row displays the subcategory in the bottom of our filtration, whereas the top row displays the quotient category in the top of the filtration. The red arrows indicate (up to multiples) the action of $\cE$ on the subquotient categories giving rise to irreducible $U_q(\mathfrak{sl}_2)$-modules. The black arrows give examples of additional terms under the action of $\cE$ which disappear when we pass to the quotient. Note that the filtration by irreducibles strictly refines the filtration given by the Gelfand-Kirillov dimension.

\begin{figure}[htb]
\includegraphics{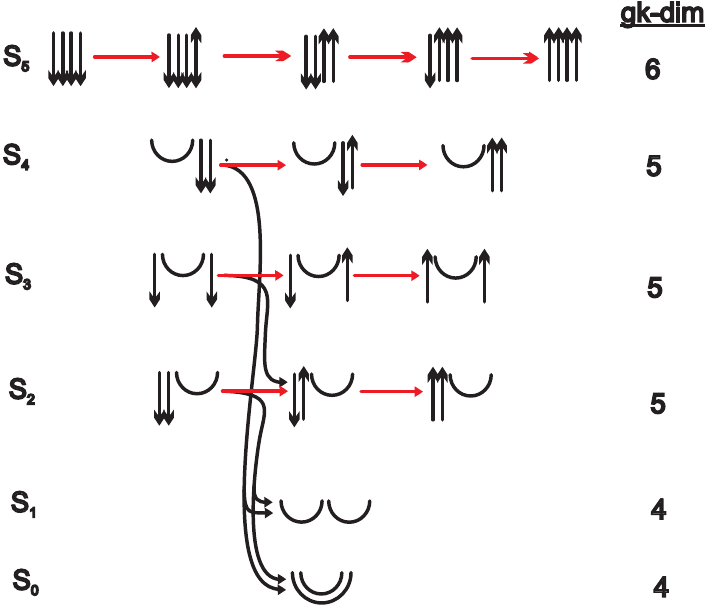}
\caption{The filtration of $V_1\otimes V_1\otimes V_1\otimes V_1$.}
\label{fig:filtration}
\end{figure}

\begin{remark}
{\rm The above filtration should be compared with the more general, but coarser filtrations \cite[Proposition 5.10]{CR} and \cite[7.2]{MS4}. Although the above filtrations carry over without problems to the quantum/graded version, the notion of  {\it minimal categorification} is not yet completely developed in this context, but see \cite{Rou2}.
}
\end{remark}

The above Theorem \ref{filto} generalizes directly to arbitrary tensor products, namely let $(A_{k,{\bf d}}-\gmod)^j$ be the Serre subcategory of $A_{k,{\bf d}}-\gmod$ generated by all simple modules corresponding to cup diagrams with at most $j$ cups, then the following holds:

\begin{theorem}
\begin{enumerate}[(i)]
\item The category $\bigoplus_{k=0}^n (A_{k,{\bf d}}-\gmod)^j$ is (for any $j$) stable under the functors $\mathcal{E}$ and $ \mathcal{F}$
\item There is an isomorphism of $\mathcal{U}_q$-modules
$$ \displaystyle \mathbb{C}(q) \otimes_{\mathbb{Z}[q,q^{-1}]} [\bigoplus_{k=0}^n(A_{k,{\bf d}}-\gmod)^j] \cong
\bigoplus_{t=j}^{\lfloor n/2 \rfloor} {V}_{n-2t}^{\oplus b_{n-2t}} $$ where $ b_{n-2t} = \dim
\Hom_{\mathfrak{sl}_2}(\bar{V}_{\bf d}, \bar{V}_{n-2t}). $
\item Set $r'=\sum_{t=0}^{\lfloor n/2 \rfloor}b_{n-2t}$. Then the filtration \eqref{isofilt2} induces a filtration which can be refined to a filtration $${0}\subset\cS_{r'}\subset \cS_{r'-1}\subset\cdots\subset\cS_{1}\subset\cS_{0}=\bigoplus_{k=0}^n A_{k,{\bf d}}-\gmod$$ such that $\cS_{m-1}$ is a Serre subcategory inside $\cS_{m}$ for all $m$ and, after forgetting the grading and with the induced additional structure from Remark \ref{rmCR}, the quotient  $\cS_{m}/\cS_{m-1}$ is an $\mathfrak{sl}_2$-categorification in the sense of \cite{CR}, isomorphic to a minimal one in the sense of \cite[5.3]{CR}.
\end{enumerate}
\end{theorem}

The case of $V_2\otimes V_2$ is displayed in Figure \ref{fig:HCfiltration} (note that then only cups which connect one of the first two points with one of the last two points are allowed.)

\begin{figure}[htb]
\includegraphics{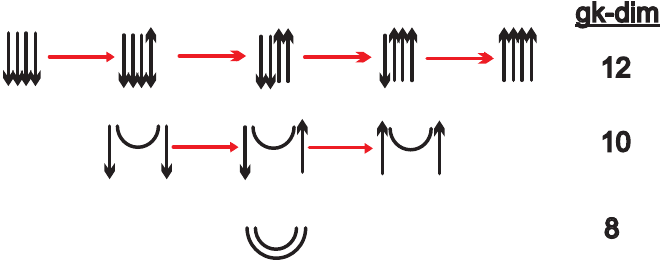}
\caption{The filtration of $V_2\otimes V_2$.}
\label{fig:HCfiltration}
\end{figure}

\section{Categorified characterization of the Jones-Wenzl projector}
In this section we give a categorical version of the characterizing property of the Jones-Wenzl projector. We first recall the special role of the cup and cap functors and put them into the context of Reshetikhin-Turaev tangle invariants. The first categorification of Reshetikhin-Turaev tangle invariants was constructed in \cite{StDuke}. The main result there is the following

\begin{theorem}{\rm(\cite[Theorem 7.1, Remark 7.2]{StDuke})}
\label{catjones} Let $ T $ be an oriented tangle from $ n $ points to $ m $ points.  Let $ D_1 $ and $ D_2 $ be two tangle diagrams of $
T. $ Let
$${\Phi}(D_1),{\Phi}(D_2)\colon \quad D^b\left(\bigoplus_{k=0}^n A_{k,n}-\gmod\right) \rightarrow
D^b\left(\bigoplus_{k=0}^m A_{k,m}\right)-\gmod $$ be the corresponding functors associated to the oriented tangle. Then there is an isomorphism of functors $
{\Phi}(D_1)\langle 3 \gamma(D_1) \rangle \cong {\Phi}(D_2)\langle 3 \gamma(D_2) \rangle. $
\end{theorem}

We have to explain briefly how to associate a functor to a tangle. This is done by associating to each elementary tangle (cup, cap, braid) a functor and checking the Reidemeister moves. To a braid one associates a certain derived equivalence whose specific construction is irrelevant for the present paper. We however introduce the cup and cap functors, since they are crucial.

Note that we work here in a Koszul dual picture of the one developed in \cite{StDuke}, since we focus more on the quantum group action which is easier in our context. The translation between these two picture is given by the result in \cite[Theorem 35, Theorem 39]{MOS} relating the corresponding functors via the Koszul duality equivalence of categories.

\subsection{Functors associated to cups and caps}
\label{catcupscaps}
In the following we briefly recall the definition of the functors and the main properties which will be used later, illustrated by a small example. For each $1\leq i\leq n$ we will define now functors which send a module to its maximal quotient which has only composition factors from a certain allowed set:

Given such $i$ consider the set $S$ of isomorphism classes of simple $A_{k,n}$-modules  ${L}(a_1,a_2,\ldots, a_n)$ where the sequence ${\bf a}=a_1,a_2,\ldots, a_n$ is obtained from the sequence $1^k0^{n-k}$ by applying an element $w\in \mathbb{S}_n$ which is a shortest coset representative in $\mathbb{S}_{i}\times \mathbb{S}_{n-i}\backslash \mathbb{S}_n$ such that each other element in the same $\mathbb{S}_n/\mathbb{S}_k\times \mathbb{S}_{n-k}$-coset has again this property (see e.g. \cite{MSslk} for an interpretation in terms of tableaux combinatorics). Let $A^i_{k,n}-\Mod$ be the full subcategory of $A_{k,n}-\Mod$ containing only modules with simple composition factors from the set $S$. There are the natural functors
\begin{eqnarray*}
\epsilon_i:\quad A^i_{k,n}-\gmod\rightarrow A_{k,n}-\gmod && Z_i:\quad A_{k,n}-\gmod\rightarrow A^i_{k,n}-\gmod,
\end{eqnarray*}
of inclusion of the subcategory respectively of taking the maximal quotient contained in the subcategory. Note that $Z_i$ is left adjoint to $\epsilon_i$.

The category $A^i_{k,n}-\Mod$ is a graded version of the so-called parabolic category $\cO$ defined as follows: let $\mathfrak{p}_i$ be the {\it $i$-th minimal parabolic} subalgebra of $ \mathfrak{g} $ which has basis the matrix units $E_{r,s}$, where $s\geq r$ or $r,s  \in \lbrace i, i+1 \rbrace$. Now replace locally $\mb$-finiteness in Definition \ref{defO} by locally $ \mathfrak{p}_i$-finiteness and obtain the {\it parabolic category} $ \mathcal{O}_k^i(\mathfrak{gl}_n)$, a full subcategory of $\cO_k(\mathfrak{gl}_n)$, see \cite[Section 9.3]{Hu}).  We have  $ \mathcal{O}_k^i(\mathfrak{gl}_n)\cong A^i_{k,n}-\op{mod}$. In this context $Z_i$ is the {\it Zuckerman functor} of taking the maximal locally finite quotient with respect to $ \mathfrak{p}_i$. That means we send a module $M\in\cO_k(\mathfrak{gl}_n)$ to the largest quotient in $\mathcal{O}_k^i(\mathfrak{gl}_n)$. A classical result of Enright and Shelton \cite{ES} relates parabolic category $\cO$ with non-parabolic category $\cO$ for smaller $n$. This equivalence had been explained geometrically in \cite{SoergelES}, and was lifted to the graded setup in \cite{Rh}.

\begin{prop}
Let $n\geq 0$. There is an equivalence of categories $$\zeta_n \colon
    \mathcal{O}_k
    (\mathfrak{gl}_n) \rightarrow \mathcal{O}_{k+1}^1 (\mathfrak{gl}_{n+2})$$ which can be lifted to an equivalence $A_{k,n}-\gmod\cong A^1_{k+1,n+2}-\gmod$, where for $n=0$ the corresponding category is equivalent to the category of graded vector spaces.
\end{prop}

Now there are functors (up to shifts in the internal and homological degree) pairwise adjoint in both directions
\begin{eqnarray*}
&{{\cap}}_{i,n} \colon\quad D^b(A_{k,n}-\gmod) \rightarrow D^b(A_{k,n-2}-\gmod)&\\
&{\zeta_n}^{-1} \hat{Z}_1 \circ \hat\epsilon_2 \hat{Z}_2 \circ \cdots \circ \hat\epsilon_{i}\circ L\hat{Z}_i.&\\
&{{\cup}}_{i,n}\colon\quad D^b(A_{k,n}-\gmod) \rightarrow D^b(A_{k,n+2}-\gmod)& \\
&\hat\epsilon_i\hat{Z}_i \circ \cdots \hat\epsilon_2\circ L\hat{Z}_2 \circ \hat\epsilon_1\circ \zeta_n.
\end{eqnarray*}
where we denote $\hat\epsilon_i=\mathbb{L}\epsilon_i$, the standard lift of the inclusion functor compatible with \eqref{defcupcap} and the lift $\hat{Z_i}=\mathbb{L}{Z_i}\lsem1\rsem\langle -1\rangle$ of the Zuckerman functor.

The following theorem means that these functors provide a functorial action of the Temperley-Lieb category:

\begin{theorem}
Let $ j \geq k$.  There are isomorphisms
\begin{enumerate}
\item $ \hat{\cap}_{i+1,n+2} \hat{\cup}_{i,n} \cong \hat{\Id} $ \item $ \hat{\cap}_{i,n+2}
    \hat{\cup}_{i+1,n} \cong \hat{\Id} $ \item $ \hat{\cap}_{j,n} \hat{\cap}_{i,n+2} \cong
    \hat{\cap}_{i,n} \hat{\cap}_{j+2,n+2} $ \item $ \hat{\cup}_{j,n-2} \hat{\cap}_{i,n} \cong
    \hat{\cap}_{i,n+2} \hat{\cup}_{j+2,n} $ \item $ \hat{\cup}_{i,n-2} \hat{\cap}_{j,n} \cong
    \hat{\cap}_{j+2,n+2} \hat{\cup}_{i,n} $ \item $ \hat{\cup}_{i,n+2} \hat{\cup}_{j,n} \cong
    \hat{\cup}_{j+2,n+2} \hat{\cup}_{i,n} $ \item $ \hat{\cap}_{i, n+2} \hat{\cup}_{i,n} \cong
    \hat{\Id}\lsem1\rsem\langle 1 \rangle \bigoplus \hat{\Id}\lsem-1\rsem\langle -1 \rangle. $
\end{enumerate}
of graded endofunctors of $\oplus_{k=0}^n A_{k,n}-\gmod$. In the Grothendieck group, $ [\hat{\cap}_{i,n}]=\cap_{i,n} $ and $ [\hat{\cup}_{i,n}]=\cup_{i,n} $.
\end{theorem}

\begin{proof}
The first part was proven in the Koszul dual case in \cite[Theorem 6.2]{StDuke} and then holds by \cite[Section 6.4]{MOS}. The second part follows directly from \cite[Proposition 15]{BFK} and Lemma \ref{ZonVermas} below.
\end{proof}

The following Lemma is the main tool in computing the functors explicitly and categorifies \eqref{defcupcap}:
\begin{lemma}
\label{ZonVermas}
Let $\hat M({\bf a})\in A_{k,n}-\gmod$ be the standard graded lift of the Verma module $M({\bf a})\in\cO_k(\mathfrak{gl}_n)$. Let ${\bf b}$ be the sequence ${\bf a}$ with $a_i$ and $a_{i+1}$ removed. Then there are isomorphisms of graded modules
\begin{eqnarray}
\mathbb{L}Z_i\hat M(\bf a)
&\cong&
\begin{cases}
0&\text{if $a_i=a_{i+1}$}\\
\hat M({\bf b})\langle-1\rangle\lsem-1\rsem\in A_{k,n-2}-\gmod&\text{if $a_i=1$, $a_{i+1}=0$}\\
\hat M({\bf b})\in A_{k,n-2}-\gmod&\text{if $a_i=0$, $a_{i+1}=1$}
\end{cases}
\end{eqnarray}
Whereas $\hat\epsilon_i(\hat M({\bf b}))$ is quasi-isomorphic to a complex of the form
$$\cdots\quad0\longrightarrow\hat M({\bf c})\langle 1\rangle\longrightarrow\hat M({\bf d})\longrightarrow0\quad\cdots$$
where ${\bf c}$ and ${\bf d}$ are obtained from ${\bf b}$ by inserting $01$ respectively $10$ at the places $i$ and $i+1$.
\end{lemma}
\begin{proof}
This follows directly from \cite[Theorem 8.2, Theorem 5.3]{Strgrad} and Koszul duality \cite[Theorem 35]{MOS}.
\end{proof}

\begin{ex}{\rm Consider the special case $n=2$, $k=1=i$. Then $A_{1,1}^1-\mod$ has one simple object, ${L}(10)$.
Note that $\hat{L}(10)$ has a resolution of the form $\hat{M}(01)\langle 1\rangle\rightarrow M(10)$. In particular, $\hat{L}(10)$ is quasi-isomorphic to this resolution and so the functor $\hat{\epsilon}_1$ categorifies the first formula in \eqref{defcupcap}. On the other hand $\mathbb{L}{Z}_1\hat{M}(10)={Z}_1\hat{M}(10)$, since $M(10)$ is projective. Furthermore, $Z_1\hat{M}(10)=L(10)$. To compute $\mathbb{L}{Z}_1\hat{M}(10)$ we apply he functor to a projective resolution and get the complex ${Z}_1\hat{P}(10)\langle 1\rangle \rightarrow Z_1\hat{P}(01)$ isomorphic to $M(01)\langle 1\rangle$ shifted in homological degree by $1$. Hence $\hat{Z_1}$ categorifies the second formula in \eqref{defcupcap}.
}
\end{ex}

The following should be viewed as a categorification of Proposition \ref{charJW} and will be used in \cite{SS} in the proof of the categorification of the colored Jones polynomial to show that one can slide projectors along strands.
\begin{theorem}
\label{charJWcat}
Assume $F:\;\bigoplus_{k=0}^nA_{k, n}-\gmod\rightarrow \bigoplus_{k=0}^nA_{k, (n)}-\gmod$ is a non-zero exact functor such that
\begin{enumerate}[(i)]
\item $Q:=(\oplus_{k=0}^n\mathbb{L}({}_k\hat\iota_{n})F$ is an idempotent (i.e. $Q^2\cong Q$),
\item $F$ commutes with the $U_q(\mathfrak{sl}_2)$-action (i.e. $FH\cong HF$, where $H=\hat{\mathcal{E}}, \hat{\mathcal{F}}, \hat{\mathcal{K}}$).
\item $F\epsilon_i=0$ for $1\leq i\leq n$,
\item $\mathbb{L}Z_iQ=0$ for $1\leq i\leq n$.
\end{enumerate}
Then $F\cong\oplus_{k=0}^n{}_k\hat\pi_{n}$.
\end{theorem}
\begin{proof}
First we prove that  $G:=\oplus_{k=0}^n{}_k\hat\pi_{n}$ has the required properties. The first two properties are clear by Theorem \ref{catJW}, and the third or fourth can be checked on the Grothendieck group. By Theorem \ref{catJW} this follows then from Proposition \ref{charJW}. For the converse let $L=\hat{L}({\bf a})$ be a simple object in $A_{k, n}-\gmod$. Without loss of generality concentrated in degree zero.
Then $QL\not=0$  is equivalent to $FL\not=0$. If we have now a simple module of the form $L:=L({\bf a})$ where ${\bf a}$ contains two consecutive numbers $a_i=1$, $a_{i+1}=0$ for some $i$ not of the form $d_1+d_2+d_3+\cdots +d_s$ for $1\leq s\leq r$, then $L$ is in the image of $\epsilon_i$ (\cite[Proposition 5.5]{ES}) and so annihilated by $F$ by the third property. Since $F$ commutes with the $U_q(\mathfrak{sl}_2)$-action, its restriction $F_k$ of $F$ to $A_{k, n}-\gmod$ has image in $A_{k, {\bf d}}-\gmod$. Since $Q$ is an idempotent and $F$ non-zero, we necessarily have $F_0={}_0\hat\pi_{\bf d}$. Thanks to Theorem \ref{filto} we see that $F$ annihilates the subcategory $S_1$ (that is rows $2-5$ in Figure \ref{fig:filtration}) that is all modules which are not of maximal Gelfand-Kirillov dimension. Using again the $U_q(\mathfrak{sl}_2)$-action, it follows that $F$ sends all the other simple modules to the corresponding simple module in $A_{k, {\bf d}}-\gmod$. This property together with the exactness determines $F$ uniquely up to isomorphism by \cite[p.90]{Soergel-HC}.
\end{proof}

\subsection{A simplified categorification and connection with Khovanov homology}
\label{normalizedJW}
Instead of working with fractional Euler characteristics one could renormalize the Jones-Wenzl projector $p_n$ for instance by multiplication with $[n]!$. This has the disadvantage that the projection is not an idempotent anymore. It might be possible to use this to categorify a renormalized Jones polynomial which would be much easier to compute than ours, but there is no hope this might be a functor valued colored tangle invariant. (Note that the categorification of the Jones polynomial presented in \cite{Khovcolor} has unfortunately similar deficiencies.)

\subsection{Example: Categorification of the colored unknot}

\subsubsection{The unknot for color $V_1$}
Let $L$ be the unknot. The Reshetikhin-Turaev invariant associated to $L$ (colored by the vector representation $V_1$) can be obtained as follows: we consider the $\mathcal{U}_q(\mathfrak{sl}_2)$ homomorphism $V_0\rightarrow V_1\otimes V_1 \rightarrow V_0$ given by coevaluation followed by evaluation, explicitly
$$1\mapsto v_1\otimes v_0-q v_0\otimes v_1\mapsto (q^{-1}+q)1$$
using the formulas \eqref{defcupcap}. Hence $q^{-1}+q$ is the (renormalized) Jones polynomial associated to the unknot colored by $V_1$. Now the associated functor $F(L)$ is $\mathbb{L}\hat{Z}_1\hat\epsilon_1$ which by Lemma \ref{ZonVermas} first produces a complex of length two
$$ \cdots\quad0\longrightarrow\hat M(01)\langle 1\rangle\longrightarrow\hat M(10)\longrightarrow0$$
with two graded Verma modules which then both get sent to a graded lift of trivial module $\hat{L}(10)$, but shifted in homological degree,
\begin{eqnarray}
\label{sum}
\hat{L}(10)\lsem-1\rsem\langle -1\rangle\oplus \hat{L}(10)\lsem1\rsem\langle 1\rangle.
\end{eqnarray}
Altogether we therefore considered a functor from a semisimple graded category (namely the graded version of the category of finite dimensional $\mathfrak{gl}_2$-modules with trivial central character) with unique simple object $\hat{L}(10)$ up to grading shift which then is mapped under $F(L)$ to \eqref{sum} which gives $-(q+q^{-1})[\hat{L}(10)]$ in the Grothendieck group. Hence it categorifies $-(q+q^{-1})$ in a quite trivial way.
On the other hand, we could realize $L$ as a pairing of a cup with a cap. The evaluation form from \eqref{scalar} gets categorified by the bifunctor $\Ext^*_{A_{1,1}-\gmod}(_-._-)$, see Proposition \ref{evaluationformcat} below; we have $\Ext^*_{A_{1,1}-\gmod}(\hat{L}(10),\hat{L}(10))\cong\mC[x]/(x^{-1})$ categorifying $1+q^{2}$.

\subsubsection{The colored unknot for color $V_2$}
We illustrate the categorification of the colored unknot with color $V_2$. The general case is similar and is sketched below. Cabling the unknot means that we first apply two nested cup functors to $\mC$. This produces the simple module $\hat{L}(1100)$. Applying the Jones-Wenzl projector yields the simple module $\hat{L}(2,2|0,2)$.  The evaluation from \eqref{scalar} gets categorified by the bifunctor $\Ext^*_{A_{2,(2,2)}}(_-,_-)$ up to a grading shift. Hence we compute
$\Ext^*_{A_{2,(2,2)}}(\hat{L}(2,2|0,2),\hat{L}(2,2|0,2))$.
Lemma \ref{ZonVermas} tells as, that $\hat{L}(1100)$ has a resolution of the form
\begin{eqnarray}
\hat{M}(0011)\langle 2\rangle\rightarrow\hat{M}(0101)\langle 1\rangle\oplus\hat{M}(1010)\langle 1\rangle\rightarrow \hat{M}(1100)
\end{eqnarray}
By applying the Jones-Wenzl projector and Theorem \ref{cattensor} we obtain a corresponding resolution of $\hat{L}(2,2|0,2)$ in terms of proper standard modules
\begin{eqnarray}
\label{resunkno}
\hat{\blacktriangle}(0,2|2,2)\langle 2\rangle\rightarrow\hat{\blacktriangle}(1,2|1,2)\langle 3\rangle\oplus\hat{\blacktriangle}(1,2|1,2)\langle 1\rangle\rightarrow \hat{\blacktriangle}(2,2|0,2)
\end{eqnarray}
On the other hand, $\hat{L}(2,2|0,2)$ is self-dual, hence dualizing the complex gives a coresolution in terms of dual proper standard modules $\op{d}\hat{\blacktriangle}({\bf a})=:\hat{\blacktriangledown}({\bf a})$.

Unfortunately, proper standard modules and their duals do not form dual families with respect to our semi-linear Ext-pairing. However, the standard modules and the dual proper standard modules form such a pairing,
\begin{equation}
\label{evaluationform}
\Ext^i_{A_{k,\bf d}}\left(\hat{\Delta}({\bf a})\langle\prod_{i=1}^rk_i(d_i-k_i)\rangle,\hat{\blacktriangledown}({\bf b})\right)
=
\begin{cases}
0 &\text{ if $i\not=0$,}\\
\delta_{{\bf a},{\bf b}}\mC &\text{ if $i=0$.}
\end{cases}
\end{equation}

Therefore we want to replace \eqref{resunkno} by a resolution in terms of standard objects. Note that the outer terms are already standard modules, $$\hat{\blacktriangle}(0,2|2,2)=\hat{\Delta}(0,2|2,2),\quad\quad\hat{\blacktriangle}(2,2|0,2)=\hat{\Delta}(2,2|0,2).$$
Combinatorially we have $[\hat{\blacktriangle}(1,2|1,2)]=\frac{1}{[2][2]}[\hat{\Delta}(1,2|1,2)]$ and we have the standard resolution from Corollary \ref{standardres}.

Now we can compute the graded Euler characteristic of our Ext-algebra by pairing with the dualized complex \eqref{resunkno}.
Then $\hat{\Delta}(2,0|2,2)$ contributes a morphism of internal and homological degree zero, whereas $\hat{\Delta}(0,2|2,2)$ contributes an extension of homological and integral degree $4$, hence together $1+q^4$.
On the other hand,  $\hat{\Delta}(1,2|1,2)$ contributes $\frac{1}{[2][2]}q^{-2}(q+q^3)^2=q^2$. Therefore, the graded Euler characteristics of the Ext-algebra equals
$$1+q^2+q^4=\lsem3\rsem.$$

The value of the unknot differs now only by a total shift in the grading. Since the cup functor is right adjoint to the cap functor up to the shift $\langle -1\rangle$, and there are two cups, we finally get $q^{-2}\lsem3\rsem=q^{-2}(1+q^2+q^4)=[3]$ as the graded Euler characteristic of the categorified unknot colored by $2$.

\subsubsection{The colored unknot for arbitrary color}
In general, we have to replace \eqref{resunkno} by a resolution with $2^n$ Verma modules following the recipe indicated in Figure \ref{fig:unknot} for $n=3$.

\begin{figure}[htb]
\begin{eqnarray*}
&\xymatrix{
& \hat{M}(100110) \langle 2 \rangle \ar[r] \ar[rd] & \hat{M}(110100) \langle 1
\rangle \ar[rd] &\\
\hat{M}(000111) \langle 3 \rangle \ar[r] \ar[ru] \ar[rd] & \hat{M}(010101) \langle
2 \rangle \ar[ru] \ar[rd] & \hat{M}(101010) \langle 1 \rangle \ar[r] &
\hat{M}(111000)\\
& \hat{M}(001011) \langle 2 \rangle \ar[r] \ar[ru] & \hat{M}(011001) \langle 1
\rangle \ar[ru]&\\}&\\
&\text {Contributions to the graded Euler characteristics:}&\\
\\
&q^6 \quad\quad\quad\quad \frac{q^{-4}(q^2+q^4+q^6)^2}{[3][3]}=q^4 \quad\quad\quad\quad
\frac{q^{-4}(q+q^3+q^5)^2}{[3][3]}=q^2 \quad\quad\quad\quad 1&
\end{eqnarray*}
\caption{The resolution computing the colored unknot.}
\label{fig:unknot}
\end{figure}

\section{Categorified evaluation form and pairing}
A comparison of \eqref{evaluationform} with \eqref{scalar} using Theorem \ref{cattensor} and Corollary \ref{standards} implies
\begin{prop}
\label{evaluationformcat}
The evaluation form \eqref{scalar} is categorified via the bifunctor
$$\Ext^i_{A_{k,\bf d}}(_-,_-).$$
\end{prop}

The semilinear form $\langle , \rangle \colon V_n \times V_n \rightarrow \mathbb{C}(q) $ from \eqref{scalarprod} can be categorified via the following bifunctor

\begin{lemma}
\label{bilform}
The anti-bilinear form  $( _-,_-)$  is categorified by
\begin{eqnarray}
([M], [N])=\sum_{i,j}(-1)^j\DIM\HOM(M, \op{d}N\langle i\rangle[j])\; q^i,
\end{eqnarray}
\end{lemma}

\begin{proof}
In \cite[Theorem 5.3 (d)]{FKS}, the bilinear form  $\langle _-,_-\rangle'$  was categorified by
\begin{eqnarray}
\label{bilformcat}
\langle [M], [N]\rangle=\sum_{i,j}(-1)^j\DIM\HOM(\op{d}N, {\hat T_{w_0}}M\langle i\rangle)[j])\; q^i,
\end{eqnarray}
where $M$, $N$ are objects in $\oplus_{k=0}^n A_{\bf d}-\gmod$ or in the bounded above derived category, ${\hat T_{w_0}}$ is the categorification of $\tilde{\Pi}_{w_0}$ and $\op{d}$ is the graded duality categorifying $D$ from Remark \eqref{identifications}. Applying the duality to \eqref{bilformcat}, we take morphisms from $X:=\op{d}{\hat T_{w_0}M}$ to $Y:=\op{d}\op{d}N$. Under the identifications from Remark \ref{identifications}, this is equivalent to taking morphisms from $M$ to $\op{d}N$. Hence the lemma follows.
\end{proof}

\section*{Part III}
\section{Categorifications of the $3j$-symbol}
In this section we present three categorifications of the $3j$-symbols corresponding to their three different descriptions.

\subsection{In terms of an abstract graded vector space of extensions}
Recall from Section ~\ref{comb} the intertwiner $A_{i,j}^k \colon V_i \otimes V_j \rightarrow V_k$ with the associated functor
$\hat{A}_{i,j}^k := \hat{\pi}_{k} \circ \hat{\Phi}_{i,j}^k \circ \mathbb{L}\hat{\iota}_{i,j}$, and the standard objects \eqref{defstandard} corresponding to the standard basis by Corollary \ref{standards}.
It follows now directly from Lemma \ref{bilform} that the $3j$-symbol can be realized as a graded Euler characteristics as follows:

\begin{theorem}
\label{cat3j}
Let $i,j,k,r,s,t$ be natural numbers such that  $C_{i,j}^k(r,s,t)$ is defined. Then
$$C_{i,j}^k(r,s,t) = \sum_{\alpha} \sum_{\beta} (-1)^{\beta} \dim\left(\Ext_{A_{k,(k)}}^{\beta}\left(\hat{\Delta}(t|k),
\hat{A}_{i,j}^k \hat\Delta(r,i|s,j) \langle -\alpha \rangle\right))\right) q^{\alpha}.$$
\end{theorem}

\subsection{In terms of a complex categorifying the triangle counting}
Note that under the equivalence of Corollary \ref{HGr}, the standard module $\hat{\Delta}(t|n)$ is mapped to the projective (free) module $H^\bullet(\op{Gr}(t,n))$. To describe a categorification of the triangle counting from Theorem \ref{cat3j} it is enough to
describe the complex $\hat{A}_{i,j}^k \hat\Delta(r,i|s,j)$ of graded  $H^\bullet(\op{Gr}(t,n))$-modules and then take morphisms from $H^\bullet(\op{Gr}(t,n))$ to it. The result is a complex $K_\bullet=K_\bullet(i,j,k,r,s,t)$ of graded vector spaces and we are interested in its graded Euler characteristics.

\begin{theorem}[Categorified triangle counting]\hfill\\
\label{trianglecat}
Each graded vector space $K_i$ in $K_\bullet(i,j,k,r,s,t)$ comes along with a distinguished homogeneous basis $\cB_i$ labeled by certain weighted signed triangle arrangements such that the weight agrees with the homogeneous degree and the sign agrees with the parity of the homological degree and such that the elements from $\bigcup_i\cB_i$ are in bijection to all possible signed weighted $(i,j,k,r,s,t)$-arrangements.
In particular, its graded Euler characteristic agrees with the alternating graded dimension formula in Theorem \ref{cat3j}.
\end{theorem}

\begin{proof}
We verify the theorem by showing that a systematic counting of the triangles from Section \ref{comb} coincides with a step by step construction of the complex:
\begin{itemize}
\item
Start with a standard module  $\hat{\Delta}(r,i|s,j)$ and apply the functor $\mathbb{L}\hat{\iota}_{(i,j)}$. The result will be a module (Proposition \ref{acyclic}) with a graded Verma flag given by the combinatorics in \eqref{in}.
\item The usual passage from short exact sequences to long exact homology sequences allows us to restrict ourselves to understanding the functors on each graded Verma module appearing in the filtration.
\item Fix such a graded Verma module. This corresponds then to fixing the arrows at the two bottom sides of the triangle in Figure~\ref{fig:example3j}. Applying the cap functors maps such a Verma module either to zero or to a Verma module shifted by one in the homological degree, as described in Lemma~\ref{ZonVermas}.
\item In our combinatorial picture it survives precisely when the cap creates an oriented arc. The highest weight of the result can be read of the arrow configuration at the top of the triangle. Applying then the projector $\hat{\pi}_{(i,j)}$ maps each Verma module to a simple module which we identify with the unique simple 1-dimensional module for the corresponding ring $H^\bullet(\op{Gr}(t,k))$. Hence if the result is non-trivial, it corresponds to precisely one signed oriented line arrangement from Theorem \ref{ClebschGordon} and each signed oriented line arrangement is obtained precisely once.
\item The signs match the homological degree shifts, whereas the $q$-weights match the grading degree shifts.
\end{itemize}
Hence, our diagram counting corresponds precisely to the basis vectors of a complex of graded vector spaces whose graded Euler characteristic agrees with the alternating graded dimension formula in Theorem \ref{cat3j}.
\end{proof}

\subsection{In terms of a positive categorification using the basis of projectives}
Instead of working in the basis of standard modules, we could work in the basis of projectives. In this case the categorification is easier and explains the positive formulas from  Theorem \ref{positivity}.
Let $\hat{P}(r,i|s,j)$ be the projective cover of  $\hat{\Delta}(r,i|s,j)$.
The binomial coefficients arising in Theorem \ref{positivity} gets categorified in terms of cohomologies of Grassmannians:

\begin{theorem}[Integrality]
Let  ${D}_{i,j}^k(r,s,t)$ be as in Theorem \ref{positivity}. Then
\begin{eqnarray}
&&{D}_{i,j}^k(r,s,t)\nonumber\\
&=&\sum_{\alpha} \dim \Hom_{H^*(\op{Gr}(t,k))-\gmod}(H^*(\op{Gr}(t,k)),\hat{A_{i,j}^k} \hat{P}(r|i,s|j) \langle \alpha \rangle) q^{\alpha}\nonumber\\
&=&\sum_{\alpha} \dim \hat{A_{i,j}^k} \hat{P}(r,s) \langle -\alpha \rangle q^{\alpha}.
\label{Dformula}
\end{eqnarray}
Hence ${D}_{i,j}^k(r,s,t)$ is the graded dimension of the vector space $\hat{A_{i,j}^k} \hat{P}(r,s)$.
\end{theorem}

\begin{proof}
The second equation follows obviously from the first. By Proposition \ref{twistedcan} below projective modules correspond to the twisted canonical basis. Now $\hat{A_{i,j}^k} $ is by definition a composition of functors, where we first have a graded lift of the Bernstein-Gelfand functor from the Harish-Chandra category to category $\mathcal{O}$, taking projective objects to projective objects, then a graded Zuckerman functor which takes projective objects to projective objects in the parabolic
subcategory.  The next part of the functor is the Enright-Shelton equivalence of categories which again takes projective objects to projective objects. The last part of the functor is at least exact on the abelian category. Hence $\hat{A_{i,j}^k} \hat{P}(r,s)$ is a module (that means a complex concentrated in degree zero).
Using Lemma \ref{bilform}, we evaluate the bilinear form categorically by taking Ext. Since $\hat{P}(r,s)$ is a projective object and the functor $\hat{A_{i,j}^k}\hat{P}(r,s)$ a complexes concentrated in one homological degree we can replace $\Ext$ by $\Hom$ and the statement follows.
\end{proof}

The indecomposable projective objects correspond to the twisted canonical basis:
\begin{prop}
\label{twistedcan}
We have the following equality $$[\hat{P}(r,i|s,j)] = v_r \spadesuit v_s\quad\in V_i\otimes V_j.$$
Hence, the formula in Definition \ref{twistedcanbasis} determines the constituents of a standard flag.
\end{prop}

\begin{proof}
When ignoring the grading, that means setting $q=1$, this was shown in \cite[p.223]{BFK}. The graded version follows then by induction from Remark~\ref{divpowers} and Theorem~\ref{cattensor}, since the graded lifts of the standard objects appearing in $\hat{\mathcal{E}}P$ uniquely determine the graded lifts in $\hat{\mathcal{E}}^{(k)}P$ for any $k$ an graded projective module $P$.
\end{proof}

\section{Projective resolutions of standard modules}
In this section we compute projective resolutions of standard modules. To keep formulas simpler, we do this only in the non-graded version. It generalizes in a straight-forward (but tedious) way to the graded (or quantum) version.
\begin{theorem}
\label{projinGroth}
We have the following equality in the Grothendieck group:
\begin{eqnarray*}
[\Delta(r,i|s,j)]&=&
\begin{cases}
\sum_{\gamma \geq 0}
(-1)^{\gamma} \binom{r+\gamma}{\gamma}\left[P(r+\gamma,i|s-\gamma,j)\right]&\text{ if $r+s\leq j$}\\
\sum_{\gamma \geq 0}(-1)^{\gamma} \binom{s-j+\gamma}{\gamma}\left[P(r+\gamma,i|s-\gamma,j)\right]&
\text{ if $ r+s \geq j$.}
\end{cases}
\end{eqnarray*}
The above formula determines the Euler characteristic of
\begin{eqnarray}
\label{Extformula}
\Ext_{A_{(i,j)}-\Mod}\left(\Delta(r,i|s,j), L({r+\gamma,i|s-\gamma,j})\right).
\end{eqnarray}
\end{theorem}
The first statement of this theorem follows directly from Corollary \ref{standards} and Proposition \ref{twistedcan}, but we will give an alternative proof at the end of this section. To verify the second statement it is enough to construct a  projective resolution of the standard modules with the corresponding projectives as indicated in the theorem.
Abbreviate ${}_{r}\iota_{i+j}\Delta(r,i|0,j)\in A_{r,i+j}-\Mod$ by $\Delta(r,i|0,j)$, then the following holds:
\begin{prop}\hfill\\
\label{exactforstandards}
\begin{enumerate}[(i)]
\item Let $0\leq s\leq j$. The projective module $\mathcal{E}^{(s)}\Delta(r,i|0,j)$  has a filtration
$$M_{s+1}\supset M_{s}\supset \cdots\supset M_1\supset M_0=\{0\}$$ such that
$$ M_{u}/M_{u-1}\cong
\begin{cases}
0&\text{if $u<r+s+1-i$},\\
\Delta(r+s+1-u|i, u-1|j)^{\oplus \binom{r+s+1-u}{s+1-u}}&\text{otherwise.}
\end{cases}
$$
\item Let $r\leq i$. The projective module  $\mathcal{F}^{(i-r)}\Delta(i,i|s,j)$  has a filtration
$$M_{i-r+1}\supset M_{i-r}\supset \cdots\supset M_1\supset M_0=\{0\}$$
such that
$$ M_{u}/M_{u-1}\cong
\begin{cases}
0&\text{if $u<i-r-s+1$},\\
\Delta(i-u+1|i, r+s-1-i+u|j)^{\oplus \binom{i+j-r-s-u+1}{j-s}}&\text{otherwise.}
\end{cases}
$$
\end{enumerate}
\end{prop}
\begin{proof}
We prove the first part, the second is completely analogous.
Recall the alternative construction of standard modules from Remark \ref{standardsinduced}. Since tensoring with finite dimensional modules behaves well with respect to parabolic induction (\cite[Lemma 3]{BFK}), we only have to compute the functors applied to the outer product of 'big projectives' and then induce. Moreover, by \cite[Proposition 11]{BFK} there is a filtration of the functor $F:={\mathcal{E}}^{(s)}$  by a sequence of exact functors
$$F=F_{s+1}\supset F_s\supset \cdots \supset F_1\supset F_0=0$$
such that the quotient $F^{u}:=F_u/F_{u-1}$ applied to $\Delta(r,i|s,j)$ can be computed easily. Concretely, $F^{u}(\Delta(r,i|s,j))$ is the parabolically induced outer tensor product of ${\mathcal{E}}^{(s+1-u)}P({\bf a})$, and ${\mathcal{E}}^{(u-1)}P({\bf b})$ where ${\bf a}$ is the sequence of $i-r$ zeroes followed by $i$ ones and, where ${\bf b}$ is the sequence of $j-s$ zeroes followed by $s$ ones.
Since the functors categorifying the divided are exact, hence send projective objects to projective objects, this is just a calculation in the Grothendieck group.  One easily verifies directly that
\begin{eqnarray*}
{\mathcal{E}}^{(s+1-k)} P({\bf a})&=&{\bigoplus}_{l=1}^{\binom{r+s+1-k}{s+1-k}}(P(\underbrace{0,\ldots,0,}_{i-r-s-1+k}\underbrace{1,\ldots, 1}_{r+s+1-k})), \\
{\mathcal{E}}^{(k-1)} P({\bf b})&=&P(\underbrace{0,\cdots,0,}_{j-k+1}\underbrace{1,\ldots, 1}_{k-1}))
\end{eqnarray*}
respectively. Parabolically inducing the external tensor product of these two projective objects gives just a direct sum of  $\binom{r+s+1-k}{s+1-k}$ many copies of the standard module $\Delta(r+s+1-k,i| k-1,j)$ which is one layer of the filtration we wanted to establish.
\end{proof}
Now we are prepared to construct a projective resolution of the standard object $
\Delta(r,i|s,j)$. (Similar constructions and arguments appear in \cite[Proposition 3.15]{W}.) Set $x=m+n-d_1- \cdots -
d_{m-1}$ and define
\begin{eqnarray}
a_{n,m}^{r,s}&=&\mathop{\sum_{(d_1, \ldots, d_m)}}_{d_1+\cdots+d_m=m+n}
\binom{r+m+n}{d_1} \binom{r+m+n-d_1}{d_2} \cdots \binom{r+x}{d_m}.\label{as}\\
 b_{n,m}^{r,s}&=&\mathop{\sum_{(d_1, \ldots, d_m)}}_{d_1+\cdots+d_m=m+n}
\binom{j-s+m+n}{d_1} \binom{j-s+m+n-d_1}{d_2} \cdots \binom{j-s+x}{d_m}\nonumber
\end{eqnarray}
\begin{theorem}
\label{resolution}
The standard module $\Delta(r,i|s,j)$ has a minimal projective resolution
$$0 \rightarrow Q_l \rightarrow \cdots \rightarrow Q_1 \rightarrow Q_0\rightarrow \Delta(r,i|s,j)\rightarrow 0.$$
where $Q_0 = P(r,i|s,j)$ and the other components are as follows:
\begin{enumerate}
\item In case $r+s\leq j$ we have $l=s$ and for $s\leq m<0$,
$$Q_m = \bigoplus_{n=0}^{s-m} P_{r+m+n,s-m-n}^{\oplus a_{n,m}^{r,s}}.$$
\item In case $r+s\geq j$ we have $l=r$ and for $r\geq m>0$
$$ Q_m=\bigoplus_{n=0}^{s-m} P_{r+m+n,s-m-n}^{\oplus b_{n,m}^{r,s}}.$$
\end{enumerate}
\end{theorem}
\begin{proof}
Assume that $ r+s \leq j$.  The other case is similar.
We construct a resolution by induction on $s$.  For the base case $ s =
0$, the object $\Delta(r,i|s,j)$ is projective and there is nothing
to prove.
Assume by induction that the proposition is true for all $\Delta(p,i|q,j)$ where $q < s$.
By Proposition ~\ref{exactforstandards}, and the induction hypothesis, we
have a projective resolution for the standard subquotients of $
M_{s+1}=\mathcal{E}^{(s)}\Delta(r,i|0,j)$, except for
the top quotient $M_{s+1}/M_{s}=\Delta(r,i|s,j)$. Explicitly, the induction hypothesis gives the following resolution of $M_s$:
$$ 0 \rightarrow V_s \rightarrow \cdots \rightarrow V_m \rightarrow \cdots
\rightarrow V_1\rightarrow M_s\rightarrow 0
$$
where $V_m = \bigoplus_{n=0}^{s-m} P_{r+m+n,s-m-n}^{\oplus c_{n,m}^{r,s}}$ with
\begin{eqnarray*}
c_{n,m}^{r,s}&=&\binom{r+1}{1}a_{n,m-1}^{r+1,s-1}+\binom{r+2}{2}a_{n-1,m-1}^{r+2,s-2}
+\cdots+ \binom{r+n+1}{n+1} a_{0,m-1}^{r+n+1,s-n-1}.
\end{eqnarray*}
By Proposition \ref{twistedcan} and Definition \ref{twistedcanbasis} $ \mathcal{E}^{(s)}_r
\Delta(r,i|0,j) \cong P(r,i|s,j)$. Finally, take the cone of the morphism
$M_s\rightarrow P(r,i|s,j)$. Then  Lemma ~\ref{idforas} below gives the asserted projective resolution of $
\Delta(r,i|s,j)$.
\end{proof}
\begin{lemma}
\label{binomid}
The following holds
\begin{enumerate}
\item $ a_{n,1}^{r,s}-a_{n-1,2}^{r,s}+a_{n-2,3}^{r,s} - \cdots + (-1)^n a_{0,n+1}^{r,s} = (-1)^n \binom{r+n+1}{n+1} $ \item $
    b_{n,1}^{r,s}-b_{n-1,2}^{r,s}+b_{n-2,3}^{r,s} - \cdots + (-1)^n b_{0,n+1}^{r,s} = (-1)^n \binom{j-s+n+1}{n+1} $.
\end{enumerate}
\end{lemma}
\begin{proof}
Assume that no entry in a tuple $ {\bf d}=(d_1, \ldots, d_m) $ is zero. Consider $ a_{n+1-m,m}^{r,s} = $
$$ \mathop{\sum_{{{\bf d}=(d_1, \ldots, d_m)}}}_{d_1+\cdots+d_m=n+1} \binom{r+n+1}{d_1} \binom{r+n+1-d_1}{d_2} \cdots
\binom{r+n+1-d_1-\cdots-d_{m-1}}{d_m}. $$
Expanding these binomial coefficients gives
\begin{eqnarray}
a_{n+1-m,m}^{r,s} = \binom{r+n+1}{n+1} \mathop{\sum_{{{\bf d}=(d_1, \ldots, d_m)}}}_{d_1+\cdots+d_m=n+1} \binom{n+1}{d_1, \ldots, d_m}.
\label{aid}
\end{eqnarray}
Then
$$ \sum_{m=1}^{n+1} (-1)^{m+1} a_{n+1-m,m}^{r,s} = \binom{r+n+1}{n+1}  \sum_{m=1}^{n+1} (-1)^{m+1} \mathop{\sum_{{{\bf d}=(d_1, \ldots,
d_m)}}}_{d_1+\cdots+d_m=n+1} \binom{n+1}{d_1, \ldots, d_m}. $$
This sum is $ (-1)^n \binom{r+n+1}{n+1}$ since
$$ \sum_{m=1}^{n+1} (-1)^{m+1} \mathop{\sum_{{{\bf d}=(d_1, \ldots, d_m)}}}_{d_1+\cdots+d_m=n+1} \binom{n+1}{d_1, \ldots, d_m} $$
counts the number of permutations of the set $ \lbrace 1, \ldots, n+1 \rbrace $ whose descent set is $ \lbrace 1, \ldots, n+1 \rbrace $
times $ (-1)^n$.  See exercises 4.7.4 and 4.7.6 of \cite{Bo} for more details.
The proof of the second identity is similar to that of the first.
\end{proof}
The following are recursion formulas for the numbers defined in \eqref{as}:
\begin{lemma}
\label{idforas}
\begin{align*}
a_{n,m}^{r,s} &= \binom{r+1}{1} a_{n,m-1}^{r+1,s-1} + \binom{r+2}{2}
a_{n-1,m-1}^{r+2,s-2} + \cdots + \binom{r+n+1}{n+1}
a_{0,m-1}^{r+n+1,s-n-1}.\\
b_{n,m}^{r,s} &= \binom{j-s+1}{1} b_{n,m-1}^{r+1,s-1} + \binom{j-s+2}{2}
b_{n-1,m-1}^{r+2,s-2} + \cdots + \binom{j-s+n+1}{n+1}
b_{0,m-1}^{r+n+1,s-n-1}.
\end{align*}
\end{lemma}
\begin{proof}
By equation \eqref{aid},
$$ \displaystyle
a_{n,m}^{r,s}= \binom{r+n+m}{n+m} \sum_{(d_1, \ldots, d_m)}
\binom{m+n}{d_1, \ldots, d_m} $$
where the sum is over all tuples $ (d_1, \ldots, d_m) $ such that no $ d_p
$ is zero and $ d_1 + \cdots + d_m = m+n$.
The right hand side of the lemma is
$$ \displaystyle
\sum_{\gamma =0}^n \binom{r+\gamma+1}{\gamma+1}
a_{n-\gamma,m-1}^{r+\gamma+1,s-\gamma-1}
=
\sum_{\gamma =0}^n \binom{r+\gamma+1}{\gamma+1}  \sum_{(d_1, \ldots,
d_{m-1})} \binom{r+m+n}{m+n-\gamma-1}\frac{(m+n-\gamma-1)!}{d_1 ! \cdots
d_{m-1} !}
$$
where the second summation is over tuples $ (d_1, \ldots, d_{m-1}) $ such
that $ d_1 + \cdots + d_{m-1} = m+n-\gamma-1 $ and no $ d_p $ is zero.
This simplifies to:
$$ \displaystyle
\binom{r+m+n}{m+n} \sum_{\gamma=0}^n \sum_{(d_1, \ldots, d_{m-1})}
\frac{(m+n)!}{(\gamma +1)! d_1 ! \cdots d_{m-1} !}. $$
This is easily seen to be equal to $ a_{n,n}^{r,s}$.
The identity for $ b_{n,m}^{r,s} $ is similar.
\end{proof}
\begin{proof}[Proof of Theorem ~\ref{projinGroth}]
By Proposition ~\ref{resolution},
\begin{eqnarray*}
&[\Delta(r,i|j,s)]=[P(r,i|j,s)]&\\
&+ \displaystyle\sum_{n=0}^{s-1} \left((-a_{n,1}^{r,s} +
a_{n-1,2}^{r,s} + \cdots + (-1)^{n+1} a_{0,n+1}^{r,s})[P(r+n+1,i|s-n-1,j)]\right).&
\end{eqnarray*}
Then by Lemma ~\ref{binomid}, this is equal to:
$$ [P(r,i|j,s)]+\sum_{m=0}^{s-1} (-1)^{n+1}
\binom{r+n+1}{n+1}[P(r+n+1,i|s-n-1,j)] $$
which gives the proposition after reindexing.
The second equality follows similarly.
\end{proof}
Theorems ~\ref{projinGroth} and ~\ref{Dformula} give a categorical interpretation of Theorem ~\ref{positivity}.

\section{$3j$-symbols as generalized Kazhdan-Lusztig polynomials}
Recall from Theorems \ref{Grothgraded} and \ref{cattensor} that simple modules correspond to dual canonical basis elements. The graphical calculus from Section \ref{graphical} associates to each dual canonical basis element an oriented cup diagram. Forgetting the orientations we obtain a cup diagram which we can view as a functor via Theorem \ref{catjones}. The following theorem explains the precise connection between the different interpretations:

\begin{theorem}
\label{simplesfromfunctors}
Let ${\bf a}$ be a sequence of $n-k$ zeroes and $k$ ones. Let $\hat{L}({\bf a})$ be the corresponding simple object in $A_{k,n}-\gmod$, concentrated in degree zero. Forgetting the orientations in $\op{Cup}({\bf a})$ we obtain a cup diagram $\op{C}({\bf a})$. Assume that the cup diagram has the maximal possible number $\min\{k,n-k\}$ of cups. Let $F$ be the functor associated via Theorem \ref{catjones}. Then the source category for $F$ is equivalent to the derived category of graded vector spaces. Let $\mC$ be the $1$-dimensional vector space concentrated in degree zero. Then
$$F(\mC)\cong \hat{L}({\bf a}).$$
\end{theorem}

\begin{proof}
The cup functors are just a composition of exact Enright-Shelton equivalences and an exact inclusions, hence map simple modules to simple modules. The highest weight can easily be computed using \eqref{defcupcap} and the source category is by the Enright-Shelton equivalence equivalent to a semisimple category with one simple object.  (Alternatively, the theorem can be deduced from \cite[Lemma 4.9, Theorem 1.2]{BS3} by Koszul duality).
\end{proof}

Theorem \ref{simplesfromfunctors} and Proposition \ref{simples} show that the diagrams from Figure \ref{fig:simples} symbolize uncategorified shifted dual canonical basis elements, and categorified simple objects.

\begin{figure}[htb]
\includegraphics{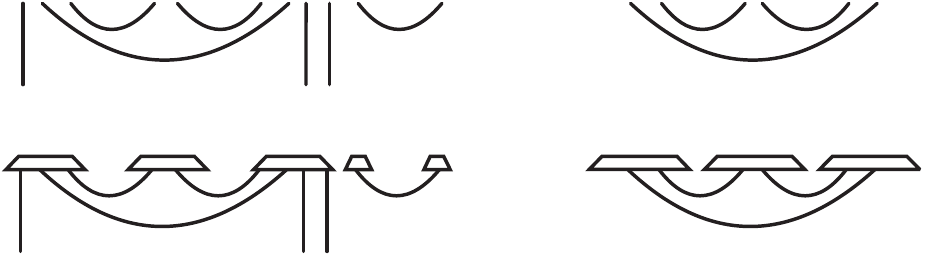}
\caption{Cup diagrams turn into simple objects}
\label{fig:simples}
\end{figure}

Recall the definition of $3j$-symbols from Definition \ref{3j}. By bending the picture \eqref{intertwiner} with the three Jones-Wenzl projectors attached we obtain the last diagram in Figure \ref{fig:simples} where the projectors are $\pi_i$, $\pi_j$, $\pi_k$ respectively. The number of cups is given by formula \eqref{xyz}. The $3j$-symbol can alternatively be defined by Figure \ref{fig:3j-alternative}:

\begin{figure}[htb]
\includegraphics{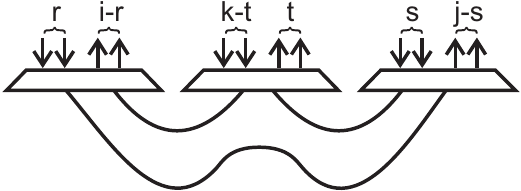}
\caption{The renormalized $3j$-symbol $\hat{C}_{i,j}^k(r,s,t)$.}
\label{fig:3j-alternative}
\end{figure}

\begin{prop}
\label{3jasKL}
The $3j$-symbol $C_{i,j}^k(r,s,t)$ equals the evaluation $\hat{C}_{i,j}^k(r,s,t)$ of the diagrams in Figure \ref{fig:3j-alternative} multiplied with
$c:=(-1)^{-r}(-q)^{s-j}q^{r(i-r)}q^{s(j-s)}.$
\end{prop}

\begin{proof}
We first argue in the classical case $q=1$. Take the diagram on the left hand side of Figure \ref{fig:3j-alternative} without putting the two outer projectors on top, hence we get an intertwiner $f$ with values in $V_1^{\otimes i}\otimes V_k\otimes V_1^{\otimes j}$. Let $f(1)$ be its value written in our standard basis. Let $p$ be the projection of $f(1)$ to the space spanned by all basis vectors
where the number of $v_0$ factors appearing in  $V_1^{\otimes i}$ is $r$ and the number of $v_0$ factors appearing in  $V_1^{\otimes j}$ is $s$. Reading from bottom we first apply the nested cups which gives a vector $h$ in $V_1^{\otimes i}\otimes V_1^{\otimes i}\otimes V_1^{\otimes j}\otimes V_1^{\otimes j}$. It is a linear combination of basis vectors, where only those survive at the end which are not annihilated by $p$. The key point is now that the middle part $V_1^{\otimes i}\otimes V_1^{\otimes j}$ determines already the whole vector and agrees with the image of $\iota_i\otimes\iota_j(v_r\otimes v_s)$. Putting back the two Jones-Wenzl projectors, the claim follows for $q=1$ now directly from the definitions. For generic $q$ let $v=v_{\bf a}\otimes v_{\bf b}\otimes v_{\bf c}\otimes v_{\bf d}$ be a basis vector occurring in $h\in (V_1^{\otimes i})\otimes (V_1^{\otimes i})\otimes (V_1^{\otimes j})\otimes (V_1^{\otimes j})$. By a direct calculation is follows that the coefficient differs from the corresponding vector in  $\iota_i\otimes\iota_j(v_r\otimes v_s)$ by the scalar $$(-q)^{n_0({\bf a})}(-q)^{n_1({\bf d})}q^{l({\bf a})}q^{l({\bf d})}$$
where $n_0$ resp. $n_1$ counts the number of $0$'s respectively $1$'s. Putting the Jones-Wenzl projectors gives an extra factor $q^{r(i-r)}q^{-l({\bf a})}q^{s(j-s)}q^{-l({\bf d})}$ and the statement follows.
\end{proof}

\begin{ex}
{\rm The non-zero evaluations $\hat{C}_{i,j}^k(r,s,t)$ for $i=j=k=2$ are
$\hat{C}_{2,2}^{2}(1,1,1)=q^2-q^{-2}$,
$\hat{C}_{2,2}^2(1,0,0)=q$, $\hat{C}_{2,2}^2(2,0,1)=-q^{3}$,
$\hat{C}_{2,2}^2(0,1,0) = -q$,
$\hat{C}_{2,2}^2(2,1,2) = q$,
$\hat{C}_{2,2}^2(0,2,1) = q^{-1}$,
$\hat{C}_{2,2}^2(1,2,2) = -q$.
This should be compared with Example \ref{3jex}.
}
\end{ex}

 Let $L$ be the simple object associated to the last diagram \ref{fig:simples} where the projectors are $\pi_i$, $\pi_j$, $\pi_k$ respectively. This is an object in $A_{\bf d}-\gMOD$, where ${\bf d}=(i,j,k)$.
With this notation the categorification is then the following:

\begin{theorem}[$3j$-symbols as generalize Kazhdan-Lusztig polynomials]\hfill\\
\label{KL}
The $3j$-symbol $C_{i,j}^k(r,s,t)$ is the graded Euler characteristic of
\begin{eqnarray}
\label{genKL}
\op{Ext}^*_{A_{\bf d}-\gMOD}\left(\hat\Delta(r,i|s,j|t,k),L\right)\langle \gamma\rangle
\end{eqnarray}
with $\gamma=c-r(i-r)-j(s-j)-t(k-t)$ as in Lemma \ref{3jasKL}.

\end{theorem}
\begin{proof}
By Proposition \ref{3jasKL} it is enough to categorify the evaluation from Figure \ref{fig:3j-alternative}. Our desired categorification will then just differ by a shift in the grading by $-c$. We start from the bottom of the diagram and translate it into categorical objects. By the arguments above, the cup diagram with the Jones-Wenzl projectors attached presents a simple object $L$ in  $A_{\bf d}-\gMOD$ corresponding to the shifted dual canonical basis. Let $L'=L\langle \beta\rangle$ be concentrated in degree zero. Hence $\beta=c-\gamma$. Now putting the arrows on top of the diagram means (not categorified) that we apply the evaluation form \eqref{scalar} picking out the coefficient in front of the dual standard basis vector indicated at the top of the diagram.  By \eqref{basisproperstandards} the proper standard modules $\hat\blacktriangle(r,i|s,j|t,k)$ corresponds to the shifted dual standard basis vector. Their isomorphism classes form a basis of the Grothendieck group and so we need a categorical way to pick out a specific dual standard vector in $[L]$ or equivalently a shifted dual standard vector in the expression for $[L']$.
 The main tool here is a homological fact for properly stratified structures: the duals  $\op{d}\hat\Delta$ of standard modules $\hat\Delta$ are homological dual to the proper standard modules $\blacktriangle$. By this we mean $\Ext^i(\blacktriangle,\op{d}\hat\Delta)=0$ for any standard module $\hat\Delta$ and any proper standard $\blacktriangle$ for $i>0$, and $\Hom(\blacktriangle,\op{d}\hat\Delta)=0$ unless $\blacktriangle=\blacktriangle({\bf a})$ and $\Delta=\Delta({\bf a})$ for the same ${\bf a}$, \cite[Theorem 5]{Dlab}. The Exts can be taken here either in the category  $A_{\bf d}-\gMOD$ or in the ungraded category  $A_{\bf d}-\MOD$. Therefore, the evaluation of the diagram from Figure \ref{fig:3j-alternative} equals the Euler characteristic of
 $$\op{Ext}^*_{A_{\bf d}}-\gMOD\left(L', \op{d}\hat{\Delta}\right).$$
 Applying the contravariant duality the latter equals $\op{Ext}^*_{A_{\bf d}}-\gMOD\left(\hat\Delta,L'\right)=\op{Ext}^*_{A_{\bf d}}-\gMOD\left(\hat\Delta,L\right)\langle -\beta\rangle$.
Hence the graded Euler characteristic of $$\op{Ext}^*_{A_{\bf d}}-\gMOD\left(\hat\Delta,L\right)\langle -\beta+c\rangle$$ equals $C_{i,j}^k(r,s,t)$. Since $-\beta+c=\gamma$ the statement of the theorem follows.
\end{proof}

Note that the standard objects in $A_{k,n}-\gmod$ agree with the proper standard objects and are just the graded lifts of the Verma modules. In this case the Euler characteristic of the space  $\Ext^*_{A_{k,n}-\gmod}(\hat{\Delta},L)$ of extensions between a standard object and a simple object is given by certain parabolic {\it Kazhdan-Lusztig polynomials}.  Lascoux and Sch\"utzenberger, \cite{LS}, found explicit formulas for these polynomials, see also \cite[Section 5]{BS2}. Hence our results introduce new generalized  Kazhdan-Lusztig polynomials for the categories of Harish-Chandra modules. The formulas \ref{ClebschGordon} are as far as we know the first explicit formulas for dimensions of extensions between standard and simple Harish-Chandra modules.

\section{Example: Categorified $3j$-symbol}
We finish the study of $3j$-symbols by describing the generalized Kazhdan-Lusztig polynomials in case $i=2$, $j=k=1$.  Using Theorem  \ref{ClebschGordon} one easily calculates
the $3j$-symbols $C_{2,1}^1(2,0,1)=-q^{-1}$, $C_{2,1}^1(1,0,0)=-q^{-1}$, $C_{2,1}^1(1,1,1)=q$, $C_{2,1}^1(0,1,0)=1$
and the evaluations $\hat{C}_{2,1}^1(2,0,1)=q^2$, $\hat{C}_{2,1}^1(1,0,0)=-1$, $\hat{C}_{2,1}^1(1,1,1)=-q$, $\hat{C}_{2,1}^1(0,1,0)=1$.
Now the category $A_{4,2}-\gmod$ has the projective modules $$\hat{P}(1100),\hat{P}(1010), \hat{P}(0110), \hat{P}(1001), \hat{P}(0101), \hat{P}(0011).$$ The projective modules in $A_{(2,1,1),2}-\gmod$ are then $$\hat{P}(1100),\hat{P}(0110), \hat{P}(0101), \hat{P}(0011),$$ in terms of standard modules
\begin{eqnarray*}
\left[\hat{P}(1100)\right]&=&\left[\Delta(1100)\right],\\
\left[\hat{P}(0110)\right]&=&\left[\Delta(0110)\right]+q^2\left[\Delta(1100)\right],\\ \left[\hat{P}(0101)\right]&=&\left[\Delta(0101)\right]+q\left[\Delta(0110)\right]+(q+q^3)\left[\Delta(1100)\right],\\ \left[\hat{P}(0011)\right]&=&\left[\Delta(0011)\right]+q\left[\Delta(0101)\right]+q^2\left[\Delta(0110)\right]+q^4\left[\Delta(1100)\right].
\end{eqnarray*}
There are projective resolutions of the form
\begin{eqnarray*}
\hat{P}(1100)\rightarrow &&\hat{\Delta}(1100)\\
\hat{P}(1100)\langle 2\rangle\rightarrow \hat{P}(0110)\rightarrow &&\hat{\Delta}(0110)\\
\hat{P}(1100)\langle 1\rangle\oplus\hat{P}(0110)\langle 1\rangle \rightarrow \hat{P}(0101)\rightarrow &&\hat{\Delta}(0101)\\
\hat{P}(1100)\langle 2\rangle\oplus\hat{P}(0110)\langle 2\rangle\rightarrow \hat{P}(0101)\langle 1\rangle\rightarrow \hat{P}(0011)\rightarrow &&\Delta(0011)\\
\end{eqnarray*}
Setting $L=\hat{L}((1,1),2|0,1|0,1)$, the values of the Poincare polynomials of $\Ext(\hat\Delta,L)$ for $\hat{\Delta}$ running through the above list of standard objects are $1$, $-q^{2}$, $-q$ and $q^2$ respectively. On the other hand setting $\eta=r(i-r)s(j-s)t(k-t)$
we obtain $\hat{C}_{2,1}^1(0,1,0)q^{\eta}=1$, $\hat{C}_{2,1}^1(1,1,1)q^{\eta}=-q^2$, $\hat{C}_{2,1}^1(1,0,0)q^{\eta}=-q$,   $\hat{C}_{2,1}^1(2,0,1)q^{\eta}=q^2$. This confirms Theorem \ref{KL}.

\section{Categorification of $6j$-symbols or tetrahedron network}
Consider the $\Theta$-network displayed in Figure \ref{fig:theta}. Assume the three projectors from left to the right are $\pi_i$, $\pi_j$, $\pi_k$ with $i=2a$, $j=2b$, $k=2c$ and denote its value by $\Theta(i,j,k)$. The value of the network is explicitly given by the following formula (see \cite[3.1.7]{KaLi}):

\begin{lemma}
\label{theta}
The value of the $\Theta$-network is $$\Theta(i,j,k)=(-1)^{a+b+c}\frac{[a+b-c]![a-b+c]![-c+a+b]![a+b+c+1]!}{[2a]![2b]![2c]!}.$$
In case $i=j=n$, $k=0$ we get $(-1)^n[n+1]$, the value of the $n$-colored unknot.
\end{lemma}

Categorically this network can be interpreted naturally as a complex of graded vector spaces (namely - analogously to Theorem \ref{simplesfromfunctors} - a functor $F$ applied to $\mC$, but now with values in the derived category of graded vector spaces). Even more, it comes along with a natural bimodule structure and after an appropriate shift in the grading, with a natural algebra structure:

\begin{figure}[htb]
\includegraphics{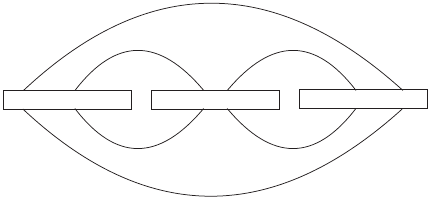}
\caption{Theta networks turn into Ext groups or even Ext algebras of simple objects.}
\label{fig:theta}
\end{figure}

The $\Theta$-network is the Euler characteristic of the Ext-algebra of a simple Harish-Chandra bimodule:

\begin{theorem}
\label{thetaext}
Let $L$ be the simple module symbolized by the last picture in Figure \ref{fig:simples}. Then the $\Theta$-network displayed in Figure \ref{fig:theta} is, up to an overall shift by $\ullcorner j\ulrcorner\langle j\rangle$ for some $j$, isomorphic as (possibly infinite) graded vector space to $\Ext^*_{A_{k,{\bf d}}}(L,L)$. Hence, incorporated this shift, it carries a natural algebra structure. In its natural occurrence without shift it is a bimodule over the above algebra.
\end{theorem}

\begin{corollary}
In particular, the $\Theta$ value from Lemma \ref{theta} is up to an overall power of $\mp q$ the graded dimension of an Ext-algebra (possibly infinite dimensional).
\end{corollary}

\begin{corollary}
In case of the colored unknot with color $n$, the value $\lsem n\rsem$ is the graded dimension of an Ext-algebra (possibly infinite dimensional).
\end{corollary}

\begin{proof} [Proof of Theorem \ref{thetaext}]
The bottom half of the network represents by Lemma \ref{simplesfromfunctors} a simple object $L$ realized as $L\cong F(\mC)$ for some functor $F$. Hence the value of the network is either $GF(\mC)$, where $G$ is (up to shifts) the left adjoint functor of $F$, or alternatively we can use the above theorem and realize it as evaluation of the pairing $\langle[L],[L]\rangle$ which is the graded Euler characteristic of the algebra $\Ext^*(FL,FL)$ up to shifts.
\end{proof}

Figure \ref{spinnetwork} represents a tetrahedron network. The crosses indicate the four faces clued together at the six edges indicated by the boxes. As drawn, the network can be read from bottom to top as composition of intertwiners. Using the functors introduced above categorifying the intertwiners, this network turns into an infinite complex of graded vector spaces categorifying the $6j$-symbol. In general, its value is hard to compute, even its Euler characteristics.

\begin{figure}[htb]
  \centering
\includegraphics{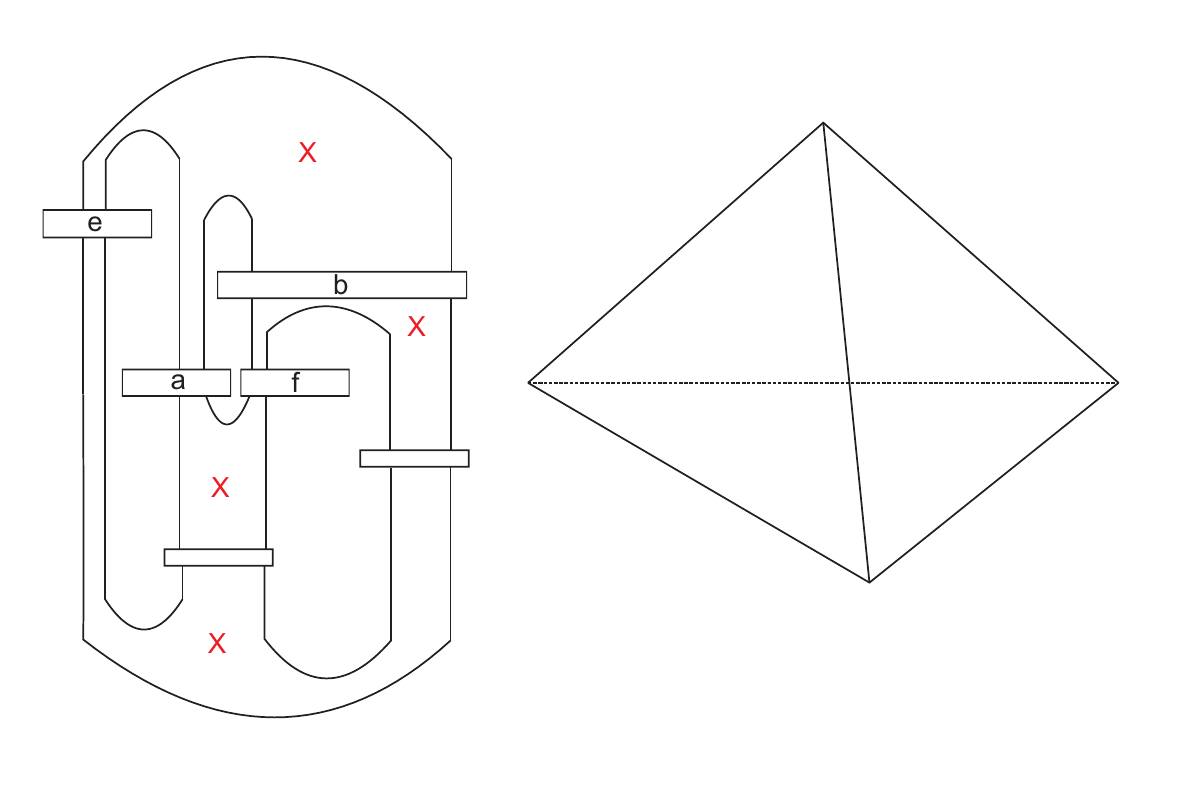}
  \caption{The tetrahedron network (the numbers indicate that there are that many copies of the strand). The Jones-Wenzl projectors correspond to the $6$ vertices of the tetrahedron.}
  \label{spinnetwork}
\end{figure}

\begin{theorem}
Assume we are given a tetrahedron network $T$ as in Figure~\ref{spinnetwork}. Depending on the labels $R_i$ we can find
some positive integer $n$ and a composition ${\bf d}$ of $n$ and graded $A_{n,{\bf d}}$-modules $M$, $N$ such that
the value of $T$ equals the graded Euler characteristics of $\Ext^*_{A_{n,{\bf d}}}(M,N)$.
\end{theorem}

Using the results from \cite{SS} it follows that, up to isomorphism in the derived category of complexes of graded vector spaces,
 the value of the tetrahedron does not depend on how we draw the network. In particular we can always start with the lower half of a tetrahedron network from the top as well as the bottom. In other words it is possible to choose $M=L$ or $N=L$ with $L$ as in Theorem \ref{thetaext} and deduce that the categorification of the tetrahedron is a module over the Ext-algebra categorifying the $\Theta$-network of either face of the tetrahedron.

The $6j$-symbol represented by the tetrahedron network is in general a $q$-rational number and by Theorem \ref{thetaext} can be viewed as a graded Euler characteristic. The parallel result for the $3j$-symbol is Theorem \ref{cat3j}. In the latter case the fractional graded Euler characteristic turned out to be a $q$-integer and had a purely combinatorial interpretation. This fact leads to the natural question of the existence of an integral version of $6j$-symbols and its combinatorial and categorial interpretation. A detailed study of these question will be presented in a subsequent paper.

I.F.\\ Department of Mathematics, Yale University, New Haven, (USA). \\ email: \email{frenkel@math.yale.edu}

C.S.\\ Department of Mathematics, University of Bonn  (Germany).\\ email: \email{stroppel@math.uni-bonn.de}

J.S.\\ Max-Planck-Institute of Mathematics Bonn (Germany). \\ email: \email{joshuasussan@gmail.com}
\end{document}